\documentclass[11pt]{ip-journal}
\usepackage[utf8]{inputenc}
\usepackage{cite}
\usepackage[top=3cm, bottom=3cm, left=3.5cm, right=3.5cm] {geometry}
\usepackage[T1]{fontenc}
\usepackage{adjustbox}
\usepackage{amsmath}
\usepackage{amsthm}
\usepackage{amsfonts}
\usepackage[msc-links]{amsrefs}
\usepackage[english]{babel}
\usepackage{array,epsfig}
\usepackage{amssymb}
\usepackage{amsxtra}
\usepackage{latexsym}
\usepackage{dsfont}
\usepackage{mathrsfs}
\usepackage[mathscr]{eucal}
\usepackage{color}
\usepackage[all]{xy}
\usepackage{hyperref}
\usepackage{graphicx}
\usepackage{mathtools}
\usepackage{caption}
\usepackage{relsize}
\usepackage[monochrome]{xcolor}
\usepackage{url}
\usepackage{dynkin-diagrams}
\usepackage{tikz}
\tikzstyle{color}=[circle,draw=black!50,fill=black!20,thick, inner sep=0pt,minimum size=1mm]

\newtheorem{thm}{Theorem}[section]
\newtheorem{thmx}{Theorem}

\newtheorem{corx}[thmx]{Corollary}
\newtheorem{lem}[thm]{Lemma}
\newtheorem{prop}[thm]{Proposition}
\newtheorem{prob}[thm]{Problem}
\newtheorem{conj}[thm]{Conjecture}
\newtheorem{defn}[thm]{Definition}
\newtheorem{rmk}[thm]{Remark}
\newtheorem{ex}[thm]{Example}
\newtheorem{conv}[thm]{Convention}
\newtheorem{ansatz}[thm]{Ansatz}

\newcommand{\set}[1]{\left\{#1\right\}}
\newcommand{\tuple}[1]{\left(#1\right)}

\newcommand{\abs}[1]{\left|#1\right|}
\newcommand{\norm}[1]{\left\|#1\right\|}
\newcommand{\sprod}[1]{\left<#1\right>}

\newcommand{\ol}[1]{\overline{#1}}
\newcommand{\wh}[1]{\widehat{#1}}
\newcommand{\wt}[1]{\widetilde{#1}}

\newcommand{\afrak}{\mathfrak{a}}

\newcommand{\gfrak}{\mathfrak{g}}
\newcommand{\hfrak}{\mathfrak{h}}
\newcommand{\kfrak}{\mathfrak{k}}

\newcommand{\tfrak}{\mathfrak{t}}

\newcommand{\Cbb}{\mathbb{C}}

\newcommand{\Fbb}{\mathbb{F}}

\newcommand{\Nbb}{\mathbb{N}}

\newcommand{\Pbb}{\mathbb{P}}
\newcommand{\Qbb}{\mathbb{Q}}
\newcommand{\Rbb}{\mathbb{R}}
\newcommand{\Sbb}{\mathbb{S}}

\newcommand{\Zbb}{\mathbb{Z}}

\newcommand{\Acal}{\mathcal{A}}

\newcommand{\Ccal}{\mathcal{C}}
\newcommand{\Dcal}{\mathcal{D}}

\newcommand{\Fcal}{\mathcal{F}}

\newcommand{\Lcal}{\mathcal{L}}
\newcommand{\Mcal}{\mathcal{M}}
\newcommand{\Pcal}{\mathcal{P}}

\newcommand{\Vcal}{\mathcal{V}}
\newcommand{\Xcal}{\mathcal{X}}

\newcommand{\Ncal}{\mathcal{N}}

\renewcommand{\phi}{\varphi}

\newcommand{\del}{\partial}
\newcommand{\delb}{\overline{\partial}}

\newcommand{\Aut}{\operatorname{Aut}}
\newcommand{\Sp}{\operatorname{Sp}}
\newcommand{\SO}{\operatorname{SO}}
\newcommand{\GL}{\operatorname{GL}}

\newcommand{\tr}{\operatorname{tr}}
\newcommand{\Ric}{\operatorname{Ric}}
\newcommand{\SL}{\operatorname{SL}}

\newcommand{\vol}{\operatorname{vol}}

\newcommand{\rank}{\operatorname{rank}}
\newcommand{\Hom}{\operatorname{Hom}}

\newcommand{\Bl}{\operatorname{Bl}}

\title[Calabi-Yau metrics on symmetric spaces]{Calabi-Yau metrics on rank two symmetric spaces with horospherical tangent cone at infinity}
\author{Tran-Trung Nghiem}
\address{Tran-Trung Nghiem, IMAG, Univ Montpellier, CNRS, Montpellier, France}
\email{tran-trung.nghiem@umontpellier.fr}

\begin{document}
\maketitle
\begin{abstract}
We show that on every non-\( G_2 \) complex symmetric space of rank two, there are complete Calabi-Yau metrics of Euclidean volume growth with prescribed horospherical singular tangent cone at infinity, providing the first examples of affine Calabi-Yau smoothings of singular and irregular tangent cone. As a corollary, we obtain infinitely many examples of Calabi-Yau manifolds degenerating to the tangent cone in a single step, supporting a recent conjecture by Sun-Zhang, which was only proved when the tangent cone at infinity has only an isolated singularity. 
\end{abstract}

\section{Introduction}

\subsection{Background}

Let \( (M,g) \) be a complete Ricci-flat Riemannian manifold of dimension \( d \) \textcolor{blue}{such that there is \( \kappa > 0 \) satisfying}
\[ \text{vol}(B_r(p)) \geq \kappa r^{d} \]
for any ball \( B_r(p) \) of radius \( r > 0 \) centered at \( p \). Under this volume condition, we say that \( (M,g) \) has \textit{Euclidean} (or \textit{maximal}) \textit{volume growth}.
By Gromov's compactness theorem, any rescaled sequence \( (M_i, g_i,p_i) = (M, \lambda_i^{-1} g, p) \) for \( \lambda_i \to +\infty \) admits a subsequence converging in the pointed Gromov-Hausdorff sense to a complete length space \( C \), \textcolor{blue}{called \textit{a tangent cone at infinity of \( (M,g) \)}}, which might \textcolor{blue}{a priori} depend on the subsequence, but not on the base point \( p \). An important result due to Cheeger-Colding \cite{CC97, CC00} \textcolor{blue}{asserts} that \textcolor{blue}{under the maximal volume growth, any tangent cone} \( C \) is in fact a \textit{metric cone}, that is, the metric completion of \( ]0,+\infty[ \times Y \), where \( Y \) is some compact metric space called the \textit{link} of \( C \). \textcolor{blue}{If one tangent cone has smooth link, then by Colding-Minicozzi \cite{CM14}, \( C \) is in fact independent of the subsequence.}

If \( (M,g,J) \) is Ricci-flat Kähler and has maximal volume growth, then the tangent cone at infinity has a complex \textit{affine cone} structure \( (C,J_0) \), and the singular set of \( C \) is the same as the algebraic singular set by Donaldson-Sun \cite{DS17}, which is of \textcolor{blue}{real} codimension \( \geq 4 \) \cite{CCT02}.  \textcolor{blue}{As remarked by Székelyhidi \cite{Sze20}, if the metric on \(  M\) is \( \del \delb\)-exact, then \cite{DS17} can be applied to show that the cone is independent of the scaling subsequence.} \textcolor{blue}{From the algebro-geometric aspect, we can also view \( C \) as a \textit{Fano cone singularity}, that is an affine algebraic variety with klt singularities and a unique fixed point under the effective action of some torus \( T_0 \simeq (\Cbb^{*})^k \), such that every orbit closure contains the fixed point} \cite{LWX}. The set of \textit{Reeb vectors} or \textit{polarizations} of \( C \) is the set of elements \( \xi \) in the compact Lie algebra of \( T_0 \) such that \( \xi \) acts on the coordinate ring of \( C \) with non-negative weights (that vanish only on constants). 

On the regular set of \( C \), the \textcolor{blue}{sequence of} Ricci-flat Kähler metrics 
\( (M,\lambda_i^{-1} g,J_i) \) converges in the smooth pointed Cheeger-Gromov sense to a Ricci-flat Kähler metric \( (C,g_0,J_0) \) \cite{CCT02}, and \( g_0 \) actually arises from a (\textcolor{blue}{weak}) \textit{Sasaki-Einstein metric} \( g_Y \) from the link \( Y \), hence is a \textit{conical Calabi-Yau metric}, or \textit{Ricci-flat Kähler cone metric}. To be precise, there is a vector \( \xi \), called the \textit{K-stable Reeb vector} or \textit{K-stable polarization}, such that \( -J_0 \xi \) is a homothetic scaling on \( C \) and
\[ \Lcal_{-J_0 \xi }\omega_0 = 2 \omega_0. \] 
\textcolor{blue}{In particular, the metric form of \( g_0 \) can be written as a Kähler current as \( \omega_0 = \frac{1}{2} i \del \delb r_0^2 \), where \( r_0 \) is the distance to the vertex and also the coordinate on \( ]0,+\infty[\)  such that \( (C,g_0) \) is the metric completion of the warped product \( (]0,+\infty[ \times Y, dr_0^2 + r^2_0 g_Y) \). The Reeb element \( \xi \) is said to be \textit{quasiregular} if it generates an \( \Sbb^1 \)-action on the link (and \textit{regular} if this action is free), and \textit{irregular} otherwise.}

\subsubsection{Asymptotically conical Calabi-Yau manifolds}
If one tangent cone at infinity satisfies an integrability condition \cite[Definition 0.11]{CT94}, and has smooth link (we simply say that the cone is \textit{smooth} in this case), then \( (M,g) \) is in fact an \textit{asymptotically conical (AC)} Ricci-flat Kähler manifold in the terminology of \cite{CH13}\footnote{The integrability condition is in fact redundant by \cite[Theorem 1.2]{SZ23}}. This means that there are compact subsets \( K, K', \lambda > 0 \) and a diffeomorphism
\[ \Phi : C \backslash K' \to M \backslash K \] 
such that 
\[ \abs{ \nabla^k_{g_0}( \Phi^{*} g - g_0)} = O(r_0^{-\lambda -k }), \; \forall k > 0. \]
If the volume form \( \Omega \) of an AC Ricci-flat Kähler manifold \( (M,g,J) \) also converges to the canonical volume form \( \Omega_0 \) of \( C \) via \( \Phi \), i.e. there is some \( \mu > 0 \) such that
\[ \abs{ \Phi^{*} \Omega - \Omega_0}_{g_0} = O(r_0^{-\mu}), \] 
then \( (M,g,J) \) is said to be an \textit{AC Calabi-Yau manifold}. This condition already implies \( C^0\)-convergence of the complex structure \( J \) to \( J_0 \) via \( \Phi \) \cite{CH13}. Moreover, if \( (M,g,J) \) is AC Ricci-flat Kähler but not simply connected then passing to the universal cover yields an AC Calabi-Yau manifold \cite{CH24}. 

The existence, uniqueness and classification of AC Calabi-Yau manifolds have now been relatively well understood thanks to the works of Conlon-Hein \cite{CH13}, \cite{CH15}, \cite{CH24}. Initially inspired by van Coevering's papers \cite{vC08, vC10, vC11}, they went on to optimize an earlier theorem in the seminal works of Tian-Yau \cite{TY90, TY91}, culminating in the classification of asymptotically conical manifolds. We can summarize the optimal Tian-Yau theorem due to Conlon-Hein as follows.
\begin{thm}[\!\!\cite{TY91, CH15}] \label{theorem_optimal_tianyau}
Let \( X^n \) be a compact Kähler orbifold without \( \Cbb\)-codimension-\( 1 \) singularities. Let \( D \supset \text{Sing}(X) \) be a suborbifold divisor in \( X \) such that \( -p K_X = qD, p,q \in \Nbb,  q/p > 1 \), and such that \( D \) admits a Kähler-Einstein metric of positive scalar curvature.

For every Kähler class \( \kfrak \) on \( X \backslash D \) and for every \( t > 0 \), there is a unique Calabi-Yau metric \( \omega_t \in \kfrak \) which is asymptotically conical to the metric \( t \omega_0 \), where \( \omega_0 \) is built on the Calabi ansatz Ricci-flat Kähler cone metric of \( p K_D \).  
\end{thm}

The \textit{Calabi ansatz} for manifolds can be outlined as follows: If \( D^{n-1} \) is a Fano Kähler-Einstein manifold, then the radius function of the conical Calabi-Yau metric on \( K_D^{\times} \) is given by \( r^2 = \norm{ . }^{\frac{2}{n}} \), where \( \norm{ . } \) is a Hermitian norm on \(  K_D \) naturally obtained from the Kähler-Einstein metric on \( D \). A similar construction works for \( (\frac{1}{k} K_D)^{\times} \) as well as for orbifolds \cite{CH15}.

As remarked in \cite[Section 3]{CH15}, the following \textit{generalized} or \textit{irregular} Tian-Yau problem was first proposed by van Coevering \cite{vC08} but with an incorrect proof.

\begin{prob} \label{problem_generalized_tianyau} Produce AC Calabi-Yau metrics on \( X \backslash D \) for pairs \( (X,D) \) as in Theorem \ref{theorem_optimal_tianyau}, where \( D \) is not Kähler-Einstein, yet the \( \Sbb^1\)-bundle over \( D \) admits an irregular Sasaki-Einstein structure inducing the CR structure (in particular \( (-N_D)^{\times} \) is conical Calabi-Yau), and \( (X,D) \) is not a blow-up of \( (-N_D)^{\times} \).
\end{prob}

In \cite[Theorem C] {CH15}, the authors solved Problem \ref{problem_generalized_tianyau} for the smooth pair \( (X,D) \simeq ( \text{Bl}_p \Pbb^3, \text{Bl}_{p_1,p_2} \Pbb^2 )\), producing the only known example of Calabi-Yau metrics of Euclidean volume growth with irregular asymptotic cone in the literature. To the author's knowledge, there has not been any counterexample to this problem.


\subsubsection{Calabi-Yau manifolds with Euclidean volume growth and singular tangent cones}

The AC condition turns out to be very restrictive, as there have been existence results of Calabi-Yau metrics of Euclidean volume growth over \( \Cbb^n, n \geq 3  \) with singular tangent cones at infinity \cite{Li19}, \cite{Sze19}, \cite{CR21}.

The general strategy consists of viewing \( \Cbb^n \) as an affine smoothing of the potential tangent cone at infinity, then building a smooth asymptotic solution of the Ricci-flat equation from a suitably chosen conical Calabi-Yau metric. Since this \textcolor{blue}{model is only smooth on the regular set of the cone, one needs to} glue it with a model metric near singularities to get a globally \textcolor{blue}{smooth asymptotic solution} over \( \Cbb^n \) with good geometry. \textcolor{blue}{After improving the Ricci potential decay rate,} known existence theorems of Tian-Yau-Hein \cite{Hein} \cite{TY90} \cite{TY91} then allow one to solve the Monge-Ampère equation \textcolor{blue}{with the given asymptotic geometry} as in the compact case to get a genuine Calabi-Yau metric. 

\subsubsection{Known results on symmetric spaces}
The geometric objects of interest to us will be \textit{complex symmetric spaces}. They are Stein complexifications of compact Riemannian symmetric spaces, and can be viewed as their (co)tangent bundles with a suitable complex structure. Our general setting consists of a \textcolor{blue}{non-identical involution \( \theta \) on a complex} linear semisimple connected group \( G \) with maximal compact subgroup \( K \), and a rank-\(k\) complex symmetric space \( G/H \) of dimension \( n\), where \( G^{\theta} \subset H \subset N_G(G^{\theta}) \) and the inclusion is of finite index. When \( H\) is not semisimple, \( G/H \) is said to be \textit{Hermitian}; if moreover \( G/H \) is indecomposable, we actually have \( \dim Z(H) = 1 \).

If \( k = 1 \), then there are only finitely many families of complex symmetric spaces up to isomorphisms. They are exactly the cotangent bundles of the sphere, the complex and quaternionic projective spaces, and the Cayley projective plane, cf. Table \ref{table_rank_one_ss}.

 The \( K \)-invariant \( \del \delb\)-exact complete Calabi-Yau metrics on complex symmetric spaces of rank 1 were constructed by Stenzel after translating the complex Monge-Ampère equation to a one-variable ODE, which can be explicitly solved \cite{Sten}. 

\begin{thm}[\!\!\cite{Sten}]
Every complex symmetric space \( G/H \) of rank \( 1 \) admits a complete \(  K\)-invariant \( \del \delb\)-exact Calabi-Yau metric. 
\end{thm}

In fact, we have a more geometric description of the \textcolor{blue}{\( K \)-invariant} \( \del \delb \)-exact Calabi-Yau metrics with maximal volume growth on rank one symmetric spaces, which should be well-known by experts, but we still include the proof as a prelude and motivation for the rank-two part. 

\begin{table} 
\adjustbox{width=1\textwidth}{
\centering
\begin{tabular}{|c|c|c|c|c|c| }
\hline
Type  & Representative & Root system & Multiplicities & Satake diagram & Hermitian \\
\hline
\( BDII \) &\( \SO_{r+2} / \SO_1 \times \SO_{r+1} \) & \( A_1 \) & \( r, r \geq 1 \) & 
\( \begin{aligned} 
&\dynkin D{o*.***} \text{if \( r \) even}, \\
&\dynkin B{o*.**} \text{else}
\end{aligned} \) & \( r = 1 \) \\
\( AIII \) & \( \SL_{r+1} /  \SL_1 \times \SL_{r+1} \) &  \( BC_1 \) & \( (2r - 3,1), r \geq 2  \) & \( \dynkin[edge length = .75cm, involutions={14}] {A}{o*.*o} \) &\( \text{yes} \) \\   
\( CII \) &\( \Sp_{2r} / \Sp_2 \times \Sp_{2r - 2} \) & \( BC_1 \) & \( (4r - 8,3), r \geq 3 \) & \( \dynkin C{*o**.**}  \) &\( \text{no} \) \\
\( FII \) &\( F_4 / B_4\) & \( BC_1 \) & \( (8,7) \) & \( \dynkin F{o***} \) & \( \text{no} \) \\
\hline
\end{tabular}}
\caption{Symmetric spaces of rank one. This table is deduced from \cite[Table 26.3]{Tim11}, \cite[Table VI]{Hel78}.}
\label{table_rank_one_ss}
\end{table}

We offer an approach using the theory of equivariant degeneration for spherical varieties \cite{BP87}. While existence is essentially contained in \cite{CH15} and uniqueness well-known \cite[Proposition 21]{Myk11} (which is in fact stronger than the uniqueness result in \cite[Theorem A]{CH15}), our strategy has the advantage of yielding immediately the potential tangent cone and some of its geometric properties, which might not be clear a priori, e.g. the spherical action, smoothness of the link and the K-stable Reeb vector. The reader will find more details in Subsection \ref{subsection_rank_one_case}.  

\begin{thm} \label{theorem_rank_one_cyss}
Every rank one symmetric space \( G/H \) admits a complete \( K\)-invariant \( \del \delb\)-exact Calabi-Yau metric which is asymptotically conical to a regular and smooth horospherical cone, obtained as a \( G\)-equivariant degeneration of the symmetric space. Moreover, the metric is unique in the trivial Kähler class up to one scaling parameter and the tangent cone is unique given the metric. 
\end{thm}

In other words, any rank one symmetric space is an affine smoothing of its horospherical tangent cone. This is consistent with Conlon-Hein's classification result in \cite[Theorem A]{CH24} which shows that AC Calabi-Yau manifolds can only arise as suitable deformations of the asymptotic cone (e.g. affine smoothing), possibly followed by a crepant resolution. 
 
By the work of Bielawski, it is now known that there is a \textit{weak} Calabi-Yau structure on any complex symmetric space \cite[Theorem 0.1]{Bie04}.
The regularity theorem \cite[Theorem 0.2]{Bie04} however only applies when the densities \( f,g \) in the real Monge-Ampère equation
\( f(\nabla \phi) \det(d^2 \phi) = g \) 
are \textit{positive}, which is not the case for our equation on symmetric spaces (see \eqref{equation_ricci_flat_symmetric_space} where the rhs is only nonnegative). Even if one may very well prove the regularity of the metric assuming only nonnegativity of the densities (using a method analogous to that of Koike for example \cite{Koi23}), this existence theorem is not very geometrically enlightening, and a more meaningful reformulation of the problem can be stated as follows.

\begin{prob} \label{problem_cy_metrics_prescribed_tangent_cone}
\textcolor{blue}{Given a complex symmetric space \( G/H \) and a prescribed \( G\)-spherical (possibly singular) Calabi-Yau cone \( C \), obtained as a \( G\)-equivariant degeneration of \( G/H \), find complete \(  K\)-invariant Calabi-Yau metrics of Euclidean volume growth with \( C \) as the tangent cone at infinity.}
\end{prob}

\begin{rmk} 
The choice of \( C \) is a priori reasonable, since the tangent cone at infinity of a complete \( \del \delb\)-exact Calabi-Yau metric with Euclidean volume growth on a complex manifold \( M \) is in fact the central fiber of a flat \( \Cbb^{*}\)-equivariant degeneration of \( M \) in (at most) two steps \cite{Sze20} \cite{DS17}.
\end{rmk}
 
On complex Hermitian symmetric spaces, Biquard-Gauduchon obtained explicit hyperKähler metrics with singular tangent cones \cite{BG96}. \textcolor{blue}{Motivated by} these examples, Biquard-Delcroix constructed Calabi-Yau metrics with singular tangent cones by viewing the space as open subset inside the \textit{wonderful compactification}, obtained by adding two irreducible divisors, then gluing the \textcolor{blue}{Calabi} ansatz (\textcolor{blue}{referred to as the Tian-Yau ansatz therein})  on one K-stable Fano boundary divisor (if any) with a specific ansatz on the other boundary divisor \cite{BD19}. \textcolor{blue}{Let us make a remark before recalling their main result.} 

In \cite{BD19}, Theorem 1.1 is only stated for symmetric spaces with \( H = N_G(G^{\theta}) \) and we recall it as such. This is to guarantee the existence of a smooth compactification, for example a wonderful one. However, the theorem can be stated for any space \( G/H \) of the same involution type as \( G/N_G(G^{\theta}) \), since then \( G/H \to G/N_G(G^{\theta}) \) is a finite covering. 

\begin{thm}[\!\!\cite{BD19}]
Let \( G/H \) be a  rank two symmetric space with \( H = N_G(G^{\theta}) \). Let \( X \) be the wonderful compactification of \( G/H \) so that \( X \backslash G/H \) is a smooth simple normal crossing divisor with irreducible components \( D_1, D_2 \). Let \( A_1, A_2 > 0 \) be as defined in Subsection \ref{subsection_setup}, and 
\[ C(D_1^{\vee} ), C(D_2^{\vee} ) = ( \frac{1}{A_1} K_{D_1^{\vee} })^{\times}, (\frac{1}{A_2} K_{D_2^{\vee}})^{\times} \]
be the Fano cone singularities over the equivariant blowdowns \( D_1^{\vee}, D_2^{\vee} \) along the closed orbit \( D_1 \cap D_2 \) of \( D_1, D_2 \). Then both cones have a regular conical Calabi-Yau structure (equivalently, both \( D_1^{\vee}, D_2^{\vee} \) are Fano Kähler-Einstein), except for the spaces with restricted root system \( G_2 \), where only one cone is Calabi-Yau.

Moreover, there is a complete \( K\)-invariant \( \del \delb\)-exact Calabi-Yau metric of Euclidean volume growth and
\begin{itemize} 
    \item with singular tangent cone at infinity \( C(D_1^{\vee} ) \) or \( C(D_2^{\vee}) \) on the non-Hermitian symmetric spaces
    \[ \SO_5 \times \SO_5 / \SO_5, \quad \Sp_8 / \Sp_4 \times \Sp_4, \quad G_2 / \SO_4, \quad G_2 \times G_2 / G_2, \]
    provided that \( C(D_i^{\vee} ), i = 1,2 \) is Calabi-Yau;
    \item with singular tangent cone at infinity \( C(D_1^{\vee}) \) or \( C(D_2^{\vee}) \)  on the Hermitian ones. More precisely, there is one choice of the tangent cone that recovers the Biquard-Gauduchon metrics, while the other choice yields a different metric on
    \[ \SO_r / \SO_2 \times \SO_{r-2}, r \geq 5, \quad \SL_5/ \SL_2 \times \SL_3. \]
\end{itemize}
\end{thm} 
Although this partially answers Problem \ref{problem_cy_metrics_prescribed_tangent_cone} for rank two symmetric spaces with prescribed horosymmetric tangent cones, one main drawback of their ansatz is that the asymptotic solutions obtained have unbounded holomorphic bisectional curvature in some cases, \textcolor{blue}{notably the simplest class of (restricted) root system \( A_2 \), as well as the infinite family \( \Sp_{2r} / \Sp_2 \times \Sp_{2r-2}, r \geq 5\) of non-Hermitian symmetric spaces with root system \( BC_2 \)}. The unbounded curvature prevents the application of any known a priori estimate methods to perturb the asymptotic solution to a genuine solution. Unboundedness of the curvature in bad cases is not explicitly proved in \cite{BD19}, but can be done similarly as in Proposition \ref{proposition_unbounded_curvature}.
 
 \subsection{Main result}
 The issues in \cite{BD19} can be remedied by using the conical Calabi-Yau metric from a degeneration of \( G/H \) induced by a valuation \textit{in the interior} of the negative Weyl chamber. The theory of equivariant degeneration of spherical varieties \cite[Corollaire 3.8]{BP87} tells us that any such direction defines an equivariant degeneration of \( G/H \) to a \textit{horospherical cone} \( C_0 \), which always admits a K-stable polarization \cite{Ngh22b} and can be taken to be the candidate of our tangent cone at infinity.
 
With this conical Calabi-Yau structure in hand, we generalize further the tentative construction in \cite{BD19} to obtain new Calabi-Yau metrics with singular tangent cones at infinity. Our main result covers all the cases of symmetric spaces left by Biquard-Delcroix, namely the spaces with restricted root sytem \( A_2 \), the non-Hermitian infinite family \( \Sp_{2r} / \Sp_2 \times \Sp_{2r-2} \) and the finite family of multiplicity \( (8,6,1) \), but also all the rank-two decomposable cases of root system \( R_1 \times R_1 \) not considered in \cite{BD19}, where each factor \( R_1 \) is \textit{any} symmetric space of rank \( 1 \). 

\begin{minipage}{\textwidth}
\begin{thmx} \label{theorem_main_theorem_cyss}
Every rank two symmetric space \( G/H \), except of restricted root system \( G_2 \), admits a complete \( K\)-invariant \( \del \delb\)-exact Calabi-Yau metric of Euclidean volume growth with singular \( G\)-horospherical tangent cone \( C_0 \) at infinity, which is
\begin{itemize}
    \item regular in the \( A_2 \) cases.
    \item generally irregular in the \( BC_2 \) and \( B_2 \) cases. 
    \item the product of the Stenzel asymptotic cones on each factor in the \( R_1 \times R_1 \) cases.
\end{itemize}
\end{thmx}
Let us make some comments related to this theorem. 
\end{minipage}

\begin{itemize}
\item On the classes of spaces considered in \cite{BD19}, our metrics are new and not isometric to the Biquard-Delcroix metrics since their choice of asymptotic behavior leads to a tangent cone at infinity with a \textit{non-horospherical} \( G\)-action in every case (see also Example \ref{example_horosymmetric_nonhorospherical}). Moreover, our choice of tangent cone at infinity yields Calabi-Yau metrics with bounded holomorphic sectional curvature in all cases (see Table \ref{table_constants}, as well as Lemmata \ref{lemma_bounded_geometry_finitefamilies_bc2}, \ref{lemma_bounded_geometry_infinitefamilies_bc2}). 
\item Our result provides a full answer to Problem \ref{problem_cy_metrics_prescribed_tangent_cone} for rank two symmetric spaces with prescribed horospherical Calabi-Yau cones. Geometrically, non-\( G_2 \) symmetric spaces up to rank two can all be interpreted as equivariant affine Calabi-Yau smoothings of horospherical singular tangent cones. \textcolor{blue}{As a result}, we obtain the first examples of Calabi-Yau metrics on manifolds that are affine smoothings of \textit{singular and irregular} tangent cones, \textcolor{blue}{but also a solution to the singular version of the irregular Tian-Yau problem \ref{problem_generalized_tianyau} on \( BC_2/ B_2 \) symmetric spaces.} Based on numerical evidence, it is very likely that \textit{all the tangent cones in the \( BC_2 / B_2 \) cases are irregular}. 


\item Fixing the choice of \( H \) so that an \( R_1 \times R_1 \)-symmetric space is a product, there is an obvious Calabi-Yau metric which is the product of the Stenzel metrics on the factors, whose tangent cone at infinity and K-stable Reeb vector are product. It turns out that the tangent cone of the Calabi-Yau metric obtained by our ansatz has \textit{the same K-stable Reeb vector} as the product of the Stenzel tangent cones (cf. Subsection \ref{subsection_rank_one_case}, Theorem \ref{theorem_kstable_horospherical_cone} and Table \ref{table_constants}), which implies that our metrics have the \textit{same tangent cone at infinity as the product metric}. This comes out of our expectation, and we anticipate that the metrics on \( R_1 \times R_1 \) can be related to the product metric by scalings on each factor. 

Example \ref{example_compactification_r1r1} provides the description of the horospherical tangent cone in the \( A_1 \times A_1 \) case that coincides with the product tangent cone (up to a finite covering).

\item The symmetric spaces of restricted root system \( G_2 \) are however excluded. This is practically due to the K-stable Reeb vector lying outside the Weyl chamber (see Theorem \ref{theorem_kstable_horospherical_cone}), hence the global potential is not well-defined if we try to expand the potential of the Calabi ansatz near the boundary divisors using Stenzel potentials. 

While this explanation seems a practical one related to the construction, we expect that the existence of a Calabi-Yau smoothing is actually prevented by algebro-geometric conditions, and that in fact \textit{there is no \( K \)-invariant Calabi-Yau metrics with horospherical tangent cone} on \( G_2 \) symmetric spaces, yielding a negative example to Problem \ref{problem_generalized_tianyau} when \( D \) is singular. This is addressed more thoroughly by means of algebraic geometry in the paper \cite{Ngh24b}.

\item Since we are only considering \textit{semisimple} symmetric spaces \( G/H \) which do not have any \( \Cbb^{*} \)-action commuting with the \( G \)-action, \( \Cbb^3 \) does not fall into this category. However as remarked in \cite{Ngh24b}, \( \Cbb^3 \) is a smooth affine \textit{spherical} variety under the complexified action of Yang Li's isometry \( \SO_{\Rbb}(3) \times U(1) \) (cf. \cite[Section 3]{Li19}), hence one can expect to apply the strategy of the author to rebuild Yang Li's metric. 
\end{itemize}

Since there is no nontrivial equivariant degeneration of a horospherical variety, the following corollary is somewhat trivial, but worth mentionning. 

\begin{corx} 
Every Calabi-Yau symmetric space constructed in Theorem \ref{theorem_main_theorem_cyss} degenerates to the horospherical tangent cone at infinity in only one step.
\end{corx}

This supports a recent conjecture made by Sun-Zhang that there is no semistability in the two-steps degeneration theory of Donaldson-Sun for the tangent cone at infinity, i.e. any Calabi-Yau metric of maximal volume growth should degenerate to the tangent cone at infinity in \textit{only one step} \cite[Conjecture 6.4]{SZ23}. 

Let us give a more detailed account. As explained in \cite{SZ23}, any complete Calabi-Yau manifold \( (M, g, J) \) with Euclidean volume growth is asymptotic to a unique Calabi-Yau cone \( C \) with possibly non-isolated singularities \cite{DS17} \cite{Liu21}. Moreover there is a \( \Cbb^{*}\)-equivariant flat affine family \( \Xcal \to  \Cbb \), induced by a negative valuation arising from the metric \( \omega \) \cite[(3.4)]{DS17}, with general fiber \( M \) and central fiber a K-semistable Fano cone singularity \( W \). By analogy with local theory, one would expect another possibly non-trivial flat affine degeneration from \( W \) to \( C \). However, the ``degeneration in one step'' conjecture predicts that \textit{\( W \) is already isomorphic to \( C \)} as normal affine algebraic cones (i.e. \( W \) is already K-stable), so the second step of degeneration from \( W \) to \( C \) is redundant.

This was proved in \cite{SZ23} for tangent cones at infinity with smooth link, and the result is quite unexpected since the \( 2 \)-step degeneration theory works very similarly for local tangent cone, but \textit{the semistable step can not always be eliminated in the local setting}.  In the singular case, as mentioned above, there are already examples of Calabi-Yau metrics on \( \Cbb^n \), \( n \geq 3 \) degenerating to the regular and singular cone \( \Cbb \times A_1 \) in one step \cite{Li19} \cite{Sze20} \cite{CR21}. Our new metrics provide many more examples of this kind, \textcolor{blue}{but also new examples of degeneration to the irregular tangent cone in only one step}. However, we do not know if this phenomenon is true in general for noncompact Calabi-Yau spherical manifolds.

\begin{table}
\adjustbox{width=1\textwidth}{
\centering
\begin{tabular}{|c|c|c|c|c|c|c| }
\hline
Type  & Representative &  \( R \) & Multiplicities & Satake diagram & Hermitian   \\
\hline
\( AI \) &\( \SL_3 / \SO_3 \) & \( A_2 \) & \( 1 \) & \( \dynkin[edgelength = .75cm ] A{o2} \) & \(  \)  \\
\( A_2 \) & \( \text{PGL}_3 \times \text{PGL}_3 / \text{PGL}_3 \) & \( -  \) & \(  2  \) &\( 

\begin{dynkinDiagram}[name=upper, edgelength=.75cm] A{o2}
\node (current) at ($(upper root 1)+(0,-.75cm)$) {};
\dynkin[at=(current),name=lower, edgelength=.75cm] A{o2}
\begin{pgfonlayer}{Dynkin behind}
\foreach \i in {1,...,2}
{
\draw[latex-latex]
($(upper root \i)$)
-- ($(lower root \i)$);
}
\end{pgfonlayer}
\end{dynkinDiagram}
\) & \( \) \\   
\( AII \) & \( \SL_6 / \Sp_6 \) &  \( - \) & \( 4 \) & \( \dynkin[edgelength=.75cm] A{*o*o*} \) & \( \)  \\
\( EIV \) & \( E_6   / F_4 \) &  \( - \) & \( 8 \) & \( \dynkin[edgelength=.75cm] E{IV}\) & \( \) \\
\hline

\( AIIIa \) & \( \SL_r / \text{S}(\GL_2 \times \GL_{r-2}) \) & \( BC_2 \) & \( (2r-8,2,1), r \geq 5 \)& \( \dynkin[edge length = .75cm, involutions={16;25}] {A}{oo*.*oo} \) & \( \text{yes} \) \\
\( CIIa \) &\( \Sp_{2r} / \Sp_{4} \times \Sp_{2r - 2} \) & \( - \) & \( ( 4r -16, 4, 3),  r \geq 3 \) & \( \dynkin C{*o*o.**} \) & \( \) \\
\( DIIIa \) & \( \SO_{10} / \GL_5 \) & \( - \) & \( (4,4,1) \) & \(
\dynkin[edge length = .75cm, involutions={54}] D{*o*oo} \) & \( \text{yes} \)\\
\( EIII \) & \( E_6 / \SO_{10} \times \SO_2 \) & \( - \) & \( (8,6,1) \)& \( \dynkin[edge length = .75cm, involutions={16}] E{oo***o} \) & \( \)\\
\hline

\( BDI \) & \( \SO_r / \SO_2 \times \SO_{r-2} \) & \( B_2 \) & \( (r-4,1,0), r \geq 5 \) &\( 
\begin{aligned} 
&\dynkin D{oo*.***} \text{if \( r\) even}, \\
&\dynkin B{oo*.**} \text{else}
\end{aligned} \) &\( \text{yes} \) \\
\( AIIIb \) & \( \SL_{4} / \text{S}(\GL_2 \times \GL_{2}) \) & \( - \) & \( (2,1,0) \) & \(\dynkin[edge length =.75cm, involutions={13}] A{ooo} \) & \( \text{yes} \) \\ 
\( DIIIb \) & \( \SO_{8} / \GL_4 \) & \( - \) & \( (4,1,0) \) & \( \dynkin[edgelength = .75cm] D{**oo} \) & \( \text{yes} \)\\ 
\( B_2 \) &\( \SO_5 \times \SO_5 / \SO_5 \) & \( - \) & \( (2,2,0) \) &\( 
\begin{dynkinDiagram}[name=upper, edgelength=.75cm] B{o2}
\node (current) at ($(upper root 1)+(0,-.75cm)$) {};
\dynkin[at=(current),name=lower, edgelength=.75cm] B{o2}
\begin{pgfonlayer}{Dynkin behind}
\foreach \i in {1,...,2}
{
\draw[latex-latex]
($(upper root \i)$)
-- ($(lower root \i)$);
}
\end{pgfonlayer}
\end{dynkinDiagram}
\) & \( \) \\
\( CIIb \) & \( \Sp_8 / \Sp_4 \times \Sp_4 \) & \( - \) & \( (4,3,0) \) &\( \dynkin[edgelength = .75cm] C{*o*o} \) & \( \) \\
\hline
\( G \) & \( G_2 / \SO_4 \) & \( G_2 \) & \( 1 \) & \( \dynkin G{oo} \) & \( \)  \\
\( G_2 \) & \( G_2 \times G_2 / G_2 \) & \( - \) & \( 2 \) & \( \begin{dynkinDiagram}[name=upper, edgelength=.75cm] G{o2}
\node (current) at ($(upper root 1)+(0,-.75cm)$) {};
\dynkin[at=(current),name=lower, edgelength=.75cm] G{o2}
\begin{pgfonlayer}{Dynkin behind}
\foreach \i in {1,...,2}
{
\draw[latex-latex]
($(upper root \i)$)
-- ($(lower root \i)$);
}
\end{pgfonlayer}
\end{dynkinDiagram} \) & \( \) \\
\hline
\end{tabular}}
\caption{Indecomposable symmetric spaces of rank two; see also Convention \ref{convention} and Remark \ref{rmk_ss_list}. This table is deduced from \cite[Table 26.3]{Tim11} and \cite[Table VI, Ch. X]{Hel78}.}
\label{table_rank_two_ss}
\end{table}

\subsection{Organization}
\begin{itemize} 
\item In Section \ref{section_setup}, we provide preliminaries on symmetric spaces and recall basic facts about spherical cones, then prove Theorem \ref{theorem_rank_one_cyss}. 
\item In Section \ref{section_geometric_prelude}, we provide details on the relevant geometric constructions in rank two and prove that all the candidates for the tangent cone are singular. We view \( G/H \) as open Zariski subset of a (in general non-smooth and non-Fano) one-divisor compactification \textcolor{blue}{\( X  = G/H \cup D_0^{\vee} \)}. In order to have reasonable models near the singularities of the tangent cone, we consider the blow-up of \( X \) along the two closed orbits, denoted by \( \wt{X} \). \textcolor{blue}{In particular, we have \( \wt{X} \backslash (G/H) = D_0 \cup D_1 \cup D_2 \), with \( D_1, D_2 \) being the two divisors corresponding to the two closed orbits} \textcolor{blue}{(see Section \ref{section_geometric_prelude} for more details on the geometry of these divisors).  The tangent cone at infinity \( C_0 \) can be interpreted as the equivariant degeneration of \( G/H \) along a direction in the interior of the (negative) Weyl chamber}.

The geometry at infinity of \( G/H \) contains: (i) the \textit{generic region} modeled on the \textcolor{blue}{open orbit \( G/H_0 \times \Cbb^{*} \) } of the \textcolor{blue}{cone, \( G/H_0 \) being the open orbit of \( D_0 \)}; and (ii) the \textit{boundary regions}, corresponding to \( D_0 \cap D_1 \backslash D_2 \) and \( D_0 \cap D_2 \backslash D_1 \), \textcolor{blue}{which are also contained in the possible singularities of \( D_0 \)}.

The readers should compare this approach with the recent work of Collins-Li \cite{CL24} where the compactification is Fano, and the \textcolor{blue}{Calabi} ansatz is modeled upon the boundary divisors similarly as in \cite{BD19}. Here we are able to choose a different ansatz, which is not a priori obvious by an analysis of K-stable valuations inside the Weyl chamber of our model, already carried out in the preparatory work \cite{Ngh22b}. 

\item In Section \ref{section_horospherical_calabi_ansatz}, we produce the \textcolor{blue}{asymptotic solution} from the conical Calabi-Yau metric on \( C_0 \) \textcolor{blue}{using a Calabi ansatz (which applies even in the irregular horospherical case, cf. our previous work \cite{Ngh22b})}, then analyze its behavior near \( D_1, D_2 \). 

\item In Section \ref{section_global_asymptotic_solution}, we build models near \(D_1, D_2 \) which are \textcolor{blue}{\textit{warped products of AC metrics} (Ansatz \ref{ansatz_boundary})}, with first order terms matching with the asymptotics of the \textcolor{blue}{Calabi} ansatz. The reader should compare this ansatz with the warped AC product in \cite{BD19} \cite{CR21}. 

The next step is to improve the decay rate of the Ricci potentials (Proposition \ref{proposition_boundary_potential}), and glue these models with the Calabi ansatz. We verify that there is no obstruction to the gluing process in Lemma \ref{lemma_pasting_coherence}. The rank two assumption plays a crucial role in both part \ref{section_horospherical_calabi_ansatz} and \ref{section_global_asymptotic_solution}. 

\item In Section \ref{section_asymptotic_geometry}, we study the asymptotic geometry of the global asymptotic solution obtained in the previous section, and prove that the tangent cone at infinity of the metric is as expected. 

\item In Section \ref{section_proof_main_theorem}, we prove our main Theorem \ref{theorem_main_theorem_cyss} \textcolor{blue}{by appealing to Hein's existence package}. To guarantee a quasi-atlas, as well as \( C^2 \)-estimate for the complex Monge-Ampère equation, we impose a condition on the injectivity radius of the initial metric, which is equivalent to the holomorphic sectional curvature being bounded (cf. Propositions \ref{proposition_unbounded_curvature},\ref{proposition_bounded_curvature_combinatorial}). This condition depends only on the combinatorial data of the tangent cone, and should be related to some fine properties of the cone's singularities. For lack of any meaningful interpretation at this moment, we proceed to check the condition case by case, and hope to provide more details on this matter in a future study. We provide in Table \ref{table_constants} the relevant constants. 
\item In Appendix \ref{appendix_spherical_varieties}, we gather the relevant theory on spherical varieties needed all along the paper. 
\end{itemize} 

\subsection{Further discussions}

\subsubsection{Uniqueness} 
We have not yet mentionned uniqueness of the Calabi-Yau metric with a fixed tangent cone at infinity. We expect that the proof can be done in the same lines as in Székelyhidi's paper \cite{Sze20} \textcolor{blue}{with input from Donaldson-Sun's theory adapted to spherical manifolds, and a finer understanding of the geometry of the smoothing}. We hope to come back to this question in a future paper.

Related to this uniqueness question, we point out in Example \ref{example_compactification_smoothdiv_finiteaut} a Kähler-Einstein log pair \( (X,D) \) where \( X \) is a smooth compactification  of \( G/H \), with finite automorphism group \( \Aut_G (X,D) \), and \( D \) a \textit{singular} divisor with strictly larger automorphism group \( \Aut_G(D) \).

By a recent result of Biquard-Guenancia \cite{BG22}, on a \textit{smooth} pair \( (X,D) \) such that \( D \in \abs{-K_X} \), the conic Kähler-Einstein metrics \( \omega_{\beta} \) with small angle \( 0 < \beta \ll 1 \) on \( D \) converges up to a rescaling factor to the Tian-Yau metric \( \omega_{TY} \) on \( X \backslash D \). This confirms the folklore expectation that the Tian-Yau metrics on \( X \backslash D \) are not fortuitous as they appear, but arise in a canonical way from conic Kähler-Einstein metrics on the given smooth compactification \( (X,D) \) with smooth divisor.

The author learned from Biquard-Delcroix-Guenancia that a similar result is expected to hold in the case where \( \alpha D \in \abs{-K_X} \) with \( \alpha > 1 \) and \( D \) being Kähler-Einstein\footnote{This is confirmed by Biquard-Guenancia \cite{BG24}.} (this is proved for smooth horosymmetric pairs of rank one in \cite{Del24}), or even in the singular context. However, note that the Calabi-Yau metrics in our example should arise in a two-parameters family (with parameters being the scaling factors), while the metrics \( \omega_{\beta} \) occur only in a finite orbit by Matsushima's uniqueness theorem. It would be interesting to obtain examples of this kind but with \textit{smooth} \( D \).

\subsubsection{Regularity of the horospherical Calabi-Yau cone structure} 
We do not know how to \textcolor{blue}{systematically determine the regularity of the Calabi-Yau structure on the horospherical tangent cones at infinity}. \textcolor{blue}{Explicit computations in rank two suggest that} all the tangent cones \textcolor{blue}{are} irregular, except in the \( A_2 \) or \( R_1 \times R_1 \) cases. \textcolor{blue}{This leads us to conjecturing} the following. 

\begin{conj}
The horospherical tangent cones at infinity of the symmetric spaces are irregular, except for the symmetric spaces with restricted root system \textcolor{blue}{of all the factors being} of type \( A \). 
\end{conj}

\textcolor{blue}{To be precise, from our previous work \cite{Ngh22b}, the K-stability condition on a rank two horospherical cone \( (C,\xi) \) is in fact equivalent to the parameter of the Reeb vector \( \xi \) being the unique positive solution to a one-variable polynomial. It is then reasonable to speculate that the regularity conjecture can be reformulated in terms of an algebro-arithmetic question. For instance, the conjecture for the tangent cones of the infinite family \( \SL_{r} / \text{S} (\GL_2 \times \GL_{r-2}), r \geq 5 \) can be translated as}
\begin{conj}
\textcolor{blue}{For all \( r \geq 5 \), the unique positive root of 
\[ Q(t,r) = \int^{\frac{2r-6}{1+t}}_{\frac{-4}{2+t}} p(4 +(2+t)p)^2 (4r-8 -tp)^2 (2r-6 -(1+t)p)^{2r-7} (2r-2+p)^{2r-7} dp = 0 \] 
(cf. Lemma \ref{lemma_bounded_geometry_infinitefamilies_bc2}) is irrational.}
\end{conj}
A possible approach is to show by induction that one can always factor out an irreducible polynomial, but the difficulty lies in proving irreducibility of the polynomial over \( \Qbb\). 

\textbf{Acknowledgements.} This paper is part of a thesis prepared under the supervision of Thibaut Delcroix and Marc Herzlich, partially supported by ANR-21-CE40-0011 JCJC project MARGE. I am grateful to Thibaut Delcroix for many helpful discussions, Pierre-Louis Montagard for providing a reference, to Ronan~J. Conlon, Hans-Joachim Hein and G\'{a}bor Székelyhidi for valuable comments. The final acknowledgement goes to the anonymous referees for their reading and comments which largely improved  the paper.

\section{Complex symmetric spaces and spherical cones} \label{section_setup}

\subsection{Complex symmetric spaces}\label{subsection_setup} 
\subsubsection{Root data and classification}
This part provides an overview on the structure of complex symmetric spaces. The reader may consult \cite{Vus74}, \cite{Vus90}, \cite{DCP83} for a detailed treatment. Let \( G \) be a complex linear semisimple group with maximal compact subgroup $K$. Given a nontrivial involution \( \theta \in \Aut(G) \), let \( G^{\theta} \subset H \subset N_G(G^{\theta}) \) be a closed subgroup of \( G \). The homogeneous space $G/H$ is called \textit{a (complex) symmetric space}. It is well-known that a \(G\)-symmetric space is \( G\)-spherical (see \cite{Vus74} for a proof and Appendix \ref{appendix_spherical_varieties}) and a smooth affine variety \cite[26.1]{Tim11}. 

Let \( T_s \subset G \) be a maximal torus satisfying \( \theta(t) = t^{-1}, \; \forall\; t \in T_s \). The \textit{rank} of the symmetric space is defined as the complex dimension of \( T_s \).  By \cite[2.3]{Vus90}, we have an identification  \( \Mcal(G/H) = \Mcal(T_s / T_s \cap H) \) (see Appendix \ref{appendix_spherical_varieties} for the definition of \( \Mcal(G/H) \)). Moreover, \( T_s / T_s \cap H \) is a subgroup of finite index in \( T_s \) (cf. \cite[2.2]{Vus90}), hence 
\[ \Mcal(G/H)_{\Rbb} = \Mcal(T_s)_{\Rbb}. \] 
Let \( \tfrak_{s} \) be the Lie algebra of \( T_s \), which decomposes as
\[ \tfrak_s = \afrak \oplus \tfrak_c,\]
where \( \afrak, \tfrak_c \) are the non-compact and compact parts of \( \tfrak_s \), corresponding to the Lie algebras of the non-compact and compact parts of \( T_s\). 
The  non-compact part \( \afrak \) of \( \tfrak_s \) is isomorphic to \( \Ncal(T_s)_{\Rbb} \simeq \text{Hom}(\Cbb^{*},T_s) \otimes \Rbb \) and called \textit{Cartan algebra} of \( G/H \).  In particular 
\[ \rank(G/H) = \dim_{\Rbb} \afrak. \]
We also have that \( \afrak^{*} \simeq \Mcal(T_s)_{\Rbb} \simeq \tfrak_c \). 
\begin{defn}[\!\!\cite{DCP83} \cite{Vus76}]
\hfill
\begin{itemize}
\item There exists a \( \theta \)-stable maximal torus \( T \) containing \( T_s \). 

\item Let \( \wh{R} \) be the root system \( (G,T) \) and \( \wh{R}^{\theta} \) be the set of roots fixed by \( \theta \)  with the set of positive roots \( \wh{R}^{+} \) chosen such that if \( \wh{\alpha} \in \wh{R}^{+} \backslash \wh{R}^{\theta} \) then \( -\theta(\alpha) \in \wh{R}^{+} \). The set defined as
\[ R := \set{ \alpha = \wh{\alpha} - \theta(\wh{\alpha}), \; \wh{\alpha} \in \wh{R} \backslash \wh{R}^{\theta}}, \]
is a (possibly non-reduced) root system which has the same rank as \( G/H \), called the \emph{restricted root system of \( G/H \)}. In particular, \( \wh{\alpha}|_{\tfrak_s} = 2 \alpha|_{\tfrak_s} \).  
\item The \textit{multiplicity} of a restricted root \( \alpha \), denoted by \( m_{\alpha} \), is the number of \( \wh{\alpha} \in \wh{R} \) such that \( (\wh{\alpha} - \theta(\wh{\alpha})) |_{\afrak}  = \alpha \). If \( G/H \) is a complex symmetric space of dimension \( n \) and of rank \( k \), then
\[ n = k + \sum_{\alpha \in R^{+}} m_{\alpha}. \]  
\item The Weyl group \( W \) associated to \( R \) is called the \emph{restricted Weyl group}.
\item The set of positive restricted roots \( R^{+} \) is defined as
\[ R^{+} := \set{\alpha = \wh{\alpha} - \theta(\wh{\alpha}), \; \wh{\alpha} \in \wh{R}^{+} \backslash \wh{R}^{\theta}}. \] 
The choice of \( R^{+} \) on a semisimple symmetric space determines a strictly convex cone, which is the fundamental domain of \( W \curvearrowright \afrak^{*} \), called the \emph{restricted positive Weyl chamber} (or simply \emph{the Weyl chamber} when the context is clear). 
\item After fixing a set of positive restricted roots and isomorphism using the Killing form on \( \gfrak \), the \( G\)-invariant valuations set \( \Vcal \subset \afrak \) (see Appendix \ref{appendix_spherical_varieties}) of a symmetric space can be considered as the \textit{negative restricted Weyl chamber}, which is strictly convex (if \( G/H \) is semisimple) and polyhedral of maximal dimension.
\end{itemize}
\end{defn}


In terms of Lie algebra, recall that we have the root decomposition \( \gfrak = \tfrak \oplus \bigoplus_{\wh{\alpha} \in \wh{R}} \gfrak_{\wh{\alpha}} \). For each \( \wh{\alpha} \in \wh{R} \) one can build a \( \mathfrak{sl}_2 \)-triple \( (e_{\wh{\alpha}}, e_{-\wh{\alpha}}, h_{\wh{\alpha}}) \). Morever, one has \( \tfrak \simeq \tfrak^{\theta} \oplus \tfrak_s \), and 
 \begin{align*}
 \hfrak = \tfrak^{\theta} \oplus \bigoplus_{\wh{\alpha} \in \wh{R}^{\theta}} \gfrak_{\wh{\alpha}} \oplus \bigoplus_{\wh{\alpha} \in \wh{R} \backslash \wh{R}^{\theta}} \Cbb (e_{\wh{\alpha}} + \theta(e_{\wh{\alpha}})),
 \end{align*}
 cf. \cite[Proposition 1.3]{DCP83}. 
 It follows that
 \begin{align*}
 \gfrak &= \hfrak \oplus \tfrak_s \oplus \bigoplus_{\wh{\alpha} \in \wh{R}^{+} \backslash \wh{R}^{\theta}} \Cbb (e_{\wh{\alpha}} + \theta(e_{\wh{\alpha}})) \\
 &= \hfrak \oplus \tfrak_c(\simeq \afrak^{*}) \oplus \afrak \oplus \bigoplus_{\wh{\alpha} \in \wh{R}^{+} \backslash \wh{R}^{\theta}} \Cbb (e_{\wh{\alpha}} + \theta(e_{\wh{\alpha}})).
 \end{align*}

\begin{thm}[\!\!{\cite[Théorème 3.1]{Las78}}]
We have the following decomposition 
\[ G/H = K \exp(-\Vcal).x_0, \]
where \( -\Vcal \) is the positive Weyl chamber. Moreover, every \( x \in G/H \) corresponds uniquely to an element \( A(x) \) in the given Weyl chamber. Since the latter is a \( W\)-fundamental domain, \( K \)-invariant objects (e.g. open sets, functions) on \( G/H \) correspond (via the exponential map on \( \afrak \) and \(  K\)-action) to \( W \)-invariant objects on \( \afrak \), i.e. objects defined in \( -\Vcal \). 
\end{thm} 
 
 From another point of view, complex symmetric spaces are tangent bundles of compact Riemannian symmetric spaces, so the data of a complex symmetric space is determined by a compact Riemannian one. The classification of compact Riemannian symmetric spaces was carried out by \'{E}lie Cartan in terms of compact Lie algebra and their involution type (see \cite{Hel78} for a detailed account). The symmetric spaces up to rank two, as well as their Satake diagrams, are listed in Table \ref{table_rank_one_ss} and \ref{table_rank_two_ss}. 

The support of a Satake diagram is the same as the Dynkin diagram of \( G\), but with some nodes marked in black corresponding to simple roots fixed by \( \theta \). The extra arrows join pairs of simple roots \( (\wh{\alpha},\wh{\beta}) \) such that 
\[ (\wh{\alpha} - \theta(\wh{\alpha})) |_{\afrak} = (\wh{\beta} - \theta(\wh{\beta}))|_{\afrak}. \]
The Satake diagrams will be used to determine which tangent cones are singular. 
 
\begin{conv} \label{convention}
Throughout the text, we use the notation \( \sprod{. , .} \) for the Killing form on \( \gfrak\), which induces a Killing form on \( \afrak^{*} \) by \( \sprod{\alpha, \beta} := \sprod{H_{\alpha}, H_{\beta}} \), where \( H_{\alpha} \in \afrak \) is the dual of \( \alpha \) by \( \sprod{.,.}\). For a symmetric space of multiplicities \( (m_1,m_2,m_3) \), we mean a symmetric space of restricted root system \( BC_2 \) with \( m_3 \geq 1 \) and \( B_2 \) with \( m_3 = 0 \), where \( (m_1,m_3) \) is the multiplicities of \( (\alpha_1, 2 \alpha_1) \), \( \alpha_1 \) being the short root, and \( m_2 \) is the multiplicity of the long root \( \alpha_2 \). 

We gather spaces of the same restricted root system (with a priori different involution type) into a common category. Given an involution type and a restricted root system, there is a unique complex symmetric space up to a finite covering.  
The representative of each involution type is taken to be \( G/G^{\theta} \), so that \( \Mcal \) is the lattice generated by the restricted fundamental weights \emph{\cite[Lemme 3.1]{Vus90}}. 
\end{conv}

\begin{rmk} \label{rmk_ss_list}
Compared with the list in \emph{\cite[Table 1]{BD19}}, we single out in Table \ref{table_rank_two_ss} two classes of symmetric spaces, namely the \( AIIIb\) type (which is tacitly assumed to be in \( AIIIa \) with \( m_3 = 0 \) in \emph{\cite{BD19}}) and \(DIIIb \) type. Even if they are redundant in terms of restricted root system and multiplicities, they truly need to be distinguished from the family \( \SO_r / \SO_2 \times \SO_{r-2} \) (in terms of involution) because their tangent cones have distinct spherical actions.
\end{rmk} 

\subsubsection{Relevant numerical data}
The Duistermaat-Heckman polynomial of \( G/H \) is defined as
\[ P(p) := \prod_{\alpha \in R^{+}} \sprod{\alpha,p}^{m_{\alpha}}. \] 

On a semisimple symmetric space of rank \( k = 2 \), we will denote by \( \wt{\alpha}_2, \wt{\alpha}_1 \) the \textcolor{blue}{\textit{generators}} or \textit{wall directions} of the \textcolor{blue}{given positive} Weyl chamber, which write
\begin{equation} 
\wt{\alpha}_2 = \alpha_2 + \zeta_2 \alpha_1, \quad \wt{\alpha}_1 = \alpha_1 + \zeta_1 \alpha_2,
\end{equation} 
where
\begin{equation}
\zeta_2 = -\frac{\sprod{\alpha_2, \alpha_1}}{\sprod{\alpha_1, \alpha_1}}, \quad \zeta_1 = - \frac{\sprod{\alpha_2,\alpha_1}}{\sprod{\alpha_2, \alpha_2}}. 
\end{equation}

\begin{figure}
\begin{tikzpicture}
\pgfmathsetmacro\ax{2}
\pgfmathsetmacro\ay{0}
\pgfmathsetmacro\bx{2 * cos(120)}
\pgfmathsetmacro\by{2 * sin(120)}
\pgfmathsetmacro\lax{2*\ax/3 + \bx/3}
\pgfmathsetmacro\lay{2*\ay/3 + \by/3}
\pgfmathsetmacro\lbx{\ax/3 + 2*\bx/3}
\pgfmathsetmacro\lby{\ay/3 + 2*\by/3}

\tikzstyle{couleur_pl}=[circle,draw=black!50,fill=blue!20,thick, inner sep = 0pt, minimum size = 2mm]


\draw[->, thick] (0,0) -- (\ax,\ay) node[below right] {\( {\alpha}_1 \)};
\draw[->, thick] (0,0) -- (\bx, \by) node[above right] {\( {\alpha}_2 \)};
\draw[->] (0,0) -- (\ax + \bx/2, \ay + \by/2) node[above right] {\( \wt{\alpha}_1 \)};

\draw (0,0)--(2,0);
\draw[->] (0,0)--(\ax/2 + \bx, \ay/2 + \by) node[above right]{\( \wt{\alpha}_2 \)};
\end{tikzpicture}
\caption{Restricted root system of \( A_2\) symmetric spaces and their Weyl generators.}
\end{figure}

An important data is the \textit{sum of the positive restricted roots} 
\[ \varpi := \sum_{\alpha \in \wh{R}^{+} \backslash \wh{R}^{\theta} }  2 \wh{\alpha} = \sum_{\alpha \in R^{+}} m_{\alpha} \alpha.  \] 
In the rank two cases, we have in particular 
\[ \varpi = A_1 \alpha_1 + A_2 \alpha_2, \]
where 
\[ A_1 = \frac{ \sprod{\varpi, \wt{\alpha}_1} }{\sprod{\alpha_1, \wt{\alpha}_1}}, \quad A_2 = \frac{\sprod{\varpi, \wt{\alpha}_2}}{ \sprod{\alpha_2, \wt{\alpha}_2}}. \]

\subsection{Compactifications of symmetric spaces}
As an immediate consequence of Theorem \ref{spherical_embeddings_classification}, we have the following. 

\begin{thm}
Given a complex symmetric space \( G/H \), there is a one-to-one correspondence between 
\begin{itemize}
\item the isomorphism classes of \( G \)-equivariant compactification of \( G/H \), 
\item the set of colored subdivisions of \( \Vcal \), i.e. colored fans with support the valuation cone \( \Vcal \). 
\end{itemize}
\end{thm}

\begin{ex}
Every semisimple symmetric space admits a natural and canonical compactification corresponding to the colored cone \( (\Vcal, \varnothing) \). This compactification is not smooth in general. 
\end{ex}

\begin{thm}[\!\!{\cite[3.1, Theorem]{DCP83}}] \label{theorem_wonderful_compactification}
Let \( G/H \) be a semisimple symmetric space of rank \( k \) such that \( H = N_G(G^{\theta}) \) and \( \Vcal \) the valuation cone. The canonical compactification \( (\Vcal, \varnothing) \) is then a \emph{wonderful compactification} in the sense that 
\begin{itemize}
\item \( X \) is a \( G\)-equivariant smooth compactification of \( G/H \) such that \( X \ \backslash G/H = \cup_{i=1}^k D_i \) is a simple normal crossing divisor.
\item The \( G\)-orbit closures are the \( 2^k \) intersections \( \cap_{i \in J} D_i \) where \( J \) runs through all subsets of \( \set{1,\dots,k} \). 
\item The unique closed orbit of \( G\) is \( \cap_{i=1}^k D_i \), which is a flag manifold.
\end{itemize}
Moreover, \( D_i \) an equivariant fibration over a flag manifold \( G/P_i \) with fiber a symmetric variety of rank \((k-1)\), which turns out to also be a wonderful compactification. 
\end{thm}

\begin{ex} \label{example_wonderful_compactification_rank_one}
Let \( G = \SO_n \) and \( H = \SO_{n-1} \), \( n \geq 3 \). Then \( G/H \) is a symmetric space of type BDII with restricted root system \( A_1 \) (see Table \ref{table_rank_one_ss}). Geometrically \( G/H \) is \( G\)-isomorphic to the affine quadric \( \set{z \in \Cbb^{n}, q(z,z) = 0} \), where \( q\) is the standard symmetric bilinear form on \( \Cbb^{n} \). Following \emph{\cite{Ah83}}, the wonderful compactification of \( G/H \) is given by \( X = \set{(z,t), q(z,z) = t^2} \subset \Pbb(\Cbb^n \oplus \Cbb) \). 
\end{ex}

\begin{rmk} \label{remark_symmetric_compact_not_smooth}
Smooth \( G\)-equivariant compactifications with Picard rank one of semi-simple symmetric spaces were classified and described by Ruzzi \emph{\cite[Theorem 5]{Ruz10}}. An explicit list of Fano locally factorial equivariant compactification in rank two is also given in \emph{\cite[Theorem 2.6]{Ruz12}}. In general, if a symmetric space \( G/H \) admits a smooth of Picard rank one or Fano compactification, then there are explicit conditions on the restricted root system \( R \) or isotropy subgroup \( H \) (such as \( H = N_G(G^{\theta}) \)). 
\end{rmk}

\subsection{Equivariant degeneration and spherical cones} \label{subsection_equivariant_degeneration_and_spherical_cones}


We say that an affine \( G\)-spherical variety \( C \) is a \textit{spherical cone} or \textit{conical embedding} if \( C \) has a unique fixed point (which is also the unique closed orbit) under the \( G\)-action. 

\begin{thm}[\!\!\cite{Ngh22b}] \label{theorem_conical_embedding_colored_cone}
The colored cone \( (\Ccal_C, \Dcal_C) \) of a conical embedding \( G/H \subset C \) consists of a strictly convex cone of maximal dimension \( \Ccal_C \) with \( \Dcal_C \) being the set of all the colors of \( G/H \). Every \( \Qbb\)-Gorenstein spherical cone is in fact a Fano cone singularity.
\end{thm}
 
\begin{defn}
A \( G \times \Cbb^{*} \)-equivariant degeneration of a \( G \)-spherical affine variety \( X \) is a \( G \times \Cbb^{*}\)-equivariant flat affine family of varieties \( \pi: \Xcal \to \Cbb \) such that \( \pi^{-1}(t) \) is \( G\)-isomorphic to \(  X \) for each \( t \neq 0 \). Here we view \( \Cbb\) as a \( \Cbb^{*}\)-toric variety with a trivial \( G\)-action. We abbreviate such a degeneration as an equivariant degeneration when the context is clear. 
\end{defn}

Given a complex semisimple symmetric space \( G/H \), a \( G\)-invariant valuation \( \nu \in \Vcal \) defines naturally an equivariant degeneration of \( G/H \). The degeneration family can be constructed as follows. The \( G \times \Cbb^{*}\)-spherical space \( G/H \times \Cbb^{*} \) has a weight lattice naturally isomorphic to \( \Mcal(G/H) \oplus \Zbb \) and valuation cone \( \Vcal(G/H) \oplus \Rbb \) (the isomorphisms depend on the choice of the \( \Cbb^{*}\)-right-action). Let \( \Dcal \) be the colors of \( G/H \) and \( \Dcal^{*} \) denote the colors of \( G/H \times \Cbb^{*}\); we have 
\[ \Dcal^{*} = \set{D \times \Cbb^{*}, D \in \Dcal}. \]
 The colored cone
\[ (\text{Cone}((\nu,1), \Dcal^{*}), \Dcal^{*}) \]
has support a strongly convex cone and takes all the colors, hence defines a conical embedding \( \Xcal \) of \( G/H \times \Cbb^{*} \) (Theorem \ref{theorem_conical_embedding_colored_cone}). By the description of spherical morphisms (Proposition \ref{proposition_spherical_morphisms}), the projection of \( \Ncal(G/H) \oplus \Zbb \) to the second factor induces a domination of \( (\text{Cone}((\nu,1), \Dcal^{*}), \Dcal^{*}) \) over \( \text{Cone}(0,1)\), hence a \( G \times \Cbb^{*}\)-equivariant flat affine morphism
\[\pi: \Xcal \to \Cbb.  \]
The central fiber of \( \pi^{-1}(0) \) corresponds to the ray \( \Rbb_{\geq 0}(\nu,1) \) by the orbit-cone correspondence (Theorem \ref{proposition_orbit_cone_correspondence}).
If \( \nu \) is in the interior of \( \Vcal(G/H) \), then \( \pi^{-1}(0) \) is a \( G \times \Cbb^{*}\)-\textit{horospherical cone} (of the same dimension as \( G/H \) by flatness) \cite[Corollaire 3.8]{BP87}.

A less combinatorial description can be done as follows. 
Let \( X \) be an equivariant projective compactification of a symmetric space such that \( X \backslash (G/H) = D \) is irreducible. Then there is a naturally \( G \times \Cbb^{*} \)-linearized ample line bundle \( L \to X \), and one can consider the \( G\times \Cbb^{*}\)-spherical cone \( C_L(X) \) obtained by equivariantly contracting the zero section of \( L^{-1} \). Clearly \( C_L(X) \) is a \( G\times \Cbb^{*}\)-conical embedding of \( G/H \times \Cbb^{*} \). The natural projection induces an equivariant degeneration \( C_L(X) \to \Cbb \) with central fiber an affine cone over \( D_0 \).   

\begin{ex}
Consider the symmetric space \( \SO_n / \SO_{n-1} \) in Example \ref{example_wonderful_compactification_rank_one}. Take the natural cone \( C(X) = \set{(z,t) \in \Cbb^{n+1} \times \Cbb, q(z,z) = t^2} \) over the wonderful compactification \( X \). The second factor projection \( C(X) \to \Cbb, (z,t) \to t \) is a \(G \times \Cbb^{*}\)-equivariant degeneration of \( \SO_{n} / \SO_{n-1} \) with central fiber \( C_0 = \set{z \in \Cbb^{n+1}, q(z) = 0} \).
\end{ex}

There are two kinds of spherical cones that appear as central fiber in the \( G\)-equivariant degeneration of a semisimple symmetric space \( G/H \) of rank two. First, the \textit{horosymmetric cone}, which corresponds to the degeneration of \(G/H \) along a valuation defined by a generator of \( \Vcal \). Second, the \textit{horospherical cone}, which corresponds to the equivariant degenerations of \( G/H \) along \textit{any} valuation in the interior of \( \Vcal \). Both cones are conical embeddings of open orbits having the following fibration structure. 

\begin{defn}[\!\!{\cite{Del20b}}]
A horosymmetric (resp. horospherical) space \( G/H \) of rank \( k \) is an equivariant fibration with fiber a symmetric space of rank \( k \) (resp. a complex torus \( (\Cbb^{*})^k \)) over a flag manifold \( G/P \). A horosymmetric (resp. horospherical) variety is an equivariant embedding of a horosymmetric (resp. horospherical) space. 
\end{defn}

The class of horosymmetric spaces strictly contains horospherical spaces. It is well-known that they are all spherical spaces. 

\begin{ex}
While toric cones are all horospherical, there are many examples of non-toric horospherical cones. Consider the odd symplectic grassmannian \( \text{Gr}_{\omega}(k,2n+1) \), \( n \geq k \geq 2 \), defined as the set of \( k\)-dimensional isotropic subspaces in a \((2n+1)\)-dimensional complex vector space equipped with a skew-symmetric bilinear form \( \omega \). This is a non-toric (since their connected automorphism group does not contain a torus of maximal dimension \emph{\cite{Pas09}}) Fano horospherical manifold of rank \( 1 \), so the cone with respect to the anticanonical polarization is naturally a \( \Qbb\)-Gorenstein horospherical cone of rank \( 2 \). 
\end{ex}

\begin{ex} \label{example_horosymmetric_nonhorospherical}
Let \( G/H \) be a semisimple symmetric space of rank \( k \) such that \( H = N_G(G^{\theta}) \) and \( X \) its wonderful compactification (see Theorem \ref{theorem_wonderful_compactification}). Let \( (\alpha_i)_{i=1}^k \) be the positive restricted roots of \( G/H \). Following \emph{\cite[last paragraph of \S 1.8]{Ruz12}}, the canonical divisor of \( X \) can be represented as 
\[ K_X = \sum_{i=1}^k (A_i+1)D_i, \; A_i \in \Nbb^{*}_{+},  \]
where the \( A_is \) are determined by \( \sum_{i=1}^k A_i \alpha_i = \varpi. \)
Taking the cone (with respect to the anticanonical polarization) over a Fano equivariant blowdown \( D_i^{\vee} \) of \( D_i \) yields a \( G\)-horosymmetric cone. The action of \( G\) is non-horospherical because the valuation cone of \( D_i \) (determined using \emph{\cite[Proposition 3.3]{GH15}}) is only a half-plane. When \( k = 2 \), \( D_i^{\vee} \) is the blowdown of \( D_i \) along the closed orbit \( D_1 \cap D_2 \) (which is a flag manifold), considered in \emph{\cite{BD19}}. 
\end{ex}

\subsection{The Ricci-flat equation on symmetric spaces}

Given the \( K\)-invariant volume form \( dV_H \) associated to the standard Hermitian metric on the canonical bundle of \( G/H \) \cite{Del20}, our goal is to find a \( K \)-invariant metric in the trivial Kähler class which has the same volume form as \( dV_H \).

\begin{rmk} \label{remark_invariance_implies_exactness}
In what follows, we only consider \( K\)-invariant \( \del \delb\)-exact Kähler metrics, but note that on non-Hermitian symmetric spaces, \( K \)-invariance already implies \( \del \delb\)-exactness by a result of Azad and Loeb \emph{\cite{AL92}}.
\end{rmk}

Following \cite{AL92}, any \( K \)-invariant \( \del \delb\)-exact metric \( \omega = i \del \delb \Phi \) on \( G/H \) is determined by a \( W\)-invariant convex function \( \rho \) on the Cartan algebra \( \afrak \). The problem then amounts to solving the following real Monge-Ampère equation.

\begin{thm}[\!\!\cite{Del20}] \label{theorem_calabi_yau_equation_symmetric_space}
A smooth \( K \)-invariant strictly psh function \( \Phi \) defines a Ricci-flat metric on \( G/H \) if and only if \( \rho(x) := \Phi(\exp(x)H) \) satisfies for all \( x \in \afrak \): 
\begin{equation} \label{equation_ricci_flat_symmetric_space}
\det (d^2 \rho) \prod_{\alpha \in R^{+}} \sprod{\alpha, d \rho}^{m_{\alpha}} = \prod_{\alpha \in R^{+}}  \sinh^{m_{\alpha}}( \alpha(x)).
\end{equation}
\end{thm}

\begin{defn} \label{definition_ricci_potential}
For a \( W \)-invariant convex function \( \rho \) on \( \afrak \), the function
\begin{equation}  \label{equation_ricci_potential}
\mathcal{P}(\rho ) := \ln \det(d^2 \rho ) + \sum_{\alpha \in R_{+}} m_{\alpha} (\ln \sprod{\alpha, d \rho } - \ln \sinh \alpha ) 
\end{equation}
is called the \emph{Ricci potential} of \( \rho \). 
\end{defn}

\subsection{The rank one case} \label{subsection_rank_one_case}


\begin{proof}[Proof of Theorem \ref{theorem_rank_one_cyss}]
Let \( G/H \) be a rank-one symmetric space of dimension \( n \) with positive roots \( (\alpha, 2 \alpha) \) of multiplicities \( (m, \wh{m}) \), where we allow \( \wh{m} = 0 \) for symmetric spaces of restricted root system \( A_1 \). \textcolor{blue}{In particular,}
\[ n = 1 + m + \wh{m}. \] 
The candidate for the tangent cone at infinity of \( G/H \) can be obtained as follows. Let \( \nu \) be a \( G\)-invariant valuation on \( G/H \). Then the total space of a \( G\)-equivariant degeneration of \( G/H \) induces by \( \nu \) is a spherical embedding of \( G/H \times \Cbb^{*} \). We choose a total space corresponding to the colored cone \( \set{\text{Cone}(\nu,1), \Dcal^{*}} \) where \( \Dcal^{*} \) are the colors of \( G/H \times \Cbb^{*} \) (see Subsection \ref{subsection_equivariant_degeneration_and_spherical_cones}). The central fiber of the degeneration is the divisor corresponding to the ray \( (\nu,1) \), which is a \( G\)-horospherical cone \( C \) of rank one again by \cite{BP87}.

The smoothness of the link is clear, because singularities occur in \( G\)-orbits, and the \( G\)-orbits of a rank one cone are the fixed point and the open \( G\)-orbit. The cone \( C \) has the same weight lattice, hence same Cartan space \( \afrak \) as \( G/H \) (cf. Definition \ref{definition_weight_lattice}). 
The total sum of the positive root of \( G/H \) is   
\[ \varpi = (m + 2 \wh{m} ) \alpha = (n-1 + \wh{m}) \alpha, \]
\textcolor{blue}{given the relation \( n = 1 + m + \wh{m} \)}. 
Using the root and fundamental coweight \( (\alpha, \omega_{\alpha}^{\vee} ) \) as basis on \( (\afrak^{*}, \afrak) \), the \( K \)-invariant conical Calabi-Yau potential on \( C \), which is uniquely determined by a strictly convex positive function \( u : \afrak \to \Rbb \), satisfies the following equation
\begin{equation} \label{equation_horospherical_conical} 
u''(x)u'(x)^{m + \wh{m}} = e^{ (m + 2 \wh{m}) x } =  e^{ (n-1+\wh{m}) x }, 
\end{equation} 
where \( e^{(n-1+\wh{m}) x} \) is the volume density of the canonical volume form on \( C \) \cite{Ngh22b}.
A generic solution is then
\[ u(x) = \int_0^x \tuple{\frac{n}{(n-1 + \wh{m})} e^{(n-1+\wh{m})t } + A}^{1/n} dt + B. \] 
The Reeb vector \( \xi = \lambda \omega_{\alpha}^{\vee} \) is uniquely determined by \( \sprod{\varpi, \xi} = n \), i.e. 
\[ \xi = \frac{n}{ n-1 + \wh{m}} \omega_{\alpha}^{\vee}. \]
The solution of Equation \eqref{equation_horospherical_conical} suggests we take 
 \[ u(x) =  ce^{ \frac{n-1+\wh{m}}{n} x }\]
 as the Calabi ansatz for a suitable constant \( c  > 0 \). The Ricci-flat ODE in rank one writes
 \begin{equation} \label{equation_rank_one_ricci_flat}
 \rho''(x) \rho'(x)^{m + \wh{m}} = \sinh^{m}(x) \sinh^{\wh{m}}(2x). 
 \end{equation}
 Plugging the \textcolor{blue}{Calabi ansatz} into Equation \eqref{equation_rank_one_ricci_flat} yields an asymptotic solution with Ricci potential decaying at rate \( O(e^{-2x}) = O(r^{-\frac{4n}{n-1+\wh{m}}}) \), where \( r^2 = e^{\frac{n-1+\wh{m}_2}{n} x}  \) is the distance function on the cone. Also, one can average \( u \) by the action of \( W = \set{1,-1} \) to obtain a smooth, even, strictly convex function \( (u(x)+u(-x))/2 \) on \( \afrak \), still denoted by \( u \),  which is also an asymptotic solution to \eqref{equation_rank_one_ricci_flat} having the same Ricci potential decay as the Calabi ansatz, and defines a \(  K\)-invariant trivial Kähler class \( [ i\del \delb u ] \) on \( G/H \). 
 
 Since \(u'' \) is asymptotic to \( r'' \) at rate \( O(e^{-2 (n-1+\wh{m})x /n}) = O(r^{-4}) \), the global \( (1,1) \)-form defined by \( u \) then decays at rate \( O(r^{-4}) \). Note that in this case, the volume density of \( G/H \) satisfies 
 \[ V(x) = e^{(n-1+\wh{m})x}( 1+ O(e^{-x})), \]
 which is asymptotically the density of the canonical volume form of \( C \) at rate \( O(e^{-x}) =  O(r^{-\frac{2n}{n-1+\wh{m}}}) \). 
By \cite{Ah83} there is a \( G\)-equivariant compactification \( X \) of \( G/H \) such that \( X \backslash (G/H) = D \) is a Kähler-Einstein horospherical Fano manifold and satisfies as an irreducible divisor 
 \[ -K_X = (A+1) D, \quad  A \in \Nbb^{*}_{+}. \] 
 The manifold \( D \) is indeed Fano by adjunction, and Kähler-Einstein since \( D \) is a spherical variety of rank \( 0 \), hence a flag manifold (in particular horospherical).
 Moreover the conormal bundle over \( D \) with zero section contracted \( (N_D^{*})^{\times} \simeq ( \frac{1}{A} K_D)^{\times} \) is in fact isomorphic to the regular Calabi-Yau cone \( C \) obtained by equivariant degeneration.
 Since the trivial class contains a form with decay rate \( O(r^{-4}) \), appealing to Conlon-Hein's existence theorem \cite{CH15} then provides us an asymptotically conical Calabi-Yau metric on \( G/H \) in the trivial Kähler class whose tangent cone at infinity is the horospherical tangent cone, with rate of convergence
 \[ \frac{2n}{n-1+\wh{m}}. \] 
Uniqueness of the metric in the class up to one scaling parameter follows from \cite[Proposition 21]{Myk11}: for every \( a > 0 \), there is a unique Calabi-Yau metric \( \omega_a \) and any other such metric \( \omega_b \) is given by \( \omega_b = \frac{b}{a} \omega_a \). Smoothness implies that the cone is unique by Colding-Minicozzi's uniqueness theorem \cite{CM14}.
\end{proof}

\begin{ex} The manifold \( M := T \Sbb^n \) with a suitable complex structure is biholomorphic to \( \SO(n+1)/ \SO(n) \) endowed with the complex symmetric space structure. 

When \( n > 2 \), \( M \) is non-Hermitian with \( (m, \wh{m}) = (n-1,0) \) and admits a \( \SO_{\Rbb}(n+1)\)-invariant \textcolor{blue}{(hence \( \del \delb\)-exact, cf. Remark \ref{remark_invariance_implies_exactness})} Calabi-Yau metric asymptotically conical to the ordinary double point \( (C,g_0) \) in \( \Cbb^{n+1} \), where 
\[ C = \set{z_1^2 + \dots + z_{n+1}^2 = 0} \] 
is Calabi-Yau with the \( \SO_{\Rbb}(n+1) \)-invariant cone metric \( \omega_0 = \frac{i}{2} \del \delb \abs{z}^{2 - 2/n} \). This is the unique \( \SO_{\Rbb}(n+1)\)-invariant Calabi-Yau metric on \( C\), viewed as a horospherical cone under the \textcolor{blue}{induced \( \SO(n+1) \times \Cbb^{*} \)} action. Geometrically, \( M \) can be interpreted as an equivariant smoothing of the ordinary double point. 

\end{ex}

\begin{rmk}
The rate of convergence is likely to be optimal  for all rank-one symmetric spaces in the following sense: there is a diffeomorphism \( \Phi \) between \( G/H \) and \( C \) outside of a compact subset such that \( h := \Phi^{*} g - g_0 \) satisfies the Bianchi gauge condition
\[ \text{div}_{g_0}(h - \frac{1}{2} \tr_{g_0}(h) g_0 ) = 0. \] 
This is proved for the case \( M = \SO(n+1)/ \SO(n) \) with \( n > 2 \) by realizing \( \Phi \) as a projection of \( C\) onto \( M \) in some \( \Cbb^N \) \emph{\cite{CH13}}. 
\end{rmk}


\section{Geometric prelude} \label{section_geometric_prelude}

\subsection{Geometric constructions}
In higher rank, there are more choices of equivariant degenerations, hence more candidates for the tangent cone at infinity. Compared with the rank-one case, a new feature in higher rank is that the potential tangent cones are singular in general, so the Calabi ansatz from the cone does not yield directly a globally smooth potential on the symmetric space. One must then appeal to gluing techniques, and decay rate improvement methods to get a global potential with good decay properties. 

The candidates for tangent cones of the symmetric space in the rank-two cases can be geometrically obtained as follows. 

\begin{itemize} 
\item First given a rational \( G\)-invariant valuation in \( G/H \), one can embed \( G/H \) into a symmetric projective variety \( X \) using the colored cone defined by a rational valuation \( \beta_0 \) in the interior of the valuation cone \( \Vcal \) (recall that it is the negative Weyl chamber) and all the colors \( \Dcal \) of \( G/H\) (see Figure \ref{figure_colored_cone_compactification}). 

By the orbit-cone correspondence (cf. Theorem \ref{proposition_orbit_cone_correspondence}), the variety \( X \) consists of the open dense orbit, two closed \( G\)-orbits \( G/P_i \), \( i = 1,2 \), of codimension \( \geq 2 \) (which are flag manifolds), and a unique \( G \)-invariant horospherical divisor \( D_0^{\vee} \) with open orbit \( G/H_0 \) and two closed orbits \( G/P_i \).

\item Taking the affine cone over \( X \) (with respect to its ``canonical'' \textcolor{blue}{\(G\)-linearized} polarization) yields the total space \( C(X) \) of the degeneration. The central fiber of this family yields the \textcolor{blue}{candidate for the tangent cone \( C_0 \)}, which is a \( G\)-horospherical cone with open orbit \( G/H_0 \times \Cbb^{*} \). 


\end{itemize} 

Since our compactification \( X \) is generally of Picard rank one with no condition on \( H \) and \( R \) whatsoever, it is in general not smooth and not Fano (cf. Remark \ref{remark_symmetric_compact_not_smooth}). By \cite{BP87}, the rational valuations in the interior of \( \Vcal \) yield \( G\)-isomorphic candidates for the tangent cones at infinity, so the construction does not depend on the choice of such valuation. 

\begin{rmk} \label{remark_cartanalgebra_coincidence}
Let \( T_{0} \) be the identity component of \( \Aut_G(C_0) \) and \( \tfrak_{nc} \) be the noncompact Lie algebra of \( T_{0} \) as in \emph{\cite{Ngh22b}}. Since \( C_0 \) is an equivariant degeneration of \( G/H \), their function fields coincide as \( G\)-modules, hence their weight lattices coincide, so \( \tfrak_{nc} = \afrak \) as duals of the weight lattices tensored over \( \Rbb \). The fact that \( \tfrak_{nc} = \afrak \) allows us to reduce questions of Gromov-Hausdorff convergence to equivalence of functions on \( \afrak \). 
\end{rmk}

To ``desingularize'' the \textcolor{blue}{Calabi} ansatz from the cone \textcolor{blue}{\( C_0 \)}, we consider the blow-up \( \wt{X} \) of \( X \) along its \( G\)-closed orbits so that
\[ \wt{X} \backslash (G/H) = D_0 \cup D_1 \cup D_2, \]
where \( D_1, D_2 \) are the divisors corresponding to the Weyl walls \( \wt{\alpha}_1, \wt{\alpha}_2 \). Let us provide some details on the geometry of the divisors.

\begin{itemize} 
\item The open orbit \( G/H_i \) of each \( D_i \) (\( i \in \set{1,2} \)) is a fibration over  the flag manifold \( G/P_i \) with fiber a symmetric space of rank one \( X_i \) of \( \dim_{\Cbb} X_i = 1+ \sum_{\alpha_i \nmid \alpha} m_{\alpha} \) (see \cite[Subsection 3.2]{BD19} for the construction of \( G/H_i \)). By \cite{DCP83}[Theorem 5.2], each \( D_i \) is itself a fibration over \( G/P_i\) with fiber the wonderful compactification \( \ol{X}_i \) of \( X_i \). 

The divisor \( D_0 \) is a three-orbits variety, which consists of the open dense orbit \( G/H_0 \) and two closed orbits \( D_0 \cap D_{i}\) \((i \in \set{1,2})\) which are the only possible singularities of \( D_0 \). Moreover, \( D_0 \cap D_{i} \) is a homogeneous fibration over \( G/P_{i} \) with fibers \( \ol{X}_{i} \backslash X_{i} \), the normal bundle in \( \ol{X}_{i} \) of which corresponds to the Stenzel asymptotic cone (see Proof of Theorem \ref{theorem_rank_one_cyss}).  
\item The \( G \)-equivariant blowup morphism \(\wt{\pi}: \wt{X} \to X \) when restricted to \( D_0 \) is the contraction of \( D_0 \) along \( D_0 \cap D_{i} \), hence \( \wt{\pi}(D_0) = D_0^{\vee} \) and \( \wt{\pi}(D_0 \cap D_{i}) = G/P_{i} \). Thus \( \wt{\pi} \) partially resolves the singularities of \( D_0^{\vee} \). In case $G/H$ has a wonderful compactification, the variety \( \wt{X} \) can also be viewed as the blowup of the wonderful compactification along the unique closed orbit. 

\item Note that \( D_0, D_1, D_2 \) are not Fano, but their \( G\)-open orbits may admit Fano embeddings (for instance the contraction of \( D_{i} \) along \( D_0 \cap D_{i} \)). The horospherical space \( G/H_0 \times \Cbb^{*} \), which is also the smooth locus \( N_{D_0 / \wt{X}}^{\text{reg}} \) of the normal line bundle \( N_{D_0/ \wt{X}} \), has a conical Calabi-Yau metric whose completion is the cone \( C_0 \) (see Theorem \ref{theorem_kstable_horospherical_cone}). 

\item 
The conical Calabi-Yau metric on \( C_0 \) is singular in all rank two cases (see Proposition \ref{proposition_tangentcones_singular}). In order to obtain a smooth asymptotic solution from the Calabi ansatz \ref{ansatz_calabi}, we proceed to build smooth models near the partially resolved singularities \( D_0 \cap D_1, D_0 \cap D_2 \) based on the Stenzel model (as suggested by their geometry) and the asymptotic behavior of the ansatz near \( D_1, D_2 \) (cf. Theorem \ref{theorem_asymptotic_behavior_KE_fano_potential}). It turns out that the ansatz behaves like a warped product of AC metrics near the singularities, cf. \eqref{equation_asymptotic_behavior} and Ansatz \ref{ansatz_boundary}. 

\item To be precise, the smooth models are built on \( G/H_{1,2} \times \Cbb^{*} = N_{D_{1,2}/\wt{X}}^{\text{reg}} \), then glued with the conical potential on \( N_{D_0 / \wt{X}}^{\text{reg}} \). 
\end{itemize} 

The directions defined by the roots \( \alpha_1, \alpha_2 \) as elements in \( \afrak^{*} \) parametrize directions to the divisors \( D_1, D_2 \) at infinity of \( G/H \). 
Namely, if \( x_k = \exp(v_k) \to x_{\infty} \in \cap_{j \in J} D_j\), where \( J \subset \set{1,2}\), then \( \alpha_j(v_k) \to +\infty \) for all \( j \in J \). To capture the behavior near \( D_1 \) and \( D_2 \), we will use orthogonal coordinates \( (\wt{\alpha}_1, \alpha_2), (\alpha_1, \wt{\alpha}_2) \) in \( \afrak^{*} \).  We make a small abuse of notation and use them both as coordinates in \( \afrak^{*} \) and as elements in the root system. To get a rough idea of the situation, the reader may keep in mind the following simplistic figure: the toric embedding \( (\Cbb^{*})^2 \subset \Cbb^2 \) with the invariant Calabi-Yau potential \( \rho(x_1,x_2) = e^{x_1} + e^{x_2} \) in logarithmic coordinates \( x_i = \log \abs{z_i}^2 \), parametrizing directions to the divisors \( \set{0} \times \Cbb, \Cbb \times \set{0} \). 

Let \( \beta \) be a (possibly irrational) direction in the \textit{interior} of the Weyl chamber. As stated in the introduction, we will mainly consider the generic region inside the Weyl chamber where \( \beta \sim \alpha_2 \gg 1, \beta \sim \alpha_1 \gg 1 \), and the two boundary regions \( \beta \sim \alpha_2 \gg \alpha_1 \sim 1 \) and \( \beta \sim \alpha_1 \gg \alpha_2 \sim 1 \), which are geometrically \( D_0 \backslash (D_1 \cup D_2 ) \), \( D_0 \cap D_2 \backslash D_1 \) and \( D_0 \cap D_1 \backslash D_2 \), respectively.

\begin{figure}
\begin{tikzpicture}
\pgfmathsetmacro\ax{2}
\pgfmathsetmacro\ay{0}
\pgfmathsetmacro\bx{2 * cos(120)}
\pgfmathsetmacro\by{2 * sin(120)}
\pgfmathsetmacro\lax{2*\ax/3 + \bx/3}
\pgfmathsetmacro\lay{2*\ay/3 + \by/3}
\pgfmathsetmacro\lbx{\ax/3 + 2*\bx/3}
\pgfmathsetmacro\lby{\ay/3 + 2*\by/3}

\tikzstyle{couleur_pl}=[circle,draw=black!50,fill=blue!20,thick, inner sep = 0pt, minimum size = 2mm]


\draw[->, ultra thick] (0,0) -- (\ax,\ay) node[below right] {\( {\alpha}_1^{\vee} \)};
\draw[->, ultra thick] (0,0) -- (\bx, \by) node[above right] {\( {\alpha}_2^{\vee} \)};
\draw[->, ultra thick] (0,0) -- (-\ax - \bx, -\ay - \by);
\node at (\ax, \ay) [couleur_pl] {};
\node at (\bx, \by) [couleur_pl] {};

\draw (-2,0)--(2,0);
\draw (0,-2)--(0,2);

\path[draw, pattern=horizontal lines](0,0) -- (-2*\lax, -2*\lay) -- (-2*\lbx, -2*\lby) -- cycle;

\path[draw, pattern=horizontal lines] (0,0)--(\ax,\ay)--(-\ax - \bx,-\ay - \by) -- cycle;
\path[draw, pattern=horizontal lines] (0,0)--(\bx,\by)--(-\ax - \bx,- \ay - \by) -- cycle;

\draw[->, ultra thick] (0,0) -- (-\ax - \bx, -\ay - \by);
\draw[dashed] (0,0) -- (-2*\lbx, -2*\lby);
\draw[dashed] (0,0) -- (-2*\lax, -2*\lay);

\end{tikzpicture}
\qquad
\begin{tikzpicture}
\pgfmathsetmacro\ax{2}
\pgfmathsetmacro\ay{0}
\pgfmathsetmacro\bx{2 * cos(120)}
\pgfmathsetmacro\by{2 * sin(120)}
\pgfmathsetmacro\lax{2*\ax/3 + \bx/3}
\pgfmathsetmacro\lay{2*\ay/3 + \by/3}
\pgfmathsetmacro\lbx{\ax/3 + 2*\bx/3}
\pgfmathsetmacro\lby{\ay/3 + 2*\by/3}

\tikzstyle{couleur_pl}=[circle,draw=black!50,fill=blue!20,thick, inner sep = 0pt, minimum size = 2mm]


\draw (0,0) -- (\ax,\ay) node[below right] {\( {\alpha}_1^{\vee} \)};
\draw (0,0) -- (\bx, \by) node[above right] {\( {\alpha}_2^{\vee} \)};
\draw[->, ultra thick] (0,0) -- (-\ax - \bx, -\ay - \by);
\node at (\ax, \ay) [couleur_pl] {};
\node at (\bx, \by) [couleur_pl] {};

\draw (-2,0)--(2,0);
\draw (0,-2)--(0,2);

\path[draw, pattern=horizontal lines] (0,0) -- (-2*\lax, -2*\lay) -- (-2*\lbx, -2*\lby) -- cycle;
\draw[->, ultra thick](0,0) -- (-2*\lbx, -2*\lby);
\draw[->, ultra thick](0,0) -- (-2*\lax, -2*\lay);
\draw[->, ultra thick] (0,0) -- (-\ax - \bx, -\ay - \by);
\end{tikzpicture}
\caption{From  left to right: the colored fans of the compactification \( X \) in the \( A_2 \) cases and its blowup \( \wt{X} \). The colored fan of \( X \) consists of two colored ones \( (\text{Cone}(\alpha_{1,2}^{\vee}, \beta_0), \Dcal_{1,2}) \), where \( \Dcal_i \) is the set of colors such that \( \sigma(\Dcal_i) = \alpha_{i}^{\vee} \).  Here \( \alpha_{1,2}^{\vee} \) denote the restricted coroots which are the images of the colors of \( X \) (bijective images in the \( A_2 \) cases \cite[2.4, Proposition 1]{Vus90}). The colored fan of \( \wt{X} \) is obtained by replacing the colored cones \( (\text{Cone}(\alpha_{1,2}^{\vee}, \beta_0), \alpha_{1,2}^{\vee}) \) with \( (\text{Cone}(-\wt{\alpha}_{1,2}, \beta_0), \varnothing) \).}
\label{figure_colored_cone_compactification}
\end{figure}

\begin{ex} \label{example_compactification_r1r1}
The \( (n-1)\)-dimensional \emph{projective quadric} \( Q^{n-1} \) defined by the equation \( \sum_{j=1}^{n+1} z_j^2 = 0 \) in \( \Pbb^{n-1}_{z_1, \dots, z_{n+1}} \), can be viewed as the compactification \( X \) of the symmetric space \( (\SO_{m_1+1} \times \SO_{m_2+1}) / H \) where \( H = \text{Stab}[1:0 \dots : 1 : \dots : 0] \). Although the restricted root system is \( A_1 \times A_1 \), this is not a product of two symmetric spaces of restricted root system \( A_1 \), but a finite covering of the latter, cf. \emph{\cite{Ruz10}[Theorem 5, (1)]}. The quadric \( Q^{n-1} \) is described in \cite{Ruz10} in terms of Grassmannians. Here we provide the more concrete description of \( Q^{n-1} \) in \emph{\cite[Sections 3.1, 3.2]{Del22}}. 

Let \( \pi_1, \pi_2 \) be the rational projections from \( \Pbb^n \) to \( \Pbb^{m_1} \) and \( \Pbb^{m_2} \). Define
\[ Q^{m_1} := \set{[z] \in \Pbb^{m_1}, \sum_{j=1}^{m_1+1} z_j^2 = 0 }, \; Q^{m_2} = \set{ [z] \in \Pbb^{m_2}, \sum_{j=m_1+2}^{n+1} z_j^2 = 0}. \]
The open orbit \( O \) is formed by the points in \( Q^{n-1} \) where both \( \pi_1, \pi_2 \) are well-defined and not lying in \( Q^{m_1}, Q^{m_2} \), namely
\begin{align*} 
O = Q^{n-1} \backslash &\set{[z], \; z_1 = \dots = z_{m_1+1} = 0} \cup \set{[z], \; z_{m_1 +2} = \dots = z_{n+1} = 0} \\ &\cup \set{ [z], \; \pi_1([z]) \in Q^{m_1}} \cup \set{[z], \; \pi_2([z]) \in Q^{m_2}} . 
\end{align*}
The two closed orbits are formed by the points in \( Q^{n-1} \) where \( \pi_1 \) or \( \pi_2 \) are not defined, so the two closed orbits are isomorphic to \( Q^{m_1}, Q^{m_2} \). 
The codimension one orbit \( D_0^{\vee} \) is formed by the points in \( Q^{n-1}\) where both \( \pi_1, \pi_2 \) are defined and lie in the corresponding quadrics. In short, \( D_0^{\vee} \) is 
\[ (\set{ [z], \; \pi_1([z]) \in Q^{m_1}} \cup \set{[z], \; \pi_2([z]) \in Q^{m_2}}) \cap Q^{n-1}. \]
The compactification \( \wt{X} \) is given as \( \text{Bl}_{Q^{m_1}, Q^{m_2}}(Q^{n-1}) \) where the corresponding divisor \( D_0 \) is a blowup of \( D_0^{\vee} \).  
\end{ex}

\begin{ex} \label{example_compactification_smoothdiv_finiteaut}
Let us now describe the compactification \( X \) of \( O := \SL_3/ \SO_3 \). In fact, \( O \) can be viewed as an open subset of the projectivization of the quadratic forms in \( S^2 (\Cbb^3)^{*} \) with unit determinant. In other words, \( O \) is the following affine subset of \( \Pbb( \Cbb \oplus S^2(\Cbb^3)^{*} )_{z_0, \dots, z_6} \):
\[ O = [Q \in \Pbb(S^2 (\Cbb^3)^{*}), \det(Q) = z_0^3 ] \cap [z_0 = 1].  \]
Let \( \wh{\omega}_{1,2} \) be the fundamental weights of \( G \) with highest weight vectors \( v_{1,2} \) where \( v_1 \in S^2(\Cbb^3)^{*} \), \( v_2 \in S^2(\Cbb^3) \). The vector spaces \( S^2(\Cbb^3)^{*}, S^2 (\Cbb^3) \) are in fact the simple \( \SL_3 \)-modules with highest weights \( \wh{\omega}_1, \wh{\omega}_2 \), respectively. Let \( Q^{\vee} \) be the adjoint matrix of \( Q \). We embed \( O \) by the map 
\begin{align*}
O &\hookrightarrow \Pbb(\Cbb \oplus S^2 (\Cbb^3)^{*} \oplus S^2(\Cbb^3)) \simeq \Pbb^{12} \\
Q &\to [1, \quad Q, \quad Q^{\vee}].
\end{align*}
Taking the closure of \( O \) inside this projective space then yields the compactification \( X \) with the unique \( G\)-stable divisor
\[ D_0^{\vee} := \set{\det(Q)= 0} \cap X. \] 
Alternatively, \( D_0^{\vee} \) is given by the closure in \( \Pbb(S^2(\Cbb^3)^{*} \oplus S^2 (\Cbb^3)) \simeq \Pbb^{11} \) of \( G(v_{1} + v_2) \), so \( D_0^{\vee} \) is in particular horospherical. The \( G\)-valuation associated to \( D_0^{\vee} \) is \( - (\wh{\omega}_1^{\vee} + \wh{\omega}_2^{\vee})/2 \).

By \emph{\cite[Theorem 3]{Ruz10}}, \( \Aut^0(X) = \SL_3 \), hence 
\[ \Aut_G^0(X) = Z( \SL_3) = \set{\lambda I_3, \lambda \in \Cbb, \lambda^3 = 1} \simeq \Zbb_3, \]
and \( \Aut_G^0 (C,\xi) \) is \( (\Cbb^{*})^2 \).  In particular, there can only be finitely many Kähler-Einstein metrics on the pair \((X,D_0^{\vee}) \), while there is a two-parameters family of conical Calabi-Yau metrics on \( C \), hence a priori two-parameters choice of the Calabi ansatz to build Calabi-Yau metrics.
\end{ex}

\subsection{Singularity of tangent cones}
We will use Pasquier's smoothness criterion for horospherical varieties to prove that all of the tangent cones are in fact singular. With the same setup as in \ref{subsection_setup}, we denote by \( \wh{S} \) the set of simple roots with respect to the choice of positive roots \( \wh{R}^{+} \) in \( (G,T) \). 

A \( G\)-horospherical homogeneous space is uniquely determined by a sublattice \( \Mcal \) of \( \Zbb \wh{S} \), and a subset \( I \subset \wh{S} \)  such that \( \sprod{\chi, \wh{\alpha}^{\vee}} = 0 \) for all \( \chi \in \Mcal \) and \( \wh{\alpha} \in I \). The dual lattice \( \Ncal \) of \( \Mcal \) is called the coweight lattice.

The set \( I \) corresponds to the parabolic subgroup \( P_I \) that left-stabilizes the open Borel orbit, and there is a one-to-one correspondence between the set of colors \( \Dcal \) and the coroots of the roots in \( \wh{S} \backslash I \). 

\begin{thm}[\!\!\cite{Pas06}]
Let \( \sigma \) be the map from \( \Dcal \) that sends a divisor to its valuation viewed in \( \Ncal \). A simple \( G\)-horospherical variety with colored cone \( (\Ccal, \Fcal) \) is \textit{locally factorial} if and only if the following two conditions are satisfied:
\begin{itemize} 
\item The elements of \( \Fcal \) have pairwise distinct images by \( \sigma \).
\item \( \Ccal \) is generated by a basis of \( \Ncal \) containing \( \sigma(\Fcal) \). 
\end{itemize} 
\end{thm} 

\begin{defn}[\!\!\cite{Pas06}]
Let \( I, J \) be two disjoint subsets of \( \wh{S} \). Let \( \Gamma_{\wh{S}} \) be the Dynkin diagram of \( G \), and \( \Gamma_{I \cup J} \) be a subgraph of \( \Gamma_{\wh{S}} \) whose vertices consist of the elements in \( I \cup J \) and the edges those in \( \Gamma_{\wh{S}} \) joining two elements in \( I \cup J \). 

The couple \( (I,J) \) is said to be smooth if every connected component \( \Gamma \) of \( \Gamma_{I \cup J} \) satisfies one of the following conditions. 
\begin{enumerate}
    \item \( \Gamma \) is a Dynkin diagram of type \( A_n \) whose vertices are all in \( I \), except one of the two endpoints being in \( J \). 
    \item \( \Gamma \) is a Dynkin diagram of type \( C_n \) whose vertices are all in \( I \), except the simple (i.e. not connected to the double edge) endpoint being in \( J \). 
    \item \( \Gamma \) is a Dynkin diagram of arbitrary type whose vertices are all in \( I \). 
\end{enumerate}
\end{defn}

\begin{thm}[\!\!\cite{Pas06}] 
A simple embedding \( E \) with colored cone \( (\Ccal, \Fcal) \) of a horospherical space \( G/H \) is smooth if and only if the following conditions hold. 
\begin{itemize}
    \item \( E \) is locally factorial. 
    \item If we let \( J_{\Fcal} = \set{\wh{\alpha} \in \wh{S} \backslash I, D_{\wh{\alpha}} \in \Fcal} \), then the couple \( (I, J_{\Fcal}) \) is smooth. 
\end{itemize}
\end{thm}

Let \( \beta_0 \) be the ray corresponding to \( D_0^{\vee} \) in the colored fan of the compactification \( X \). We view \( D_0^{\vee} \) as the base of the tangent cone \( C_0 \), as well as the unique \( G\)-stable horospherical divisor inside \( X \) with open orbit \( G/H_0 \). 

Since the tangent cone does not depend on the chosen valuation inside the Weyl chamber \cite{Ngh22b}, one can take \( \beta_0 = \alpha_1 + \alpha_2 \). Thanks to \cite[Théorème 3.6]{BP87}, we can determine the combinatorial data of the homogeneous horospherical space \( G/H_0 \) in terms of the data of \( G/H \). The combinatorial data of the open orbit \( G/H_0 \times \Cbb^{*} \) inside \( C_0 \) then easily follows. 

\begin{lem} \label{lemma_combinatorial_data_asymptotic_cone}
Suppose that \( \Mcal(G/H) = \Zbb \sprod{\omega_{1}, \omega_{2}} \), where \( \omega_{1,2} \) are the restricted fundamental weights of \( G/H \) with respect to Convention \ref{convention}. Then
\leavevmode
\begin{itemize} 
\item The lattice \( \Mcal_0 \) of \( G/H_0 \) corresponds to \( \Zbb (\omega_{1} - \omega_{2})\). In particular, \( \Ncal_0 = \Zbb \frac{\alpha_1^{\vee} - \alpha_2^{\vee}}{2}\).
\item The subset \( I \) that determines \( G/H_0 \) is the set \( \set{ \wh{\alpha} \in \wh{S}, \wh{\alpha} = \theta(\wh{\alpha})} \). In particular, the roots in \( I \) correspond to the black nodes in the Satake diagram of \( G/H \). 
\end{itemize}
\end{lem} 

\begin{proof}
If \( \beta_0 = \alpha_1 + \alpha_2 \) generates the ray corresponding to \( D_0^{\vee} \) in the colored fan of \( X\), then \( \Mcal_0 = \beta_0^{\perp} \cap \Mcal \). Hence \( \Mcal_0 = \Zbb(\omega_{1} - \omega_{2}) \). 

Since the left-stabilizer of the open Borel orbit of \( G/H \) is the same as that of \( G/H_0 \) \cite{BP87}, \( \wh{S} \backslash I \) consists of positive simple roots \( \wh{\alpha} \) such that \( \wh{\alpha}|_{\afrak} \neq \theta(\wh{\alpha})|_{\afrak} \), hence \(  I\) consists of the black nodes in the Satake diagram of \( G/H \). 
\end{proof}


A spherical cone has smooth link iff it is smooth outside the unique fixed point \( \set{0} \) of \( G \). Since the cone \( C_0 \backslash \set{0}\) consists of two rank one simple embeddings, it is enough to show that one of the embeddings is singular via their combinatorial data. 

\begin{lem} 
The colored cones associated to the rank one simple embeddings in \( C_0 \backslash \set{0} \) are \( (\Ccal_{i},\Fcal_{i}), i = 1,2 \), where \(\Ccal_i =  \Rbb_{\geq 0} \alpha_i^{\vee}, \Fcal_i = \set{D_{\wh{\alpha} }, \wh{\alpha} \in \wh{S}, \wh{\alpha} - \theta(\wh{\alpha}) = \alpha_i } \).

The couple of simple roots associated to each rank one embedding of \( C_0 \backslash \set{0} \) is
\[ J_{\Fcal_i} = \set{\wh{\alpha} \in \wh{S} \backslash I, \; \wh{\alpha} - \theta(\wh{\alpha}) = \alpha_i}, \quad I_i = \wh{S} \backslash J_{\Fcal_i}, \; i = 1,2. \] 
In particular, \( J_{\Fcal_i} \) corresponds to the white nodes in the Satake diagram that restrict to the simple restricted root \( \alpha_i \). 
The graphs \( \Gamma_{I_i \cup J_{\Fcal_i}} \) of the pairs \( (I_i, J_{\Fcal_i}) \) of the indecomposable spaces are summarized in Table \ref{table_singularities_tc}. 
\end{lem}

\begin{proof}
Following \cite{Ngh22b}, the support of the colored cone of \( C_0 \) is the convex cone generated by \( \alpha_{1,2}^{\vee} \) and the colors of \( C_0 \) are in bijection with \( \wh{S} \backslash I \). The rank one embeddings in \( C_0 \backslash \set{0} \) then correspond to the two one-dimensional colored cones \( (\Ccal_i, \Fcal_i), i = 1,2 \) by the orbit-cone correspondence. Since the embeddings are themselves horospherical cones, it follows from Theorem \ref{theorem_conical_embedding_colored_cone} and definition of \( J_{\Fcal_i} \) that \( I_i = \wh{S} \backslash J_{\Fcal_i} \). 
\end{proof}

\begin{prop} \label{proposition_tangentcones_singular}
All the candidates for horospherical tangent cones at infinity of the rank two symmetric spaces are singular. Moreover, the tangent cones at infinity of the decomposable ones are product of the Stenzel asymptotic cones on each factor.  
\end{prop}
\begin{proof}
From Table \ref{table_singularities_tc} and Pasquier's criterion, it is straightforward to see that the only tangent cones with both couples \( (I_i, J_{\Fcal_i}), i=1,2 \) being smooth are the horospherical tangent cones of the symmetric spaces of involution type \( AI, A_2, G \). 

In these cases, suppose that our symmetric space has maximal weight lattice \( \Zbb \sprod{\omega_1, \omega_2} \). Since \( \theta(\alpha_1) = - \alpha_1 \) for these spaces, we have \( \alpha_1^{\vee} = \wh{\alpha}_1^{\vee}/2 \) (see \cite[Section 2.3]{Vus90}, or \cite[Definition 1.22]{Del20b}), so \( \omega_1 = 2 \wh{\omega}_1 \). But \( \sigma(D_{{\wh{\alpha}}_1}) = \wh{\alpha}_1^{\vee}|_{\Mcal_0} = 2 \alpha_1^{\vee}|_{\Mcal_0} \), hence \( \sigma(D_{{\wh{\alpha}}_1})( \omega_1 - \omega_2) = 2 \), so \( \Ccal_1 \) is not generated by \( \sigma(D_{{\wh{\alpha}}_1}) \), i.e. the cone is not locally factorial. The same argument applies for symmetric spaces whose weight lattice is an integer multiple of \( \Zbb \sprod{\omega_1, \omega_2} \).

The same reasoning as in Lemma \ref{lemma_combinatorial_data_asymptotic_cone} and the fact that \( I = \wh{S} \backslash J \) for a horospherical cone implies that the graph \( \Gamma_{I \cup J} \) of the Stenzel asymptotic cone is exactly the Dynkin diagram of the symmetric space with \( (I,J) \) being \( (\text{black nodes}, \text{white nodes}) \) in the Satake diagram. It follows that every Stenzel asymptotic cone has a unique singularity since the graph of each space is not smooth, except for the \( A_1 \) case where \( r \) is odd (cf. Table \ref{table_rank_one_ss}). But the cone is not locally factorial in this case by the same reasoning as above. 

Finally, the asymptotic cones of the \textit{decomposable symmetric spaces} have in fact the same K-stable Reeb vector as that of the product of the Stenzel asymptotic cones (cf. Theorem \ref{theorem_kstable_horospherical_cone}). The vector when pulled back to the valuation cone degenerates the decomposable symmetric space to the product cone, which is singular as a product of two cones with an isolated singularity. This terminates our proof. 
\end{proof}

\begin{table}
\adjustbox{width=1\textwidth}{
\centering
\begin{tabular}{|c|c|c|c|c| }
\hline
Type  &  \( R \) & \( (I_1, J_{\Fcal_1}), (I_2, J_{\Fcal_2}) \) \\
\hline
\( AI \) & \( A_2 \) & \( \dynkin[edgelength=.75cm] A{to}, \dynkin[edgelength=.75cm] A{ot} \)  \\
\( A_2 \) &  \( -  \)  &
\( \begin{dynkinDiagram}[name=upper, edgelength=.75cm] A{to}
\node (current) at ($(upper root 1)+(0,-.75cm)$) {};
\dynkin[at=(current),name=lower, edgelength=.75cm] A{to}
\begin{pgfonlayer}{Dynkin behind}
\foreach \i in {1,...,2}
{
\draw[latex-latex]
($(upper root \i)$)
-- ($(lower root \i)$);
}
\end{pgfonlayer}
\end{dynkinDiagram}, 

\begin{dynkinDiagram}[name=upper, edgelength=.75cm] A{ot}
\node (current) at ($(upper root 1)+(0,-.75cm)$) {};
\dynkin[at=(current),name=lower, edgelength=.75cm] A{ot}
\begin{pgfonlayer}{Dynkin behind}
\foreach \i in {1,...,2}
{
\draw[latex-latex]
($(upper root \i)$)
-> ($(lower root \i)$);
}
\end{pgfonlayer}
\end{dynkinDiagram}
\) 
  \\   
\( AII \) &  \( - \) & \( \dynkin[edgelength = .75cm] A{*t*o*}, \dynkin[edgelength = .75cm] A{*o*t*} \) \\
\( EIV \) &  \( - \) & \( \dynkin[edgelength = .75cm] E{t****o}, \dynkin[edgelength = 0.75] E{o****t}\) \\
\hline

\( AIIIa \) & \( BC_2 \) & \( 
\dynkin[edge length = .75cm, involutions={16;25}] {A}{to*.*ot}, \dynkin[edge length = .75cm, involutions={16;25}] {A}{ot*.*to}   
\) \\
\( CIIa \)  & \( - \) & \( \dynkin C{*o*t.**}, \dynkin C{*t*o.**} \)  \\
\( DIIIa \) & \( - \) & \( \dynkin[edge length = .75cm, involutions={54}] D{*t*oo}, \dynkin[edge length = .75cm, involutions={54}] D{*o*tt}  \) \\
\( EIII \) & \( - \) & \( \dynkin[edge length = .75cm, involutions={16}] E{to***t}, \dynkin[edge length = .75cm, involutions={16}] E{ot***o} \) \\
\hline

\( BDI \) & \( B_2 \) & 
\( 
\begin{aligned} 
&\dynkin D{to*.***}, \dynkin D{ot*.***} \text{if \( r\) even}, \\
&\dynkin B{to*.**}, \dynkin B{ot*.**} \text{else}
\end{aligned} \) 
 \\
\( AIIIb \) & \( - \) & \( \dynkin[edge length =.75cm, involutions={13}] A{tot}, \dynkin[edge length =.75cm, involutions={13}] A{oto} \)  \\ 
\( DIIIb \) & \( - \) & \( \dynkin[edgelength = .75cm] D{*t*o},\dynkin[edgelength = .75cm] D{*o*t} \)  \\ 
\( B_2 \) & \( - \) & \( 
\begin{dynkinDiagram}[name=upper, edgelength=.75cm] B{to}
\node (current) at ($(upper root 1)+(0,-.75cm)$) {};
\dynkin[at=(current),name=lower, edgelength= .75cm] B{to}
\begin{pgfonlayer}{Dynkin behind}
\foreach \i in {1,...,2}
{
\draw[latex-latex]
($(upper root \i)$)
-- ($(lower root \i)$);
}
\end{pgfonlayer}
\end{dynkinDiagram}, 

\begin{dynkinDiagram}[name=upper, edgelength=.75cm] B{ot}
\node (current) at ($(upper root 1)+(0,-.75cm)$) {};
\dynkin[at=(current),name=lower, edgelength=.75cm] B{ot}
\begin{pgfonlayer}{Dynkin behind}
\foreach \i in {1,...,2}
{
\draw[latex-latex]
($(upper root \i)$)
-- ($(lower root \i)$);
}
\end{pgfonlayer}
\end{dynkinDiagram}
\) \\
\( CIIb \) & \( - \) & \( \dynkin[edgelength = .75cm] C{*t*o}, \dynkin[edgelength=.75cm] C{*o*t} \) \\
\hline
\( G\) & \( G_2 \) & \( \dynkin G{ot}, \dynkin G{to} \) \\
\( G_2 \) & \( - \) & \( 
\begin{dynkinDiagram}[name=upper, edgelength=.75cm] G{ot}
\node (current) at ($(upper root 1)+(0,-.75cm)$) {};
\dynkin[at=(current),name=lower, edgelength=.75cm] G{ot}
\begin{pgfonlayer}{Dynkin behind}
\foreach \i in {1,...,2}
{
\draw[latex-latex]
($(upper root \i)$)
-- ($(lower root \i)$);
}
\end{pgfonlayer}
\end{dynkinDiagram}
\),  \begin{dynkinDiagram}[name=upper, edgelength=.75cm] G{to}
\node (current) at ($(upper root 1)+(0,-.75cm)$) {};
\dynkin[at=(current),name=lower, edgelength=.75cm] G{to}
\begin{pgfonlayer}{Dynkin behind}
\foreach \i in {1,...,2}
{
\draw[latex-latex]
($(upper root \i)$)
-- ($(lower root \i)$);
}
\end{pgfonlayer}
\end{dynkinDiagram} \\
\hline
\end{tabular}}
\caption{Graph data of of \( C_0 \backslash \set{0} \). Here \( I_i = I \cup \set{\text{crossed nodes}} \) where \( I \) is the set of black nodes in the Satake diagram, and \( J_{\Fcal_i} \) is one (pair) of the non-crossed white nodes.}
\label{table_singularities_tc}
\end{table}

\section{The horospherical Calabi ansatz} \label{section_horospherical_calabi_ansatz}

\subsection{Horospherical cones}
We recall briefly the setting in \cite{Ngh22b}. The colored cone \( (\Ccal, \Dcal) \) of a horospherical cone in the orbit-cone correspondence consists of a cone \( \Ccal \) of maximal dimension, and a set \( \Dcal \) containing all the colors of the open \( G\)-orbit. Let \( Q \) be the left-stabilizer of the open Borel-orbit with Levi decomposition 
\[ Q = Q^u L \] 
and \( \wh{R}_{Q^u} = \wh{R}_Q \backslash \wh{R}_L \) be the set of ``unipotent'' roots in \( Q \). The Duistermaat-Heckman polynomial of the cone is defined as
\[ P_{DH} (p) = \prod_{\wh{\alpha} \in \wh{R}_{Q^u}} \sprod{ \wh{\alpha} , p}. \] 

\begin{thm}[\!\!\cite{Ngh22b}]
Let \( Y \) be an \( n\)-dimensional \( \Qbb\)-Gorenstein \( G\)-horospherical cone of rank \( k\) and \( \xi \) a Reeb vector polarizing \( C \). Then the following are equivalent
\begin{itemize}
    \item \( Y \) admits a \( K\)-invariant conical Calabi-Yau metric compatible with \( \xi \). 
    \item \( (Y,\xi) \) is K-semistable. 
    \item \( (Y, \xi) \) is K-stable. 
    \item The real Monge-Ampère equation
    \[ \det(d^2 v) P_{DH}(dv) = e^{\varpi}, \quad \ol{ \del v(\Rbb^r)} = \Ccal^{\vee} \]
    admits a positive strictly convex smooth solution \( v \) such that \( v(x-t \xi) = e^{2t} v(x) \). Here \( \varpi \) is the canonical linear function on \( \Ccal \) determined by the \( \Qbb\)-Gorenstein condition. 
\end{itemize}
Furthermore, a \( \Qbb\)-Gorenstein \( G\)-horospherical cone is always K-stable, i.e. there is always the choice of a Reeb vector such that 
\[ \emph{bar}_{DH}(\Delta_{\xi}) = \varpi, \]
where \( \Delta_{\xi} \) is the polytope \( \Ccal \cap \set{\sprod{\varpi, \xi} = n} \), and \( \emph{bar}_{DH} \) is the barycenter of \( \Delta_{\xi} \) with respect to \( P_{DH} \). 
\end{thm}

Horospherical cones appear as a model at infinity of a symmetric space \( G/H \) when we degenerate \( G/H \) along a direction in the interior of the Weyl chamber. Let \( C_0 \) be such a cone. We can express the data of \( C_0 \) in terms of the combinatorial data on the symmetric spaces, see \cite[Subsection 4.1.2]{Ngh22b} for more detail. For example, the Reeb cone \( \Ccal_R \) of \( C_0 \) can be identified with the cone generated by the restricted simple roots of \( G/H \) (in particular \( \Ccal_R^{\vee} \) is the Weyl chamber). 

In rank two, let \( \xi \in \Ccal_R \) be a Reeb vector and \( \delta = \alpha_2 - t \alpha_1, t > 0 \) such that \( \sprod{\delta, \xi} = 0 \).  Note that \( P_{DH}|_{\afrak} = P \) (up to a positive constant factor), since \( \wh{R}_{Q^u} = \wh{R}^{+} \backslash \wh{R}^{\theta} \) and that \( (\wh{\alpha} - \theta(\wh{\alpha}))|_{\afrak} = 2 \alpha \) for all \( \wh{\alpha} \in \wh{R} \).

\begin{thm}[\!\!\cite{Ngh22b}] \label{theorem_kstable_horospherical_cone}
\begin{itemize}
\item The horospherical cone \( C_0 \) is \( \Qbb\)-Gorenstein, hence has klt singularities by \cite{Pas17}.
\item If \( G/H \) is a rank two symmetric space, the Reeb vector \( \xi \) defines a conical Calabi-Yau metric on \( C_0 \) if and only if \( t \) is a positive solution of 
\[ \int_{\lambda_{-}}^{\lambda_{+}} p P(\varpi + p \delta ) dp = 0, \] 
where \( P \) is the Duistermaat-Heckman polynomial of the symmetric space, and
\[ \lambda_{-} = -\frac{\sprod{\alpha_2, \varpi}}{\sprod{\alpha_2, \delta}} < 0, \quad \lambda_{+} = -\frac{\sprod{\alpha_1, \varpi}}{\sprod{\alpha_1, \delta}} > 0 \]
are the intersections with \( \wt{\alpha}_1, \wt{\alpha}_2 \) of the line passing through \( \varpi \) and parallel to \( \delta \). 
\item The K-stable Reeb vector of a horospherical degeneration arising from non-\( G_2 \) symmetric spaces of rank two lies in the interior of the restricted Weyl chamber, and outside the Weyl chamber while being irregular for the \( G_2 \)-symmetric spaces. 
\end{itemize} 
\end{thm}

\begin{proof}
\textcolor{blue}{The first two points are already contained in \cite{Ngh22b}}. The last point was proved in \cite{Ngh22b} for \textcolor{blue}{all the finite families of non-\( G_2 \)} indecomposable symmetric spaces, while for the infinite families it follows from a continuity argument. Indeed, the function \( \xi \to \text{bar}_{DH}(\Delta_{\xi}) \) is continuous and its value lies on different positive sides of \( \varpi \) when \( \xi = \wt{\alpha}_1, \wt{\alpha}_2 \) by \cite{BD19}, hence there is a \( \xi_0 \) in the interior of the Weyl chamber such that \( \text{bar}_{DH}(\Delta_{\xi_0}) = \varpi \). 

The setup is almost the same for the decomposable cases and \( G_2\)-symmetric spaces. A quick way to show that the vector is always rational and lies in the interior of the Weyl chamber in the \( R_1 \times R_1 \)-cases is to use uniqueness of the volume minimizer in \cite{Ngh22b}. In fact, the K-stable Reeb vector is \textit{exactly} the product of two K-stable Reeb vectors on each factor of the root system. Indeed, recall that the normalized volume formula in \cite{Ngh22b} is 
\[ \vol_{DH}(\xi) = \int_{\Ccal_R^{\vee}} e^{- \sprod{p,\xi} } P_{DH}(p) d \lambda(p). \]
From this formula and the fact that the polynomial \( P_{DH} \) on \( R_1 \times R_1 \) is the product of \( P_{DH_i} \) on each factor of the root system, it follows easily that if \( \xi = \xi_1 + \xi_2 \) where \( \xi_i \) is the Reeb vector of the factor \( i \in \set{1,2} \), then \( \vol_{DH}(\xi) = \vol_{DH_1} (\xi_1) \vol_{DH_2} (\xi_2) \). Hence \( \xi \) minimizes \( \vol_{DH} \)  iff \( \xi_i \) minimizes \( \vol_{DH_i} \), i.e. \( \xi_i = (m_i + \wh{m}_i +1)/ (m_i + 2 \wh{m}_i) \omega_{i}^{\vee} \) by the proof of Theorem \ref{theorem_rank_one_cyss}. 

Let us now show that the K-stable Reeb vector lies outside of the positive Weyl chamber in the \( G_2 \)-cases. The positive roots are \( \alpha_1, \alpha_2, \alpha_1 + \alpha_2,  2 \alpha_1 + \alpha_2, 3 \alpha_1 + 2 \alpha_2, 3 \alpha_1 + \alpha_2 \) with multiplicities all equal to \( m = 1 \) or \( m = 2 \) and \( \sprod{\alpha_1, \alpha_1} = 1 \), \( \sprod{\alpha_2, \alpha_2} = 3\), \( \sprod{\alpha_1, \alpha_2} = -3/2\) . The data of the problem  is then \( \varpi = 10 m \alpha_1 + 6m \alpha_2 \),
\[ \lambda_{+} = \frac{2m}{2t+3}, \quad \lambda_{-} = - \frac{2m}{t+2},\] 
and 
\begin{align*} 
P(\varpi + p \delta) = &(2m - (2t+3)p)^m (6m + (3t+6)p)^m(8m + (t+3)p)^m \\
&(10m-tp)^m (12m - (3t+3)p)^m (18+3p)^m. 
\end{align*}
For \( m = 1 \) and \( m = 2 \), the polynomial equations obtained from the K-stability condition are respectively
\[ 2376 + 9225 t + 13407 t^2 + 9357 t^3 + 3179 t^4 + 424 t^5 = 0,\]
and 
\begin{equation*}
\begin{aligned}
&20558772 + 134444448 t + 374274594 t^2 + 590688162 t^3 + 587394519 t^4 +\\
& 383740299 t^5 + 165293858 t^6 + 45384306 t^7 + 7221048 t^8 + 507988 t^9 = 0. 
\end{aligned}
\end{equation*}
It is obvious that both equations have no positive solution since all the coefficients are positive. Finally, one can show that there is no rational solution in both cases by using the rational root theorem. 
\end{proof}

\subsection{The ansatz and its Ricci potential}
We now let \( \beta \) be a direction \textit{in the interior} of the Weyl chamber. Let 
\[ \delta := \alpha_2 - t \alpha_1 \] 
be such that $\left<\beta, \delta \right> = 0$ and $\left< \delta, \alpha_1 \right> < 0$. Define
\[b = \frac{\sprod{\varpi, \beta}}{n \sprod{\beta, \beta}}, \quad d = \frac{\sprod{\varpi, \delta}}{n \sprod{\delta, \delta}}  \]
where \( n \) is the dimension of the symmetric space.  Recall that
\[ P(.) = \prod_{\alpha \in R^{+}} \sprod{\alpha, .}^{m_{\alpha}}. \] 

\begin{ansatz} \label{ansatz_calabi}
Consider the Calabi ansatz
$$\rho^{(0)} = \exp(\phi), \; \phi = b \beta + \psi(\delta).$$
Here \( \psi \) is a smooth strictly convex function defined in such a way that the function \( u : \Rbb \to \Rbb \)
\begin{equation*}
u(x) = n \psi(x) - \frac{\left<\varpi, \delta \right>}{\left< \delta, \delta \right>} x + \log(b^2 2^{n-2} n^{1-n})  = n \psi(x)  - n d x +  \log(b^2 2^{n-2} n^{1-n})
\end{equation*}
satisfies:
\begin{equation}
 u'' > 0, \quad
u'(\Rbb) = ]\lambda_{-}, \lambda_{+}[, \quad 
u''(x) P(\varpi + u'(x) \delta) = e^{-u(x)}. 
\end{equation}
\end{ansatz}

\begin{rmk} \label{remark_calabiansatz} \hfill
\begin{itemize} 
\item Note that \( n \phi = u + \varpi \) up to a constant. The choice of \( \rho^{(0)} \) is justified by the solution of the conical Monge-Ampère equation in \emph{\cite[Proposition 3.17]{Ngh22b}}, where we make the change of variable \( v = e^{u + \varpi}  \) to reduce the equation to an analog of the Kähler-Einstein problem on Fano horospherical varieties. The choice of $u$ is unique up to an action by translation in the variable \emph{\cite[Proposition 3.19]{Ngh22b}}. 

\item By Remark \ref{remark_cartanalgebra_coincidence}, we can view \( \rho^{(0)} \) as the pullback on \( G/H \) of the conical Calabi-Yau potential over \( G/H_0 \times \Cbb^{*} = N_{D_0/ \wt{X}}^{\text{reg}} \) (which always exists by \emph{\cite{Ngh22b}}), and \(\psi \) as the (pullback of a) singular Kähler-Einstein potential of a Fano variety when the tangent cone is regular, or a transverse Kähler-Einstein potential in the irregular case. The potential \( \rho^{(0)} \) is only smooth in the interior of the Weyl chamber, hence not a globally asymptotic solution.
\end{itemize}
\end{rmk}

\begin{lem} \label{lemma_asymptotic_solution_generic_region}
The potential $\rho^{(0)}$ is an asymptotic solution to the Ricci-flat equation \eqref{equation_ricci_flat_symmetric_space} on the generic region. Moreover its Ricci-potential (Definition \ref{definition_ricci_potential}) has faster than quadratic decay in this region. 
\end{lem}

\begin{proof}
We have 
$$d \rho^{(0)} = \rho^{(0)} d \phi = \rho^{(0)} (b\beta + \psi'(\delta) \delta).$$
This implies: 
\begin{equation*}
\begin{aligned}
d^2 \rho^{(0)} = \rho^{(0)}(d^2 \phi + d \phi \otimes d \phi ) = \rho^{(0)} (( &\psi''(\delta) + \psi'(\delta)^2) \delta \otimes \delta + b^2 \beta \otimes \beta \\
&+  b \psi'(\delta) (\beta \otimes \delta + \delta \otimes \beta) ), 
\end{aligned} 
\end{equation*}
hence by using $n = 2 + \sum_{\alpha \in R^{+}} m_{\alpha}$, and $\varpi = \sum_{\alpha \in R^{+}} m_{\alpha} \alpha = nb \beta + n d \delta$, we obtain
\begin{align*}
\det(d^2 \rho^{(0)}) \prod_{\alpha \in \Phi^{+}} \left< \alpha, d\rho^{(0)} \right>^{m_{\alpha}} &= (\rho^{(0)})^{n} b^2 \psi''(\delta) \prod_{\alpha \in R^{+}} \left< \alpha, b\beta +\psi'(\delta) \delta \right>^{m_{\alpha}} \\
&= b^2 \exp(n b \beta - \log(b^2 2^{n-2}) + n d  \delta) \\
&= \prod_{\alpha \in R^{+}} \left( \frac{e^{\alpha}}{2}\right)^{m_{\alpha}} = \prod_{\alpha \in \Phi^{+}} \sinh(\alpha)^{m_{\alpha}} \prod_{\alpha \in R^{+}} \left( 1 - e^{-2 \alpha} \right)^{-m_{\alpha}}. 
\end{align*}
Recall that the generic region is defined to be \( \beta \sim \alpha_1 \gg 1, \beta \sim \alpha_2 \gg 1 \), that is an open subset inside of the Weyl chamber where \( \alpha_1, \alpha_2 \) are unbounded (identified with an open subset of \( G/H \) via the exponential map and $K$-action). From the computation, it follows that \( \rho^{(0)} \) is an asymptotic solution to the Ricci-flat equation in the generic region. Indeed, consider its Ricci potential
\begin{equation} 
\mathcal{P}(\rho^{(0)} ) = \ln \det(d^2 \rho^{(0)} ) + \sum_{\alpha \in R_{+}} m_{\alpha} (\ln \sprod{\alpha, d \rho^{(0)} } - \ln \sinh \alpha ). 
\end{equation}
When \( \alpha_2 \gg \alpha_1 \), resp. \( \alpha_1 \gg \alpha_2 \) (so that \( \beta \sim \alpha_2 \gg \alpha_1 \), resp. \( \beta \sim \alpha_1 \gg \alpha_2 \)), we have
\begin{equation}
\mathcal{P}(\rho^{(0)}) = \ln ( 1 + O(e^{-2\alpha_2})), \quad \text{resp.} \;  \mathcal{P}(\rho^{(0)}) = \ln ( 1 + O(e^{-2\alpha_1})),
\end{equation}
while on the region \( \beta \sim \alpha_2 \sim \alpha_1 \gg 1 \),
\begin{equation}
\mathcal{P}(\rho^{(0)}) = \ln ( 1 + O(e^{-2\beta})).     
\end{equation}
It follows that on the generic region
\[ \mathcal{P}(\rho^{(0)}) = O(e^{-2\beta}), \]
and the same holds for all derivatives. 
Let \( r \) be the radius of the cone metric on \( C_0 \) so that \( r^2 \) restricts to \( \rho^{(0)} = e^{b \beta + \psi(\delta)} \) on the open orbit of \( C_0 \). Since \( \psi(\delta) \) is bounded for \( \alpha_2 \sim \alpha_1 \), we have \( e^{-2 \beta } \sim r^{-4/b} \) outside of a compact subset of \( G/H \), hence on the generic region
\begin{equation} \label{equation_ricci_potential_control_generic_region}
\abs{ \Pcal(\rho^{(0)})} \leq r^{-4/b}. 
\end{equation}
Using Table \ref{table_constants}, one can easily check that \( 4/b  > 2 \), hence \( \rho^{(0)} \) has faster-than-quadratic decay in the generic region. 
\end{proof}

We will later glue this asymptotic solution with two potentials in the direction of \( D_1 \) and \( D_2 \)  in such a manner that the Ricci potential still behaves well in a neighborhood of the infinity direction near \( D_1, D_2 \). To justify the choice of the model near \( D_1, D_2 \), we first need to study the asymptotic behavior of the Calabi ansatz. 

\subsection{Asymptotic behavior of the transverse potential}

The following theorem will be used to understand the asymptotic behavior of the horospherical Calabi ansatz. The reader may compare with \cite[Theorem 4.2]{BD19} where a similar result is obtained for the invariant Kähler-Einstein potential on a Fano horosymmetric variety that lifts to the Calabi ansatz. 

\begin{thm} \label{theorem_asymptotic_behavior_KE_fano_potential}
Consider two real numbers \( \lambda_{-} < 0 < \lambda_{+} \) and a one-variable polynomial \( P_0 \) which is positive on \( ]\lambda_{-}, \lambda_{+}[ \) and vanishes on \( \lambda_{\pm} \) with multiplicities \( m_{\pm} \). 
Let \( u : \Rbb \to \Rbb \) be a strictly convex smooth positive function such that
\begin{equation*}
u'' P_0(u') = e^{-u}, \quad u'(\Rbb) = ]\lambda_{-}, \lambda_{+}[     
\end{equation*}
and \(\delta_{-} < 0 < \delta_{+} \) be the constants defined by
\[ \delta_{\pm} = \frac{\lambda_{\pm}}{m_{\pm} + 1}. \] 
Then there exist two sequences \( (C_j^{\pm})_{j \in \Nbb} \) such that 
\begin{equation}
u(x) = \lambda_{\pm} x + C_0^{\pm} + \sum_{j=1}^m C_j^{\pm} e^{- j \delta_{\pm} x} + o \tuple{ e^{-m \delta_{\pm} x} }   
\end{equation}
when \( x \to \pm \infty \). 
\end{thm}


\begin{proof}
Let \( u^{*}(p) := \sup_{x \in \Rbb} (xp - u(x)) \) be the Legendre transform of \( u \). Since \( u \) is smooth and that \( u'(\Rbb) = ]\lambda_{-}, \lambda_{+}[ \), we have \( \text{dom}(u^{*}) = ]\lambda_{-}, \lambda_{+}[ \). Moreover, \( \lambda_{+} x - u(x) \) (resp. \( \lambda_{-} x -u(x) \)) is a strictly increasing (resp. strictly decreasing) function, hence
\[ \lim_{x \to \pm \infty} ( \lambda_{\pm} x - u(x)) = \lim_{p \to \lambda_{\pm}} u^{*}(p). \]
By \cite[Lemma 3.22]{Ngh22b}, the solution \( u \) actually satisfies \( \norm{u^{*}}_{L^{\infty}(]\lambda_{-}, \lambda_{+}[)} < +\infty \). Therefore
\[ C_0^{\pm} := \lim_{x \to \pm \infty} ( \lambda_{\pm} x - u(x) ) < \infty. \]
We thus obtain the first two terms in the expansion of \( u \) near \( \pm \infty \). Suppose that we have the expansion of \( u \) at order \( N \). The goal is to obtain an expansion at order \( N + 1 \). We have 
\[ e^{-u(x)} = e^{-\lambda_{\pm} x} \tuple{\sum_{j=0}^{N} C^{\pm, (0)}_j e^{-j\delta_{\pm} x } + o \tuple{e^{-N \delta_{\pm} x}} }  \]
where \( C^{\pm,(0)}_0 = e^{C^{\pm}_0}\) and \( C^{\pm,(0)}_j = e^{C_0^{\pm}} C_j^{\pm}\) for \( j \geq 1 \). We only consider the behavior of \( u\) near \( + \infty \) since the reasoning is the same near \( -\infty \). Integrating the equation we obtain
\begin{equation}  
\int_x^{+\infty} u''(t) P_0(u'(t)) dt = e^{-\lambda_{+} x } \tuple{\sum_{j=0}^N C_j^{+,(1)}e^{-j\delta_{+} x} + o \tuple{ e^{-N\delta_{+} x}} } 
\end{equation}
where \( C^{+,(1)}_j = \frac{C_j^{+,(0)}}{\lambda_{+} + j \delta_{+} } \). 
Let \( Q \) be a primitive of \( P_0 \), we then have 
\[ \int_{x}^{+\infty} u''(t)P_0(u'(t))dt = \int_{u'(x)}^{\lambda_{+}} P_0(p)dp = Q(\lambda_{+}) - Q(u'(x)). \]
Since \( (-1)^{m_{+}}P^{m_{+}}_0(\lambda_{+}) > 0 \), the following function
\[ F(w) := \tuple{ \frac{(m_{+}+1)! (Q(\lambda_{+}) - Q(\lambda_{+} - w))}{(-1)^{m_{+}} P^{m_{+}}_0(\lambda_{+})}}^{1/(m_{+}+1)} \] 
is a well-defined analytic function and admits an analytic development of arbitrary order with \( F'(0) > 0 \). In particular \( F \) is invertible and the inverse of \( F \) is also analytic in a neighborhood of \( 0 \).  From the previous expansion, it follows that 
\[ F(\lambda_{+} - u'(x)) = C_{0}^{+,(2)} e^{- \delta_{+} x} \tuple{ 1 + \sum_{j=1}^N C_j^{+,(2)} e^{-j \delta_{+} x} + o(e^{-N \delta_{+} x})}   \]
where \( C_0^{+,(2)} = (C_0^{+,(1)})^{1/(m_{+} + 1)} \) and \( C_{j}^{+,(2)} = \frac{C_j^{+,(1)}}{m_{+} + 1} \). Applying the inverse function of \( F \) onto the above expression yields an asymptotic expression of \( u' \) to the order \( N \), hence of \( u \) to the order \( N+1 \) near \( +\infty \). 
\end{proof}

Since the function \( P_0(.) = P(\varpi + (.)\delta) \) satisfies the assumptions of Theorem \ref{theorem_asymptotic_behavior_KE_fano_potential}, the transverse potential \( \psi(x) \) defined in Ansatz  \ref{ansatz_calabi} admits the asymptotic expansion 
\begin{equation} \label{equation_transverse_potential_expansion}
\psi(x) = \tuple{d + \frac{\lambda_{\pm}}{n}} x + K_0^{\pm} + \sum_{j=1}^m K_j^{\pm} e^{-j \delta_{\pm} x}  + o(e^{-m \delta_{\pm} x}), 
\end{equation} 
where 
\[ K_0^{\pm} = \frac{C_0^{\pm}-  \log(b^2 2^{n-2} n^{1-n}) }{n}, \quad K_j^{\pm} = \frac{C_j^{\pm}}{n}, \; j \geq 1. \]

\section{Construction of the global asymptotic solution} \label{section_global_asymptotic_solution}

Recall that \( \wt{X} = G/H \cup D_0 \cup D_1 \cup D_2 \) is a compactification of \( G/H \) with three irreducible divisors, where \( D_0 \) is horospherical, \( D_1, D_2 \) are horosymmetric with open orbit \( G/H_1, G/H_2 \) fibering over the flag manifolds \( G/P_1, G/P_2 \) with fibers rank one symmetric spaces \( X_1, X_2 \). The singularities \( D_0 \cap D_{1,2} \subset D_0 \) are also fibrations over \( G/P_{1,2} \) with fibers the Kähler-Einstein base of the Stenzel asymptotic cone of \( X_{1,2} \).

Recall also that \( \wt{\alpha}_2 = \alpha_2 + \zeta_2 \alpha_1 \), \( \wt{\alpha}_1 = \alpha_1 + \zeta_1 \alpha_2 \) are wall directions of the Weyl chamber, where \( \zeta_{1,2} \) are chosen so that \( \sprod{\wt{\alpha}_1, \alpha_2} = \sprod{\wt{\alpha}_2, \alpha_1} = 0 \). By an abuse of notation, we use \( (\wt{\alpha}_1, \alpha_2) \) (resp. \( (\wt{\alpha}_2, \alpha_1) \)) as coordinates on the dual Cartan subspace \( \simeq \afrak^{*} \) of \( \Cbb^{*} \times G/H_1 \) (resp. \( \Cbb^{*} \times G/H_2 \)). We also view \( \Cbb^{*} \times G/H_{1,2} \) as open orbits of \( N_{D_{1,2}/ \wt{X}} \), which coincide with the smooth loci.

In what follows, the alpha coordinates are logarithmic in length parameters, so any exponentially small quantity is power law small in the distance. The reason we need exponential asymptote in any order (for instance in Theorem \ref{theorem_asymptotic_behavior_KE_fano_potential}) is that we want the error to be better than any power law decay. The reader should compare this with the toric case where toric symmetry requires working with coordinates logarithmic in length \(  x_i = \log \abs{z_i}^2 \). 

\subsection{Asymptotic of the Calabi ansatz near the boundary divisors}

Let $b_1^{\pm} := ( d + \lambda_{\pm}/n)$, $a_0^{\pm}$ be constants such that 
$$ a_0^{+} \wt{\alpha}_2 = b \beta + b_1^{+} \delta, \quad a_0^{-} \wt{\alpha}_1 = b \beta + b_1^{-} \delta. $$
Recall that
\[ \rho^{(0)} = \exp(b \beta + \psi(\delta)),  \]
where \( \delta = \alpha_2 - t \alpha_1 \). Formally, there exist two sequences \( (K_j^{\pm})_{j \in \Nbb} \) such that 
\begin{align*}
\rho^{(0)} &= \exp( b \beta + b_1^{\pm} \delta + K_0^{\pm} ) \exp \sum K_j^{\pm} e^{-j \delta_{\pm} (\alpha_2 - t \alpha_1)},
\end{align*}
when \( \delta \to \pm \infty \). 

By the asymptotic behavior \eqref{equation_transverse_potential_expansion} of \( \psi \), if \( \alpha_2 \ll \alpha_1 \) and  $\alpha_1 \to +\infty$ (hence \( \delta \to -\infty \)), then
\begin{equation*}
\begin{aligned}
\rho^{(0)} &\sim e^{b \beta + b_1^{-} \delta + K_0^{-}} \exp \sum_{j \geq 1} c_j^{-} e^{-a_j^{-} \alpha_1} \\
&\sim e^{b \beta + b_1^{-} \delta + K_0^{-}} ( 1+ \sum_{j \geq 1} c_j^{-} e^{-a_j^{-} \alpha_1} ). 
\end{aligned} 
\end{equation*}
Here $c_j^{-} := K_j^{-} e^{j \delta_{-} \alpha_2} $ and $a_j^{-} = - tj \delta_{-}$. Similarly, when \( \alpha_1 \ll \alpha_2 \) and $\alpha_2 \to +\infty$ (so that \( \delta \to +\infty \)), 
\begin{equation*}
\rho^{(0)} \sim e^{b \beta + b_1^{+} \delta + K_0^{+}} \exp \sum_{j \geq 1} c_j^{+} e^{-a_j^{+} \alpha_2} 
\end{equation*}
where $c_j^{+} := K_j^{+} e^{- j \delta_{+} \alpha_1}$ and $a_j^{+} = j \delta_{+} > 0$. 

Thus we see that the first order asymptotic behavior of \( \rho^{(0)} \) near \( D_2 \) and \( D_1 \) is of the form
\begin{equation}  \label{equation_asymptotic_behavior}
e^{K^{\pm}_0 + a_0^{\pm} \wt{\alpha}_{2,1} } ( 1 + \sum_{j \geq 1} c_j^{\pm} e^{-a_j^{\pm} \alpha_{2,1}}). 
\end{equation}


\subsection{Construction of the boundary potentials}

The first step is to provide good initial ansätze \( \rho_1^{(1)}, \rho_1^{(2)} \) near \( D_1, D_2 \) based on the asymptotic \eqref{equation_asymptotic_behavior}. Let \( w \) be the Stenzel potential on the symmetric fiber of, say, \( G/H_2 \), which satisfies
\begin{equation} 
A w''(x) (w'(x))^{m_{\alpha_1} + m_{2 \alpha_1}} = \sinh^{m_{\alpha_1}}(x) \sinh^{m_{2\alpha_1}}(x), 
\end{equation} 
where \( A \) is a constant choosen so that the solution of the ODE has an expansion when \( \alpha_1 \to \infty \) of the form 
\begin{equation} \label{equation_asymptotic_expansion_stenzel} 
w(\alpha_1) = K_1^{+} e^{a_1^{+} \zeta_2 \alpha_1} \tuple{1 + \sum_{k \geq 1} w_k e^{-2k \alpha_1} }, 
\end{equation}
where the constant \( K_1^{+} \) is exactly the constant in the asymptotic expansion \eqref{equation_transverse_potential_expansion} of the transverse potential \( \psi \). The following result shows that one can indeed always make such choice of \( A \). 

\begin{prop} \label{prop_stenzel_asymptotic}
Up to an additive constant, the Stenzel potential \( w \) solution to
\begin{equation*}
A w''(x) (w'(x))^{m + \wh{m}} =  \sinh^{m}(x) \sinh^{\wh{m}}(2x),
\end{equation*}
where \( A > 0 \) is arbitrary, has the asymptotic expansion
\begin{equation*}
w(x) = K e^{a x} ( 1 + \sum_{k \geq 1} c_k e^{-2kx}), 
\end{equation*}
where \textcolor{blue}{\( a = \frac{n-1+\wh{m}}{n} \)} and
\begin{equation}
K = \frac{n}{n-1+\wh{m}} \tuple{\frac{n}{A (n-1+\wh{m}) 2^{n-1}}}^{\frac{1}{n}}.    
\end{equation}
\end{prop}

\begin{proof}
Let \( v := w' \). Remark that the equation is equivalent to 
\[ \frac{A}{n} (v(x)^{n})' = \sinh^{m}(x) \sinh^{\wh{m}}(2x). \]
It follows that 
\[ (v(x)^{n})' = \frac{n}{A} \frac{e^{(n-1 + \wh{m})x}}{2^{n-1}} ( 1 - e^{-2x})^{m}(1- e^{-4x})^{\wh{m}}, \]
hence
\[ v(x) = \tuple{\frac{n}{A(n-1 + \wh{m}) 2^{n-1}}}^{\frac{1}{n}} e^{ \frac{n-1 +  \wh{m}}{n} x} ( 1 + \sum_{k \geq 1} c_k e^{-2kx}). \] 
Thus we obtain the result after suitably choosing a constant in the integration. 
\end{proof}

Note that the constant \( \zeta_2 \) appears due to the choice of coordinates on the Cartan algebra. We also have a similar potential near \( D_1 \) after replacing \( K_1^{+}, a_1^{+}, \zeta_2 \) by \( K_1^{-}, a_1^{-}, \zeta_1 \). 

\begin{ansatz} \label{ansatz_boundary} Define \textcolor{blue}{the following warped product of AC metrics}
\begin{equation} \label{equation_stenzel_potential_near_D2}
\begin{aligned}
    \rho_{1}^{(2)} := e^{K_0^{+} + a_0^{+} \wt{\alpha}_2 } \tuple{ 1 + e^{-a_1^{+} \wt{\alpha}_2} w (\alpha_1) }
\end{aligned}
\end{equation}
where \textcolor{blue}{\( e^{a_0^{+} \wt{\alpha}_2}\) is the potential of the metric \( \frac{i}{2} \del \delb \abs{z}^{2a_0^{+}} \) with \( \wt{\alpha}_2 = \log \abs{z}^2 \) on \( \Cbb\)}, and \( w \) is the Stenzel potential on the open orbit \( G/H_2 \) of \( D_2 \).
\end{ansatz}

Let us give a brief explanation on this choice. Suppose that we have some hypothetical global potential \( \rho \) on \( G/H \) that coincides with \( \rho^{(0)} \) in the generic region. We would then like \( \rho \) to behave like \( \rho^{(0)} \) at the first order near \( D_1, D_2 \), i.e. 
\[ \rho \sim e^{K_0^{+} + a_0^{+} \wt{\alpha}_2}( 1+  \sum_{j \geq 1} r_j^{+} e^{-a_j^{+} \alpha_2}) \] 
at the first order near \( D_2 \) (and similarly near \( D_1 \)), cf. \eqref{equation_asymptotic_behavior}. We would also like \( \rho \) to be translated in terms of the coordinates \( (\wt{\alpha}_2, \alpha_1) \) on \( \Cbb^{*} \times G/H_2 \). This suggests we rewrite each term \( r_j^{+} e^{-a_j^{+} \alpha_2} \) as \( r_j^{+} e^{-a_j^{+} \wt{\alpha}_2} e^{a_1^{+} \zeta_2 \alpha_1} \), but the term \( r_j^{+} e^{a_1^{+} \zeta_2 \alpha_1} \) can be identified with the radius function of the Stenzel metric in \eqref{equation_asymptotic_expansion_stenzel}. This somehow justifies the choice of Ansatz \ref{ansatz_boundary}. 

It turns out that the Ricci potential of this initial ansatz might not have the decay rate as we want, so the next step is to improve the decay rate of the Ricci-potential by adding on new terms to kill off the badly-decaying terms in \( \rho_1^{(2)} \). 

To be precise, as \( \alpha_2 \to +\infty \) and \( \alpha_1 \) stays bounded,
\begin{equation*}
    \rho_{1}^{(2)} = e^{K_0^{+} + a_0^{+} \wt{\alpha}_2} \tuple{1 + e^{-a_1^{+} \alpha_2}(r_1^{+} + O(e^{-2 \alpha_1}))}, \; r_1^{+} := K_1^{+}. 
\end{equation*}
The idea of Proposition \ref{proposition_boundary_potential} is to successively add terms of the form 
\[ e^{-a_k^{+} \wt{\alpha}_2}R_k(\alpha_1) = e^{-a_k^{+} \alpha_2}(r_k^{+} + O(e^{-2\alpha_1})) \]
such that \( R_k \) mimics the Stenzel behavior at each order but with improved Ricci-potential decay rate, and that the final first order term will have the form 
\[ e^{K_0^{+} + a_0^{+} \wt{\alpha}_2}( 1 + \sum_{k} r_k^{+} e^{-a_k^{+} \alpha_2}), \]
which is built to be the same as the first order asymptotic behavior \eqref{equation_asymptotic_behavior} of \( \rho^{(0)} \) near \( D_2 \), suggesting a consistency in some gluing region as we will see.  

\begin{prop} \label{proposition_boundary_potential}
There is a potential \( \rho^{(2)} \) on \( \Cbb^{*} \times G/H_2 \) with asymptotic expansion as \( \alpha_2 \to \infty \)
\[ \rho^{(2)} \sim e^{K_0^{+}} e^{a_0^{+} \wt{\alpha}_2} \tuple{1 + \sum_k e^{-a_k^{+} \wt{\alpha}_2} R_k(\alpha_1)} \]
where
\begin{itemize}
    \item \( R_1 = w \) and for all \( k \), \( R_k \) is an even function of \( \alpha_1 \) such that when \( \alpha_1 \to \infty \), \( R_k \) has Stenzel-like behavior, that is, 
    \[ R_k(\alpha_1) = e^{a_k^{+} \zeta_2 \alpha_1}( r_k^{+} + O(e^{-2\alpha_1})), \]
    where \( r_k^{+} > 0 \) and \( O(e^{-2\alpha_1}) \) is a function whose all derivatives behave like \( O(e^{-2\alpha_1}) \); 
    \item \( 0 < a_1^{+} < a_2^{+}  < \dots \) and \( a_i^{+} \in a_1^{+} \Nbb + 2 \Nbb, \forall i \geq 2 \);
    \item For every \( k \geq 1 \), the Ricci potential \( \Pcal \) (cf. \eqref{equation_ricci_potential}) of the term \( \rho^{(2)}_k \) up to the \( k\)-th order satisfies
    \[ \abs{ \nabla^l \Pcal(\rho_k^{(2)})} \leq C_{k,l} e^{-a_k^{+} \alpha_2}. \]
\end{itemize}
There is also an potential \( \rho^{(1)} \) on \( \Cbb^{*} \times G/H_1 \) with similar properties. 
\end{prop}

The proof of Proposition \ref{proposition_boundary_potential} can be carried out in the same manner as \cite[Proposition 5.1]{BD19}. We first need the following computational lemma, which can also be proved in the same way as in \cite[Proposition 5.1]{BD19}. We detail the arguments for the reader's convenience.

\begin{lem} \label{lemma_linearization_expansion}
Let \( L \) be the linearization of \( \Pcal \) at \( \rho_1^{(2)} \), i.e.
\begin{equation} 
L (f) := \text{tr}((d^2 \rho^{(2)}_1)^{-1} d^2 f) + \sum_{\alpha \in R^{+}} m_{\alpha} \frac{\sprod{\alpha, df}}{\sprod{\alpha, d \rho^{(2)}_1}}.
\end{equation} 
To simplify the expression, we drop the \( \pm \)-signs on the constants. Then as \( \alpha_2 \to \infty \), we have asymptotically
\begin{equation} \label{equation_asymptotic_form_linearization}
\begin{aligned}
L(f) = \frac{1}{e^{K_1^{+}}e^{a_0 \wt{\alpha}_2}} (&e^{a_1 \wt{\alpha}_2} \Delta_1 f + a_0^{-2} \del_{\wt{2}}^2 f + (n-1-d_1) a_0^{-1} (\del_{\wt{2}} f)  \\
&+ a_0^{-2} (a_0-a_1)^2 \frac{w(\alpha_1)}{w''(\alpha_1)} \del_1^2 f +  \sum_{\alpha_1 \nmid \alpha}
a_0^{-1} (\del_1 f) \frac{ \sprod{\alpha, \alpha_1}}{\sprod{\alpha, \wt{\alpha}_2}} \\
& - 2 a_0^{-2} (a_0 - a_1) \frac{w'(\alpha_1)}{w''(\alpha_1)} \del_{\wt{2}} \del_1 f \\
&+ O(e^{-a_1 \alpha_2} d^2 f) + O(e^{-a_1 \alpha_2} df) ).
\end{aligned} 
\end{equation}
where \( d_1 := \sum_{\alpha_1 \mid \alpha} \alpha \),
\[ \Delta_1 f := \frac{\del_1^2 f}{w''(\alpha_1)} + d_1  \frac{\del_1 f}{w'(\alpha_1)}, \]
and \( O(e^{-a_1 \alpha_2} d^2 f) \) with \( O(e^{-a_1 \alpha_2} df) \) are terms involving second and first derivatives of \( f \) with coefficients growing like \( O(e^{-a_1 \alpha_2}) \) as \( \alpha_2 \to + \infty \). 
We also have a similar expression when \( \alpha_1 \to \infty \).
\end{lem}

\begin{rmk}
In \emph{\cite{BD19}}, two terms involving the first and second derivatives of \( f \) were missing in the asymptotic formula for \( L \) (cf. Equation (32) in \textit{loc. cit.}), but this doesn't alter the decay improvement, since the effect of the missing terms is asymptotically insignificant in the given direction. 
\end{rmk}

\begin{proof}
To make the proof more readable, we remove the \( + \) sign in the constants. By definition of \( \rho_1^{(2)} \), we have 
\begin{equation}  
d \rho_1^{(2)} = e^{K_1 + a_0 \wt{\alpha}_2} \tuple{ (a_0 + (a_0 - a_1) w(\alpha_1) e^{-a_1 \wt{\alpha}_2} ) \wt{\alpha}_2 + w'(\alpha_1) e^{-a_1 \wt{\alpha}_2} \alpha_1 }.
\end{equation}
It follows that 
\begin{equation} 
\begin{aligned} d^2 \rho_1^{(2)} = e^{K_1 + a_0 \wt{\alpha}_2} ( &(a_0^2 + (a_0-a_1)^2 w(\alpha_1) e^{-a_1 \wt{\alpha}_2} ) \wt{\alpha}_2 \otimes \wt{\alpha}_2 \\
&+ (a_0-a_1) w'(\alpha_1)e^{(a_0-a_1) \wt{\alpha}_2}( \wt{\alpha}_2 \otimes \alpha_1 + \alpha_1 \otimes \wt{\alpha}_2) \\
&+ w''(\alpha_1)e^{-a_1 \wt{\alpha}_1} \alpha_1 \otimes \alpha_1). 
\end{aligned}
\end{equation}
A straightforward calculation then yields 
\begin{equation}  \label{equation_determinant_hessian}
\det(d^2 \rho_1^{(2)}) = e^{2 K_1 + (2a_0 - a_1) \wt{\alpha}_2} a_0^2 w''(\alpha_1)\tuple{ 1 + a_0^{-2} (a_0-a_1)^2 e^{-a_1 \wt{\alpha}_2} \tuple{ w(\alpha_1) - \frac{w'(\alpha_1)^2}{w''(\alpha_1)}}}. 
\end{equation}
From the computation of \( d \rho_1^{(2)} \), we obtain 
\begin{equation} \label{equation_scalarproduct}
\sprod{\alpha, d\rho_1^{(2)}} = 
\begin{cases}
e^{K_1} e^{(a_0-a_1) \wt{\alpha}_2} w'(\alpha_1) \sprod{\alpha, \alpha_1}, \; \alpha_1 \mid \alpha, \\
e^{K_1 + a_0 \wt{\alpha}_2} a_0 \sprod{\alpha, \wt{\alpha}_2} (1 + e^{-a_1 \wt{\alpha}_2}(\frac{a_0-a_1}{a_0} w(\alpha_1) + \frac{\sprod{\alpha, \alpha_1}}{a_0 \sprod{\alpha, \wt{\alpha}_2}} w'(\alpha_1))), \; \alpha_1  \nmid \alpha. 
\end{cases}
\end{equation}
Now use \( \wt{\alpha}_2, \alpha_1 \) as the orthogonal basis of the dual Cartan algebra \( \afrak^{*} \), we can write \( df = \del_{\wt{2}} f \wt{\alpha}_2 + \del_1 f \alpha_1 \), where \( \del_1 f = \sprod{df, \alpha_1} / \abs{\alpha}^2 \). Let us first compute the trace term of \( L \). Note that as \( \alpha_2 \to +\infty \),  
\begin{equation} \label{equation_inverse_determinant}
\det(d^2 \rho_1^{(2)})^{-1} = e^{-2K_1 - (2 a_0 - a_1) \wt{\alpha_2} } a_0^{-2} w''(\alpha_1)^{-1} (1 + O(e^{-a_1 \alpha_2})). 
\end{equation} 
Letting \( A \) be the adjoint matrix of \( d^2 \rho_1^{(2)} \), we obtain
\begin{equation} \label{equation_trace_adjoint_times_hessian}
\begin{aligned}  
\tr(A d^2 f) = e^{K_1 + a_0 \wt{\alpha}_2} ( e^{-a_1 \wt{\alpha}_2} w''(\alpha_1) \del^2_{\wt{2}} f  &- 2(a_0 - a_1) e^{-a_1 \wt{\alpha}_2}w'(\alpha_1) \del_{\wt{2}} \del_1 f \\
&+ (a_0^2 + (a_0 - a_1)^2 e^{-a_1 \wt{\alpha}_2} w(\alpha_1)) \del_1^2 f).
\end{aligned}
\end{equation}
Now combining \eqref{equation_inverse_determinant} and \eqref{equation_trace_adjoint_times_hessian} gives us 
\begin{equation} \label{equation_asymptotic_form_linearization_trace}
\begin{aligned} \tr( (d^2 \rho_1^{(2)})^{-1} d^2 f) = e^{-(K_1 + a_0 \wt{\alpha}_2)} (&a_0^{-2} \del_{\wt{2}}^2 f - 2 a_0^{-2} (a_0 - a_1) \frac{w'(\alpha_1)}{w''(\alpha_1)} \del_{\wt{2}} \del_1 f \\
&+ e^{a_1 \wt{\alpha}_2} \frac{\del_1^2 f}{w''(\alpha_1)} + a_0^{-2} (a_0-a_1)^2 w(\alpha_1) \frac{\del_1^2 f}{w''(\alpha_1)} \\
&+ O(e^{-a_1 \alpha_2} d^2 f) + O(e^{-a_1 \alpha_2} df) ). 
\end{aligned}
\end{equation}
On the other hand, for the term involving roots, when \( \alpha_1 \mid \alpha \), we have 
\[ \frac{\sprod{\alpha, df}}{\sprod{\alpha, d \rho_1^{(2)}}} = \frac{\del_1 f}{w'(\alpha_1)} e^{-K_1} e^{-(a_0 - a_1) \wt{\alpha}_2}, \]
while if \( \alpha_1 \nmid \alpha \), 
\[\frac{\sprod{\alpha, df}}{\sprod{\alpha, d \rho_1^{(2)}}} = \frac{(\del_{\wt{2}} f) \sprod{\wt{\alpha}_2, \alpha} + (\del_1 f) \sprod{\alpha, \alpha_1}  }
{e^{K_1 + a_0 \wt{\alpha}_2} a_0 \sprod{\alpha, \wt{\alpha}_2} (1 + e^{-a_1 \wt{\alpha}_2}(\frac{a_0-a_1}{a_0} w(\alpha_1) + \frac{\sprod{\alpha, \alpha_1}}{a_0 \sprod{\alpha, \wt{\alpha}_2}} w'(\alpha_1)))}. \]
Summing up, we get
\begin{equation} \label{equation_asymptotic_form_linearization_root}
\begin{aligned}
\sum_{\alpha \in R^{+}} \frac{\sprod{\alpha, df}}{\sprod{\alpha, d \rho_1^{(2)}}} = e^{-(K_1 + a_0 \wt{\alpha}_2)}(& d_1  \frac{\del_1 f}{w'(\alpha_1)} e^{a_1 \wt{\alpha}_2} + (n-1-d_1) a_0^{-1} (\del_{\wt{2}} f) \\
&+ \sum_{\alpha_1 \nmid \alpha} a_0^{-1} (\del_1 f) \frac{ \sprod{\alpha, \alpha_1}}{\sprod{\alpha, \wt{\alpha}_2}} + O(e^{-a_1 \alpha_2} df) ),
\end{aligned}
\end{equation}
where \( d_1 := \sum_{\alpha_1 \mid \alpha} \alpha \). 
Putting together \eqref{equation_asymptotic_form_linearization_trace} and \eqref{equation_asymptotic_form_linearization_root} gives us the desired asymptotic formula for \( L \). 
\end{proof}

\begin{proof}[Proof of Proposition \ref{proposition_boundary_potential}]
Denote by \( \Pcal(\rho_{1}^{(2)}) \) the Ricci potential of \( \rho_{1}^{(2)} \). Define the algebra 
    \[ \Acal_{\delta} = \set{ \sum_{a_k \geq \delta} e^{-a_k \wt{\alpha}_2} f_k(\alpha_1) }, \]
    where \( f_k \) are some even functions satisfying when \( \alpha_1 \to \infty \) 
    \[f_k(\alpha_1) = e^{a_k^{+} \zeta \alpha_1}(A_k + O(e^{-2\alpha_1})), \]
    and all the derivatives of \( f_k \) have the same expansion. 
   Following \eqref{equation_determinant_hessian}, \eqref{equation_scalarproduct} and the asymptotic of the Stenzel potential \eqref{equation_asymptotic_expansion_stenzel},  there is a formal power series
    \[ \Pcal(\rho_{1}^{(2)}) = \sum_{ a_k^{+} \geq a_1^{+}} e^{-a_k^{+} \wt{\alpha}_2} f_k(\alpha_1) \in \Acal_{a_1^{+}}. \] 
    One would then like to construct a potential \( \rho_2^{(2)} \) whose Ricci potential has a better decay rate, i.e. such that
    \[ \Pcal(\rho_2^{(2)}) \in \Acal_{a_2^{+}}, \]
    where \( a_2 = \inf(2a_1, 2) \). This can be done by killing off the term that decays badly in \( \Pcal(\rho_1^{(2)}) \) with a small perturbation of \( \rho_1^{(2)} \). The procedure is as follows. Let \( L \) be the linearization of \( \Pcal \). By Lemma \ref{lemma_linearization_expansion}, when \( \alpha_2 \to +\infty \) (that is, near \( D_2 \)), the leading order of \( L \) is given by 
    \[ e^{-K_1^{+}} e^{(-a_0^{+} + a_1^{+}) \wt{\alpha}_2} \Delta_1, \]
    where \( \Delta_1 \) is the weighted Laplacian defined by 
    \[ \Delta_1 = \frac{\del_1^2}{w''(\alpha_1)}  + d_1 \frac{\del_1}{w'(\alpha_1)}. \]
    Given a sufficiently small perturbation \( \rho_2^{(2)} \) of \( \rho_1^{(2)} \), one has
    \begin{align*} 
    \Pcal(\rho_2^{(2)}) &= \Pcal(\rho_1^{(2)}) + L(\rho_2^{(2)} - \rho_1^{(2)}) + Q \\
    &= e^{-a_1^{+} \wt{\alpha}_2} g(\alpha_1) + h + L(\rho_2^{(2)} - \rho_1^{(2)}) + Q, 
    \end{align*}
    where \( g \) is an even function satisfying \( g(\alpha_1) = e^{a_1^{+} \zeta_2 \alpha_1}(A + O(e^{-2\alpha_1})) \), \( h \in \Acal_{a_2^{+}} \), and \( Q \) is some non-linear term. In order to obtain \( \Pcal(\rho_2^{(2)}) \in \Acal_{a_2^{+}} \) one then needs a potential \( \rho_2^{(2)} \) such that the term \( L(\rho_2^{(2)} - \rho_1^{(2)}) + e^{-a_1^{+} \wt{\alpha}_2} g \) is in \( \Acal_{a_2^{+}} \). This can be done by first considering the equation
    \[ \Delta_1 f = g. \]
    We view \( f,g \) as one-variable functions in the weighted spaces  
    \[ C^{k}_{\eta} := \cosh(x)^{\eta} C^{k}(\Rbb), \eta > 0, \; \text{and} \; C^{k}_{\eta - \zeta_2 a_1^{+}}. \] 
    By surjectivity of \( \Delta_1 \) \cite{LM85}, viewed as a linear operator
    \[ \Delta_1: C^{k}_{\eta} \to C^{k-2}_{\eta-\zeta_2 a_1^{+}}, \] 
    and the asymptotic behavior of \( g \), the equation always admits a solution of the form \( f(\alpha_1) = e^{2 a_1^{+} \zeta_2 \alpha_1}(B + O(e^{-2\alpha_1}))\) (compare with the Stenzel potential \( w \)). 
    Next, by taking \( \rho_2^{(2)} \) as
    \[ \rho_{2}^{(2)} := \rho_{1}^{(2)} - e^{K_0^{+}} e^{(a_0^{+} - 2 a_1^{+}) \wt{\alpha}_2} f(\alpha_1), \] 
    one obtains by straightforward computation using the asymptotic form \eqref{equation_asymptotic_form_linearization} of \(  L\) 
    \[ L(\rho_2^{(2)} - \rho_1^{(2)}) + e^{-a_1^{+} \wt{\alpha}_2} g(\alpha_1) \in \Acal_{a_2^{+}}, \; \text{as} \; \alpha_2 \to \infty.\]
    This is because applying the non-Laplacian term of \( L \) to \( e^{(a_0^{+} - 2 a_1^{+}) \wt{\alpha}_2} f(\alpha_1) \) yields only function in \( \Acal_{a_2^{+}} \). 
    To sum up, we have killed the badly-decaying term \( g \) in \( \Pcal(\rho_1^{(2)}) \) by solving an explicit Laplacian equation, and obtained: 
    \[ \Pcal(\rho_2^{(2)}) \in \Acal_{a_2^{+}}. \] 
  A reiteration of the construction yields \( \rho_{k+1}^{(2)} \) from \( \rho_{k}^{(2)} \). Finally, we obtain as \( \alpha_2 \gg \alpha_1 \) and \( \alpha_1 \gg 1 \)
    \begin{align*} \rho^{(2)} \sim e^{K_0^{+} + a_0^{+} \wt{\alpha}_2} \tuple{1 + \sum_{j \geq 1} e^{-a_j^{+} \alpha_2}(r_j^{+} + O(e^{-2 \alpha_1})) } 
    \end{align*} 
    with first order terms
    \[ e^{K_0^{+} + a_0^{+} \wt{\alpha}_2} \tuple{1 + \sum_{j \geq 1} r_j^{+} e^{-a_j^{+} \alpha_1}}, \]
    as formal series. This matches with the first order term of \( \rho^{(0)} \).  
The construction of \( \rho^{(1)} \) is identical.
\end{proof}

\subsection{Global potential}
Before gluing the potentials, we need to verify that the top order terms of \( \rho^{(0)} \) and \( \rho^{(2)}, \rho^{(1)} \) coincide near the boundary regions. The potentials \( \rho^{(2)}, \rho^{(1)} \) can thus be interpreted somehow as expansions of \( \rho^{(0)} \) along the divisors. 

By the computation carried out above, we have all the ingredients needed to reperform the arguments as in \cite{BD19}, and obtain:
\begin{lem}
The top order term of \( \rho^{(2)} \) coincides with \( \rho^{(0)} \) near the region \( \beta \sim \alpha_2 \gg 1 \), that is
\begin{equation*}
\exp \sum_{j \geq 1} c_j^{+} e^{-a_j^{+} \alpha_1}  = 1 + \sum_j r_j^{+} e^{-a_j^{+} \alpha_1}.
\end{equation*}
In particular, 
\[ \rho^{(0)} - \rho^{(2)} \sim e^{K_0^{+} + a_0^{+} \wt{\alpha}_2} \sum_{k \geq 1} e^{- a_k^{+} \alpha_2} g_k(\alpha_1), \]
where \( g_k(\alpha_1) = O(e^{-2\alpha_1}) \), i.e. a function such that all derivatives satisfy this asymptotic behavior. 
The same can be said for \( \rho^{(1)} \) in the region \( \beta \sim \alpha_1 \gg 1 \). 
\end{lem}

Next, we paste \( \rho^{(0)} \) with the boundary potentials \( \rho_{k}^{(2)} \) and \( \rho_k^{(1)} \) along the lines \( \alpha_1 - \theta^{-} \alpha_2 = 0 \) and \( \alpha_1 - \theta^{+} \alpha_2 = 0 \), where 
\[ k \in \Nbb,  \quad \theta^{+} > \theta^{-} > 0 \]
are to be chosen so that the Ricci potential of the final metric has the decay rate that we want. 

\begin{figure}
\begin{tikzpicture}
\pgfmathsetmacro\ax{2}
\pgfmathsetmacro\ay{0}
\pgfmathsetmacro\bx{2 * cos(120)}
\pgfmathsetmacro\by{2 * sin(120)}
\pgfmathsetmacro\lax{2*\ax/3 + \bx/3}
\pgfmathsetmacro\lay{2*\ay/3 + \by/3}
\pgfmathsetmacro\lbx{\ax/3 + 2*\bx/3}
\pgfmathsetmacro\lby{\ay/3 + 2*\by/3}

\tikzstyle{couleur_pl}=[circle,draw=black!50,fill=blue!20,thick, inner sep = 0pt, minimum size = 2mm]


\draw[->] (0,0) -- (\ax,\ay) node[below right] {\( {\alpha}_1 \)};
\draw[->] (0,0) -- (\bx, \by) node[above left] {\( {\alpha}_2 \)};
\draw[->] (0,0) -- (\ax + \bx/2, \ay + \by/2) node[below right] {\( \wt{\alpha}_1 \)};
\draw (0,0)--(2,0);
\draw[->, ] (0,0)--(\ax/2 + \bx, \ay/2 + \by) node[below left]{\( \wt{\alpha}_2 \)};
\draw[-, ultra thick] (0,0)--(\ax + 4*\bx/5, \ay + 4*\by/5) node[above right] {\( \alpha_1 - \theta^{+} \alpha_2 = 0 \)};
\draw[-, ultra thick] (0,0)--(4*\ax/5 + \bx, 4*\ay/5 + \by) node[above] {\(\alpha_1 - \theta^{-} \alpha_2 = 0 \)};
\end{tikzpicture}
\caption{The pasting lines \( \alpha_1 - \theta^{\pm} \alpha_2 = 0 \). We impose the condition \( \theta^{+} > \theta^{-} > 0 \) so that the lines are distinct, and that they stay in the interior of the Weyl chamber. The larger (resp. smaller) is the \( \eta\) in \( \alpha_1 - \eta \alpha_2 \), the closer the line gets to the wall \( \wt{\alpha}_1 \) (resp. \( \wt{\alpha}_2 \)) . }
\end{figure}

Let \(\gamma: \Rbb \to \Rbb \) be a smooth nondecreasing cutoff function with bounded derivatives such that \( \gamma \) is increasing on \( ]0,1[ \), \( \gamma(t) = 0 \) if \( t \leq 0 \) and \( \gamma(t) = 1 \) if \( t \geq 1 \) and define
\begin{equation} \label{equation_potential_outside_compact_subset}
\rho = (1 - \gamma(\alpha_1 - \theta^{+} \alpha_2) - \gamma(\theta^{-} \alpha_2 - \alpha_1)) \rho^{(0)} + \gamma(\alpha_1 - \theta^{+} \alpha_2) \rho_k^{(1)} + \gamma(\theta^{-} \alpha_2 - \alpha_1) \rho_k^{(2)}. 
\end{equation}
Clearly, 
\begin{equation*}
\rho = \begin{cases}
&\rho^{(0)}, \; \theta^{-} \alpha_2 \leq \alpha_1 \leq \theta^{+} \alpha_2, \\
&\rho_k^{(1)}, \; \alpha_1 - \theta^{+} \alpha_2 \geq 1, \\
&\rho_k^{(2)}, \; \alpha_1 - \theta^{-} \alpha_2 \leq -1. 
\end{cases}    
\end{equation*}
In particular, on the region \( 0 \leq \alpha_1 - \theta^{+} \alpha_2 \leq 1 \), we have
\[ \rho = (1 - \gamma(\alpha_1 - \theta^{+} \alpha_2)) \rho^{(0)} + \gamma(\alpha_1 - \theta^{+} \alpha_2) \rho_k^{(1)}. \]
Rewriting 
\[\rho - \rho_k^{(1)} = ( 1 - \gamma(\alpha_1 - \theta^{+} \alpha_2)(\rho^{(0)} - \rho_k^{(1)})\]
on this region, we obtain from the asymptotic expression of \( L \)
\[ L(\rho - \rho_k^{(1)}) = O(e^{-2 \alpha_2} + e^{(a_1^{-} - a_{k+1}^{-}) \alpha_1}). \] 
Similarly, when \( -1 \leq \alpha_1 - \theta^{-} \alpha_2 \leq 0 \), we have
\[ \rho = ( 1 - \gamma(\theta^{-} \alpha_2 -  \alpha_1)) \rho^{(0)} + \gamma( \theta^{-} \alpha_2 - \alpha_1) \rho_k^{(2)}, \] 
and also
\[ L(\rho - \rho_k^{(2)}) = O(e^{-2 \alpha_1} + e^{(a_1^{+} - a_{k+1}^{+}) \alpha_2}). \] 
Since \( \Pcal(\rho) \sim \Pcal(\rho_k^{(i)}) + L(\rho - \rho_k^{(i)}) \), the asymptotic of \( \Pcal(\rho_k^{(i)}) \) in Proposition \ref{proposition_boundary_potential} and of \( L \) together yield
\begin{equation*}
\Pcal(\rho) = \begin{cases}
&O(e^{-2\alpha_2} + e^{(a_1^{-} - a_{k+1}^{-}) \alpha_1}),  0 \leq \alpha_1 - \theta^{+} \alpha_2 \leq 1,\\ 
&O(e^{-2\alpha_1} + e^{(a_1^{+} - a_{k+1}^{+}) \alpha_2}), -1 \leq \alpha_1 - \theta^{-} \alpha_2 \leq 0.
\end{cases}
\end{equation*}

Let us justify that our choice of \( \theta^{\pm} \) in Proposition \ref{proposition_ricci_potential_control} is a priori coherent.

\begin{lem} \label{lemma_pasting_coherence}
For all symmetric spaces of rank two, we have 
\[ \tuple{\frac{2}{a_0^{-}} - \zeta_1} \tuple{ \frac{2}{a_0^{+}} - \zeta_2} > 1. \]
In particular, we can always choose \( \theta^{+} > \theta^{-} \) such that 
 \[ \frac{2}{a_0^{-}} - \zeta_1 > \theta^{+} > \theta^{-} > \frac{1}{\frac{2}{a_0^{+}} - \zeta_2 }. \]
\end{lem}

\begin{proof}
Table \ref{table_constants} summarizes all the constants needed for the computation. 
In the \( R_1 \times R_1 \) case,  we have \( \zeta_1 = \zeta_2 = 0 \), hence the condition is equivalent to 
\[ a_0^{-} a_0^{+} < 4, \]
which is satisfied since \( a_0^{-} = (m_1 + 2 \wh{m}_1) / (m_1 + \wh{m}_1 + 1) < 2 \), and the same holds for \( a_0^{+} \). 

In the \( A_2 \) case, \( \zeta_1 = \zeta_2 = 1/2 \), so it is enough to verify that
\[ \frac{2}{a_0^{-}} - \frac{1}{2} > 1, \; \frac{2}{a_0^{+}} - \frac{1}{2} > 1, \] 
which translate to \( a_0^{-} < 4/3 \) and \( a_1^{+} < 4/3 \). Both conditions are equivalent to \( 8m/3n < 4/3 \), which is true for all \( m \) and \( n\) in this case. 

Let us now verify the condition in the \( BC_2 / B_2 \) cases, where \( \zeta_1 = 1/2 \) and \( \zeta_2 = 1 \). Again, it is enough to check that \( 2/a_0^{-} - 1/2 > 1 \) and \( 2/a_0^{+} - 1 > 1 \), which translate to
\( a_0^{-} < 4/3 \) and \( a_1^{+} < 1 \), that is 
\[ \frac{4m_2}{n} \frac{1+t}{2+t} + \frac{2m_1 + 4m_3}{n} < 4/3, \quad \frac{2m_2}{n} + \frac{m_1 + 2m_3}{n} \frac{2+t}{1+t} < 1, \]
or equivalently,
\[ t > \frac{-2m_1 - 2m_2 + 4m_3 - 8}{5m_1 + 2m_2 + 2 m_3 + 8}, \quad t > \frac{2m_3 - 2}{m_1 + 2}. \] 
It is easy to check that the first condition is trivial in all cases, and the second condition is satisfied for \( m_3 = 0 \) (i.e. for all the \( B_2 \) cases) and for all the families of restricted root system \( BC_2 \) with \( m_3 = 1 \).  

It remains to check the condtion 
\[ t > \frac{2m_3 - 2 }{m_1 + 2} \] 
for the infinite family \( (4r-16,4,3) \) of restricted root system \( BC_2 \), i.e. \( t > 4/ (4r - 14) = 2 / (2r-7) \). 
It is enough to verify that \( t > 1 \) for all \( r \geq 5 \) in this family. This will be proved together with the bounded geometry condition in Lemma \ref{lemma_bounded_geometry_infinitefamilies_bc2} since the proof idea is similar. 
\end{proof}

\begin{prop} \label{proposition_ricci_potential_control}
If we take 
 \[ \frac{2}{a_0^{-}} - \zeta_1 > \theta^{+} > \theta^{-} > \frac{1}{\frac{2}{a_0^{+}} - \zeta_2 }, \]
and \( k \) large enough so that 
\[ a_{k+1}^{-} - a_1^{-} > a_0^{-}(1+\zeta_1/ \theta^{+}), \quad a_{k+1}^{+} - a_1^{+} > a_0^{+}(1+\theta^{-} \zeta_2), \] 
then outside of a compact subset, we have for \( \varepsilon > 0 \) small enough and for all \( l \)
\[ \abs{\nabla^l \Pcal(\rho)} \leq C_l e^{-(1+\varepsilon) \phi}, \]
where \( \phi = b \beta + \psi(\delta) \). 
\end{prop}

\begin{proof}
We already have such a control in the generic region (cf. Lemma \ref{lemma_asymptotic_solution_generic_region}), as well as in the two boundary regions where \( \alpha_1 \gg \alpha_2 \sim 1 \), resp. \( \alpha_2 \gg \alpha_1 \sim 1 \), cf. \eqref{equation_ricci_potential_control_generic_region}. We need to show that such a control persists in the regions \( 0 \leq \alpha_1 - \theta^{+} \alpha_2 \leq 1 \) and \( -1 \leq  \alpha_1 - \theta^{-} \alpha_2 \leq 0 \). 

We have when \( \alpha_1 \to + \infty \), \( \phi = a_0^{-} \wt{\alpha}_1 + O(1) \), hence 
    \begin{enumerate} 
    \item If \( \alpha_1 \leq \theta^{+} \alpha_2 \), then \( a_0^{-} \wt{\alpha}_1 \leq  a_0^{-} (\theta^{+} + \zeta_1) \alpha_2 \), so 
    \[ e^{-2\alpha_2} = O(e^{-(1+ \varepsilon) \phi}) \] 
    on this region if \( \theta^{+} < (2/a_0^{-} - \zeta_1) \). Note that we can always make such a choice of \( \theta^{+} \) since \( 2 > a_0^{-} \zeta_1 \) as we can check in Table \ref{table_constants}. 
    \item On the region \( \alpha_1 \geq \theta^{+} \alpha_2 \), we have \( a_0^{-} \wt{\alpha}_1 \leq a_0^{-}(1+ \zeta_1/ \theta^{+}) \alpha_1 \), so \( e^{(a_1^{-} - a_{k+1}^{-}) \alpha_1} = e^{-(1+\varepsilon) \phi} \) on this region if \( a_{k+1}^{-} - a_1^{-} > a_0^{-}(1+\zeta_1/ \theta^{+})\).
    \end{enumerate}
On the other hand, when \( \alpha_2 \to + \infty \), \( \phi = a_0^{+} \wt{\alpha}_2 + O(1) \), hence 
    \begin{enumerate}
    \item On the region \( \alpha_1 \geq \theta^{-} \alpha_2 \), we have \( a_0^{+} \wt{\alpha}_2 \leq a_0^{+}(1/ \theta^{-} + \zeta_2) \alpha_1 \), so 
    \[ e^{-2\alpha_1} = O(e^{-(1+\varepsilon) \phi}) \] 
    on this region if \( \theta^{-} > 1/ (2/a_0^{+} - \zeta_2) \). One needs to verify that this is consistent with our choice of \( \theta^{+} \) and \( \theta^{-} \), i.e. 
    \[ \frac{2}{a_0^{-}} - \zeta_1 > \theta^{+} > \theta^{-} > \frac{1}{\frac{2}{a_0^{+}} - \zeta_2 }, \]
    which turns out to be the case for all symmetric spaces of rank two (cf. Lemma \ref{lemma_pasting_coherence}).
    \item Finally, on the region \( \alpha_1 \leq \theta^{-} \alpha_2 \), we have 
    \[ a_0^{+} \wt{\alpha}_2 \leq a_0^{+}( 1 + \theta^{-} \zeta_2) \alpha_2, \]
    so if \( a_{k+1}^{+} - a_1^{+} > a_0^{+}(1+\theta^{-} \zeta_2) \), we have
    \[ e^{(a_1^{+} - a_{k+1}^{+})\alpha_2} = O(e^{-(1+\varepsilon)\phi}). \]
    \end{enumerate}
This completes our proof. 
\end{proof}

Next, we need to show the following.

\begin{lem} 
There is a potential, still denoted by \( \rho \), which is \( W \)-invariant, strictly convex and smooth outside of a compact subset, that restricts  to the potential \( \rho \) defined by \eqref{equation_potential_outside_compact_subset} on the given positive Weyl chamber. 
\end{lem} 

\begin{proof}
By construction, the potentials \( \rho_k^{(1)}, \rho_k^{(2)} \) are invariant under the reflection through the Weyl walls \( \wt{\alpha}_1, \wt{\alpha}_2 \), respectively, and \( \rho^{(0)} \) is smooth, strictly convex in the interior of the given positive Weyl chamber. Thus extending \( \rho \) to the whole Cartan algebra \( \afrak \) using the Weyl action then results in a potential, still denoted by \( \rho \), which is smooth outside of a compact subset, and restricts to \( \rho \) on the given positive Weyl chamber. 

Since \( \rho^{(0)} \) is strictly convex by definition, we only need to show that outside of a compact subset in the boundary region, \( \rho \) is strictly convex. This can be shown in the same manner as the end of section 5 in \cite{BD19}, since near each boundary region, for instance \( \alpha_2 \gg \alpha_1 \sim 1 \), \( \rho \) has the form \( \chi \rho_k^{(2)} + (1 - \chi) \rho^{(0)} \). 
\end{proof}

From this one can argue as in \cite{BD19} to paste \( \rho \) with a potential in the compact subset where \( \rho \) is not well-behaved (i.e. not smooth or strictly convex), and obtain a global potential (for more detail see the end of Section 5 in \cite{BD19}).

\begin{prop} \label{prop_global_potential}
There is a globally \( W \)-invariant, smooth and strictly convex potential on \( \afrak \), still denoted by \( \rho \), such that
\begin{itemize} 
\item \( \rho = \rho^{(0)} \) in the generic region, 
\item \( \rho \) coincides with \( \rho^{(1)}_k \) and \( \rho^{(2)}_k \) in the boundary regions (for \( k \) large enough so that the Ricci potential \( \Pcal(\rho) \) has the desired decay rate). 
\end{itemize}
\end{prop}

\section{Asymptotic geometry of the asymptotic solution} \label{section_asymptotic_geometry}

In the rest of the paper, we denote by \( g = i\del \delb\rho \) the final metric on \( G/H \) coming from the global \( K\)-invariant potential \( \rho \) in Proposition \ref{prop_global_potential}, constructed by gluing the Tian-Yau potential \( \rho^{(0)} \) with two models near \( D_{1,2} \). The metrics \( \ol{g}_0, g_0 \) denote respectively the Calabi-Yau cone metrics on \( C_0 \) and its restriction to the smooth locus \( G/H_0 \times \Cbb^{*} \) of \( C_0 \), which is also the smooth locus of \( N_{D_0/ \wt{X}} \). We often implicitly identify \( g_0 \) with a smooth metric in the generic region of \( G/H \), still denoted by \( g_0 \). 

Recall that the tangent space of \( G/H \) with an involution $\theta$ can be decomposed as
\begin{equation*}
\begin{aligned} 
\gfrak / \hfrak &= \Cbb \afrak \oplus \bigoplus_{\wh{\alpha} \in \wh{R}^{+} \backslash \wh{R}^{\theta}} \Cbb \mu_{\wh{\alpha}}, \\
&= \Cbb l_1 \oplus \Cbb l_2 \oplus \bigoplus_{\wh{\alpha} \in \wh{R}^{+} \backslash \wh{R}^{\theta}} \Cbb \mu_{\wh{\alpha}}, 
\end{aligned} 
\end{equation*}
where \( \wh{R}^{+} \) (resp. \( \wh{R}^{\theta} \)) is the set of positive roots (resp. positive roots fixed by \( \theta \)), \( \mu_{\wh{\alpha}} = e_{\wh{\alpha}}  + \theta(e_{\wh{\alpha}}) \) with \( (e_{\wh{\alpha}}, e_{-\wh{\alpha}}, h_{\wh{\alpha}})\) being a \( \mathfrak{sl}_2 \)-triple, and \( \afrak \) is the Cartan algebra of \( G/H \) with basis \( (l_1,l_2) \) dual to the basis \( (\alpha_1, \alpha_2) \). We can then locally parametrize \( G/H \) by complex coordinates \( (z_1, z_2, (z_{\wh{\alpha}})) \) with respect to theses bases under the exponential map. Let \( \omega_{a \ol{b}} = \frac{i}{2} dz_a \wedge d \ol{z} _b \), \( \omega_{\wh{\alpha} \ol{\wh{\alpha}}} = \frac{i}{2} dz_{\wh{\alpha}} \wedge d \ol{z}_{\wh{\alpha}} \)  and let \( \wh{\alpha}_r = \wh{\alpha} - \theta(\wh{\alpha}) \) be the restricted root obtained from \( \wh{\alpha} \). 

\begin{prop}[\!\!{\cite{Del20b}}] \label{proposition_metric_local_expression}
The metric \( g\) has the following expression
\[ g = \sum_{a,b \in \set{1,2}} d^2 \rho(l_a,l_b) \omega_{a \ol{b}} + 2 \sum_{\wh{\alpha} \in \wh{R}^{+} \backslash \wh{R}^{\theta}} \coth(\wh{\alpha}) \frac{\sprod{d\rho, \wh{\alpha}_r}}{\abs{\wh{\alpha}}^2} \omega_{\wh{\alpha} \ol{\wh{\alpha}}}. \] 
\end{prop}

From the local expression of the metric, we have

\begin{lem} 
In the generic region, the metric \( g\) can be identified with the cone metric from the Calabi ansatz \( \rho^{(0)} = e^{ \phi} = e^{b \beta + \psi(\delta)} \), which is 
\[ g_{0} = \rho^{(0)} \tuple{ \abs{d \phi}^2 + \psi''(\delta) \abs{\delta}^2 + \sum_{\wh{\alpha} \in \wh{R}^{+} \backslash \wh{R}^{\theta}} \frac{2}{\abs{\wh{\alpha}}^2} \coth(\wh{\alpha}) \sprod{d \phi, \wh{\alpha}_r} \omega_{\wh{\alpha} \ol{\wh{\alpha}}} }. \]
\end{lem} 

\begin{proof} 
This is straightforward from the expression of the metric in local coordinates. Note that since \( d \rho^{(0)} \) belongs to the interior of the cone generated by the fundamental weights, we have \( \sprod{d \rho^{(0)}, \wh{\alpha}_r} = e^{\phi} \sprod{d\phi, \wh{\alpha}_r} > 0 \) for all simple restricted roots, hence \( g_0 \) is a well-defined metric. 
\end{proof}

Let \( g_2 \) be the metric associated to the potential
\[ \rho_{1}^{(2)} = e^{K_0^{+} + a_0^{+} \wt{\alpha}_2 }(1 + e^{-a_1^{+} \wt{\alpha}_2} w(\alpha_1)), \]
and \( g_1 \) the metric associated to \( \rho_1^{(1)} \).
Then same computations as in \cite{BD19} show that

\begin{lem} \label{lemma_highest_order_boundary_term}
As \( \alpha_2 \to + \infty \), the highest order term of \( g_2 \) takes the form
\begin{equation*} 
\begin{aligned}
g_{2}^{\infty} := \rho_{1}^{(2)} ( & (a_0^{+})^2 \abs{\wt{\alpha}_2}^2 + 2 a_0^{+} \sum_{\alpha_1 \nmid \wh{\alpha}_r} \coth(\wh{\alpha}) \frac{\sprod{\alpha_1, \wh{\alpha}_r}}{\abs{\wh{\alpha}}^2} \omega_{\wh{\alpha} \ol{\wh{\alpha}}}  \\
&+ e^{-a_1^{+} \wt{\alpha}_2} (w''(\alpha_1) \abs{\alpha_1}^2 + 2 w'(\alpha_1) \sum_{\alpha_1 \mid \wh{\alpha}_r} \coth(\wh{\alpha}) \frac{\sprod{\alpha_1, \wh{\alpha}_r}}{\abs{\wh{\alpha}}^2} \omega_{\wh{\alpha} \ol{\wh{\alpha}}} ) ),
\end{aligned}
\end{equation*}
and moreover, when \( \alpha_2 \to +\infty \)
\[ \abs{g_2 - g_2^{\infty}}_{g_2^{\infty}} = O(e^{-a_1^{+} \alpha_2}). \] 
A similar behavior holds for the asymptotic \( g_1^{\infty} \) of \( g_1 \). 

\end{lem} 

Viewed as metric on the smooth locus \( \Cbb^{*} \times G/H_2\) of \( N_{D_2/ \wt{X}} \), the component \( (\alpha_1, \wh{\alpha}_{\alpha_1 \mid \wh{\alpha}_r}) \) of \( g_2^{\infty} \) corresponds to the Stenzel metric on the rank one symmetric fiber \( X_2 \) of the open orbit \( G/H_2 \subset D_2 \), while the component \( (\wh{\alpha}_{\alpha_1 \nmid \wh{\alpha}_r}) \) corresponds to a metric on \( G/P_2 \) and the component \( (\wt{\alpha}_2) \) a metric on \( \Cbb^{*} \). 

\begin{lem} \label{lemma_metric_boundary_asymptotic}
As  \( \alpha_1, \alpha_2 \to +\infty \), the metrics \( g_1^{\infty} \) and \( g_2^{\infty} \) satisfy 
\begin{equation}
\abs{g_0 - g_{1}^{\infty}}_{g_1^{\infty}}   = O(e^{-a_1^{-} \alpha_1}), \; \abs{g_0 - g_2^{\infty}}_{g_2^{\infty}} = O(e^{-a_1^{+} \alpha_2}). 
\end{equation}
In other words, as a metric on \( G/H_0 \times \Cbb^{*} = N_{D_0/\wt{X}}^{\text{reg}} \), \( g_0 \) is equivalent to \( g_1^{\infty} \) and \( g_2^{\infty} \) near \( D_0 \cap D_1 \) and \( D_0 \cap D_2 \) (that is, inside any open subset of  \( G/H_0 \times \Cbb^{*} \) which gets closer to \( N_{D_0/\wt{X}}^{\text{sing}} = \ol{D_0 \cap D_{1,2} \times \Cbb^{*}}\) as \( \alpha_{1,2} \) get larger).
\end{lem}

\begin{proof}
It is enough to compare \( g_0 \) and \( g_2^{\infty} \) components by components as \( \alpha_2 \to +\infty \). In the \( \abs{dz_{\wh{\alpha}}}^2 \) component, if \( \alpha_1 \nmid \wh{\alpha}_r \), the quotient \( g_0 / g_2^{\infty} \) reads 
\[ \frac{\rho^{(0)} \sprod{d\phi, \wh{\alpha}_r}}{ a_0^{+} \rho_1^{(2)}\sprod{\wt{\alpha}_1, \wh{\alpha}_r}}. \]
Since \( d\phi = b \beta + \psi'(\delta) \delta \) and \( \psi'(\delta) \) is bounded, the contribution of the term \( \frac{ \sprod{d\phi, \wh{\alpha}_r}}{ a_0^{+} \sprod{\wt{\alpha}_1, \wh{\alpha}_r}} \) is just \( O(1) \). From the ansatz \eqref{equation_stenzel_potential_near_D2} and asymptotic behavior of \( \rho^{(0)} \) near \( D_2 \), we have 
\[ \frac{\rho^{(0)}}{\rho_1^{(2)}} -  1 = O(e^{-a_1^{+} \alpha_2}). \] 
On the other hand, if \( \alpha_1 \mid \wh{\alpha}_r \), the corresponding component of the quotient \( g_0 / g_2^{\infty} \) is
\[  \frac{e^{\phi} \sprod{d\phi, \wh{\alpha}_r}}{ \rho_1^{(2)} w'(\alpha_1) \sprod{\alpha_1, \wh{\alpha}_r}}. \]
Again by the same reasoning, the term \( \frac{\sprod{d\phi, \wh{\alpha}_r}}{ \sprod{\alpha_1, \wh{\alpha}_r}} \) is just \( O(1) \), and since \( w'(\alpha_1) = O(e^{a_1^{+} \zeta_2 \alpha_1}) \), the term \( \rho^{(0)} / (\rho_1^{(2)} w'(\alpha_1)) \) satisfies
\[ \frac{\rho^{(0)}}{\rho_1^{(2)} w'(\alpha_1)} - 1 = O(e^{-a_1^{+} \alpha_2}). \] 
It remains to compare the ``Cartan component'' \( \abs{d \phi}^2 + \psi''(\delta) \abs{\delta}^2 \) of \( g_0 \) with the Cartan component \( (a_0^{+})^2 \abs{\wt{\alpha}}^2 + e^{-a_1^{+} \wt{\alpha}_2} w''(\alpha_1) \abs{\alpha_1}^2 \) of \( g_2^{\infty} \). Remark that 
\[ \phi \sim a_0^{+} \wt{\alpha}_2 + K_0^{+} + K_1^{+} e^{-a_1^{+} \alpha_2} + O(e^{-2 a_1^{+} \alpha_2}), \] 
and moreover in the region \( \alpha_2 \gg \alpha_1 \sim 1\), 
\[ \psi''(\delta) \sim \psi''(\alpha_2) \sim (a_1^{+})^2 e^{-a_1^{+} \alpha_2} \sim e^{-a_1^{+} \wt{\alpha}_2} w''(\alpha_1). \]
One argues similarly on the region \( \alpha_1 \gg \alpha_2 \sim 1 \). 
\end{proof}

\begin{prop} \label{prop_GH_convergence}
The tangent cone at infinity of \( (G/H, g) \) is the horospherical cone \( (C_0, \ol{g}_0) \). 
\end{prop}

\begin{proof} 
Since \( g, g_0 \) are \(  K\)-invariant \( \del \delb \)-exact, and that the Cartan algebra \( \afrak \) coincides with \( \tfrak_{nc} \) (cf. Remark \ref{remark_cartanalgebra_coincidence}), it is enough to view the potentials of \( g, g_0 \) as functions on \( \afrak \) and show that on a large enough open subset of \( \afrak \), their components are polynomially equivalent as functions on \( \afrak\), i.e.  
\begin{equation} \label{equation_asymptotically_conical}
\abs{g - g_0}_{g} = O( \rho ^{-\mu}), \; \mu > 0,
\end{equation} 
where the global potential \( \rho \) of \( g\) has the form \eqref{equation_potential_outside_compact_subset} outside of a compact subset. 
Then \eqref{equation_asymptotically_conical} is obvious on the generic region where \( \rho = \rho^{(0)} \). It is then enough to verify \eqref{equation_asymptotically_conical} near the boundary regions \( \alpha_1 \gg \alpha_2 \sim 1 \) and \( \alpha_2 \gg \alpha_1 \sim 1 \), i.e. near \( D_1 \) and \( D_2 \). \textcolor{blue}{Since} \( g\) behaves exactly like \( g_k^{(1,2)} = i \del \delb \rho_k^{(1,2)} \) on these regions for \( k \) large enough, and that \( \rho \sim  \rho_k^{(1,2)} \sim \rho_1^{(1,2)} \sim e^{a_0^{-,+} \alpha_{1,2}} \) as \( \alpha_{1,2} \to +\infty\) as well as all derivatives, it follows that \( g \) is polynomially close to the metrics \( g_{1,2}^{\infty} \) with respect to \( \rho \) near these regions (with decay rate \( \rho^{-a_1^{\pm} / a_0^{\pm}} \)). 
Notice that by our choice of a nondecreasing cutoff function  with bounded derivatives \eqref{equation_potential_outside_compact_subset}, the comparison does not break down on the region where the cutoff is non-constant.

By Lemma \ref{lemma_metric_boundary_asymptotic}, the metrics \( g_0 \) and \( g_1^{\infty}, g_2^{\infty} \) are polynomially close with respect to \( \rho \) near \( D_0 \cap D_1 \) and \( D_0 \cap D_2 \), i.e. their models on \( G/H \) have the same property near the boundary regions inside \( G/H \). This allows us to conclude.
\end{proof} 

\begin{rmk}
Recall that outside of the fixed point, the tangent cone \( C_0 \) consists of the open \( G\)-orbit \( G/H_0 \times \Cbb^{*}\) and the singularities (which occur in two \( G\)-orbits \( G/P_{1,2}\) of codimension \( \geq 2 \)), cf. Section \ref{section_geometric_prelude}. Heuristically, the generic region GH-converges  to the open orbit of the cone, and the two boundary regions converges to the singularities of the cone by collapsing the rank one fibers of \( D_{1,2} \to G/P_{1,2} \) into points.

\end{rmk}

\begin{prop} \label{prop_ricci_estimate}
For a given integer \( l_0 \), if \( \theta^{\pm} \) are chosen as in Proposition \ref{proposition_ricci_potential_control} and \( k \gg 1  \), then outside of a compact subset, we have for all \( l \leq l_0 \) and \( 1 \gg \varepsilon > 0 \) 
\[ \abs{\nabla^l \Pcal(\rho)}_{g} \leq C_l e^{-(1+ \varepsilon  + l/2)\phi},  \]
where \( \phi = b \beta + \psi(\delta) \) as defined in the Calabi ansatz. Moreover, the Ricci curvature of \( g \) satisfies 
\[ \abs{\nabla^l \Ric(\rho)}_g \leq C_l e^{-(3 + \varepsilon + l/2) \phi}. \] 
\end{prop}

\begin{proof}
We already have a control of the ``pure'' derivatives of \( \Pcal \) (cf. Proposition \ref{proposition_ricci_potential_control}). Let us now control the derivatives with respect to the metric. Recall that if \( A_g \) is the matrix of \( g \) in some coordinates, then 
\[ \nabla_g \Pcal(\rho) = A_g^{-1} \nabla \Pcal(\rho), \]
where \( \nabla \Pcal(\rho) \) is the column vector of the pure derivatives of \( \Pcal(\rho) \). After iteration, the terms of \( \nabla^l_{g} \Pcal(\rho) \) not involving Christoffel symbols only contain   the pure derivatives up to order \( l\) of \( \Pcal(\rho) \), the components of \( g^{-1} \) and their derivatives up to order \( l \).  

In the generic region (where \( g = g_0 \)), if \( r \) is the cone distance of \( g_0 \), the Christoffel symbols \( \Gamma^k_{ij} \) involve the product of first derivatives of \( g_0 \) (of order \( r \sim e^{\phi/2} \)) with the components of \( g_0^{-1} \) (of order \( r^{-2} \sim e^{-\phi} \)) so they are asymptotically irrelevant. 
Each $\nabla_{g_0}$ thus comes with a weight \( e^{-\phi} \times (\text{bounded term}) \), hence in the generic region, we obtain 
\[ \abs{\nabla^l \Pcal(\rho)}_{g} = O(e^{-(l/2)\phi - \beta}).\]  
Near a boundary region, say, \( D_2 \), we have \( \phi \sim a_0^{+} \wt{\alpha}_2 \), hence each derivative with respect to \( g_2^{\infty} \) comes with the weight 
\[ e^{-a_0^{+} \wt{\alpha}_2} e^{a_1^{+} \alpha_2} \times (\text{bounded term}) \sim e^{-\phi + a_1^{+} \alpha_2} \times (\text{bounded term}) \]
(in the \( H_{\alpha_2} \) direction) and \( e^{-\phi} \times (\text{bounded term}) = O(e^{-\phi + a_1^{+} \wt{\alpha}_2}) \) in other directions. Thus given the estimate \( \Pcal(\rho_k^{(2)}) = O(e^{-a_k^{+} \alpha_2}) \) in Proposition \ref{proposition_boundary_potential}, which also holds for all derivatives, we obtain near \( D_2 \)
\[ \abs{\nabla^l \Pcal(\rho)}_{g} = O(e^{(-a_k^{+} + a_1^{+}l /2) \alpha_1 - (l/2) \phi}). \] 
The computation is carried out likewise near the other boundary region \( D_1 \). We can now take \( k \) large enough so that \( (-a_k + a_1^{+} (l/2)) \) is as negative as we want, and proceed as in Proposition \ref{proposition_ricci_potential_control}. 
\end{proof}

\section{Proof of the main theorem}  \label{section_proof_main_theorem}

\subsection{Hein's existence package} \label{subsection_existence_package}
 Recall that a \textit{\( C^{k,\alpha}\)-quasi-atlas} for \( (M,g) \) consists of a given ball \( B=B(0,r) \subset \Cbb^n \)  of radius \( r > 0 \), and local biholomorphisms \( \Phi_x : B = B(0,r) \to M \) at each point \( x \in M \) such that
\begin{enumerate} 
\item \( \Phi_x(0) = x \), and there is \( c > 0 \) satisfying \( \text{inj}(\Phi_x^{*} g) \geq 1/c \), \( 1/c g_{\text{eucl}} \leq \Phi_x^{*} g \leq c g_{\text{eucl}} \), 
\item the \( C^{k,\alpha}\)-Hölder-norm is bounded, i.e. \( \norm{\Phi^{*}g}_{C^{k,\alpha}} \leq c \). 
\end{enumerate}

The following package of technical conditions in Hein's thesis \cite{Hein} allow one to solve the complex Monge-Ampère equation for Ricci-flat metrics.

\begin{itemize}
    \item Bounded geometry condition, i.e. existence of a \( C^{k,\alpha}\) quasi-atlas. This is used in the Laplacian and \( C^2 \)-estimate. 
    \item A sufficient condition for the weighted Sobolev inequality, called SOB(\(p\)). Here \( p \) is the growth rate of the volume form, which is \( p = 2n \) in our case. A complete manifold \( (M,g) \) is called SOB(\(p\)) if 
     \begin{enumerate} 
     \item There is \( x_0 \in M \) and \( c \geq 1 \) such that
    \( B(x_0,s) \backslash B(x_0,t) \) is connected for all \( s \geq t \geq c \),
    \item \( \text{vol}(B(x_0,s)) \leq c s^p \), for all \( s \geq c \),
    \item \( \text{vol}(B(x,(1-c^{-1}) d(x_0,x)) \geq c^{-1} d(x_0,x)^p \), \( \text{Ric}(x) \geq -c d(x_0,x)^{-2}\) if \( d(x_0,x) \geq c \). 
    \end{enumerate}
    The SOB(\( 2n \)) condition is used in the \( L^{\infty}\)-estimate. 
    \item  Existence of a function \( \wt{\rho} \sim d(x_0,.) + 1 \). If furthermore \( p > 2 \), we require that  \( \abs{\nabla \wt{\rho}} + \wt{\rho} \abs{i \del \delb \wt{\rho}} \leq c \). 
\end{itemize}

\begin{prop} \label{proposition_unbounded_curvature}
The Riemannian curvature of \( (G/H,g) \) is bounded if and only if \( a_0^{\pm} \geq a_1^{\pm} \).
\end{prop}

\begin{proof}
By symmetry of the situation, it is enough to check the boundedness near one boundary region, say \( \alpha_2 \gg \alpha_1 \sim 1 \), where \( g\) behaves like \( g_2^{\infty} \) (see Lemma  \ref{lemma_highest_order_boundary_term}, Lemma \ref{lemma_metric_boundary_asymptotic}). We omit the \( \pm \) signs for simplicity.
The metric components of $g_2^{\infty}$ on $\afrak$ are 
\[ g_{22} \abs{\wt{\alpha}_2}^2 + g_{11} \abs{\alpha_1}^2, \quad g_{22} = \rho_1^{(2)} a_0^2, \; g_{11} = \rho_1^{(2)} e^{-a_1 \wt{\alpha}_2} w''(\alpha_1), \]
and the ones on the complement of $\afrak$ (up to positive constants)
\[ g_{\wh{\alpha} \wh{\alpha}, \alpha_1 \nmid \wh{\alpha}} = 2 \rho_1^{(2)} \coth(\wh{\alpha}) , \quad g_{\wh{\alpha} \wh{\alpha}, \alpha_1 \mid \wh{\alpha}} = 2 \rho_1^{(2)}e^{-a_1 \wt{\alpha}_2} w'(\alpha_1) \coth(\wh{\alpha}). \] 
First remark that $\coth$, $w(\alpha_1)$ and all of their derivatives are bounded since $\alpha_2 \gg \alpha_1 \sim 1$. Using Ansatz \ref{ansatz_boundary}, we have as $\alpha_2 \gg \alpha_1 \sim 1$ 
\[ \rho_1^{(2)} \sim e^{a_0 \alpha_2} + e^{(a_0-a_1)\wt{\alpha}_2} \sim e^{a_0 \alpha_2}, \]  
which also holds for all derivatives. It follows that
\[ g_{22}, \; g_{\wh{\alpha} \wh{\alpha}, \alpha_1 \nmid \wh{\alpha}}  \sim e^{a_0 \alpha_2}, \quad \quad g_{11}, \; g_{\wh{\alpha} \wh{\alpha}, \alpha_1 \mid \wh{\alpha}} \sim e^{(a_0-a_1) \alpha_2}, \]
and likewise for all derivatives. Recall the local expression of the curvature
\[ R_{i \ol{j} k \ol{l}} = - \frac{\del^2 g_{ij}}{\del z_k \del \ol{z}_l} + \sum_{s,t} g^{st} \frac{\del g_{sj}}{\del z_k} \frac{\del g_{ti}}{\del \ol{z}_l}. \]
After translating this into logarithmic coordinates on $\afrak^{*}$ with \( \wt{\alpha}_2 = \log \abs{z_2}^2 \), \( \alpha_1 = \log \abs{z_1}^2 \), a straightforward computation shows that when \( \alpha_2 \gg \alpha_1 \sim 1 \)
\[ R_{1 \ol{1} a \ol{b}}, \; R_{\wh{\alpha} \wh{\alpha} a \ol{b}, \alpha_1 \mid \wh{\alpha}} \sim e^{(a_0-a_1) \alpha_2}, \quad \quad R_{2 \ol{2} a \ol{b}}, \; R_{\wh{\alpha} \wh{\alpha}a \ol{b}, \alpha_1 \mid \wh{\alpha}} \sim e^{a_0 \alpha_2}, \] 
where \( a, b \) runs over \( \set{1,2} \) in the local expression of the metric (cf. Proposition \ref{proposition_metric_local_expression}), while the other components vanish. Given the curvature norm 
\[ \abs{\text{Rm}}_g^2 = g^{i \ol{i}} g^{\ol{j}j} g^{k\ol{k}} g^{\ol{l}l} R_{\ol{i}j\ol{k}l} R_{\ol{i}j\ol{k}l}, \]
it follows that the terms of $\abs{\text{Rm}}_g^2$ grow like $e^{-(a_0-a_1) \alpha_2}$ or $e^{-a_0 \alpha_2}$ on the region $\alpha_2 \gg \alpha_1 \sim 1$, hence bounded if $a_0 - a_1 \geq 0$ and unbounded otherwise. 
\end{proof}

\begin{prop} \label{proposition_bounded_curvature_combinatorial}
The metric \( (G/H,g) \) satisfies Hein's existence package if \( a_0^{\pm} \geq a_1^{\pm} \). 
\end{prop}

\begin{proof}
Connectedness of the annulus follows from the connectedness of \( C_0 \) and the Gromov-Hausdorff convergence of \( (G/H,g) \) to \( (C_0, \ol{g}_0) \) (Proposition \ref{prop_GH_convergence}). The Ricci-estimate is already in Proposition \ref{prop_ricci_estimate}.  To check the volume growth condition, it is enough to show that for \( c > 0 \) large enough and any \( s > c \), the volume of  \( B(x_0,s) \subset (G/H,g) \) satisfies
\[ c^{-1} s^{2n} \leq \text{vol}(B(x_0,s)) \leq c s^{2n}. \] 
Since \( (G/H,g) \xrightarrow{\text{pGH}} (C_0,\ol{g}_0) \), this follows from Colding's volume convergence theorem \cite[Theorem 0.1]{C97} (which is valid under a Ricci lower bound), and the fact that \( C_0 \) is non-collapsed. 

In particular, our volume growth is maximal (of order \(2n \) with \( n > 1 \)), hence to apply Hein's existence package, we need to verify 
\[ \abs{ \nabla \rho } + \rho \abs{i \del \delb \rho } \leq c. \] 
Indeed, \( \nabla \rho \) is bounded since \( \rho \) is asymptotically the distance function \( \rho^{(0)} \) on the cone, and \( \rho \abs{i \del \delb \rho } \) is bounded since \( \Delta \rho^2 \sim \Delta (\rho^{(0)})^2 \sim n \). 
 
To check the existence of a quasi-atlas, first remark from the proof Proposition \ref{prop_GH_convergence} that the metric \( g \) is complete.  Now if \( a_0^{\pm} \geq a_1^{\pm} \), the curvature is bounded by Proposition \ref{proposition_unbounded_curvature}, and we also have a lower bound on the volume of the unit ball, hence
\[ \text{inj}_g \geq \kappa > 0 \]
by Cheeger's lemma (see e.g. \cite[Lemma 11.4.9]{Pet16} where the result is stated for compact manifolds, but the proof, using pointed Gromov-Hausdorff convergence, can be repeated almost verbatim for complete manifolds). 

A lower bound on \( \text{inj}_g \), combined with the \( C^k\)-Ricci bound in Proposition \ref{prop_ricci_estimate}, shows that the harmonic radius is bounded away from zero by Anderson's lemma \cite[Main Lemma 2.2]{A90}, hence the harmonic \( C^{k,\alpha} \)-norm of the manifold is bounded. Together with bounded curvature this guarantees the existence of a good quasi-atlas for \( (G/H, g) \), cf. \cite[Lemma 4.3]{Hein} or \cite[Proposition 1.2]{TY90}. 
\end{proof}

\begin{rmk}
Propositions \ref{proposition_unbounded_curvature}, \ref{proposition_bounded_curvature_combinatorial} show that the curvature of \( (G/H,g) \) is bounded if and only if a condition depending only on the root data of \( G/H \) is satisfied. 
\end{rmk}

\subsection{Existence}
 One can check by direct computation that \( a_0^{\pm} \geq a_1^{\pm} \) for all finite family of symmetric spaces of rank two (see Table \ref{table_constants} for the \( R_1 \times R_1 \) and \( A_2 \) cases, and Lemma \ref{lemma_bounded_geometry_finitefamilies_bc2} with Table \ref{table_bg_bc2} below for the finite families of the \( BC_2 / B_2 \) cases). On the other hand, verifying this condition is non-trivial for the infinite families, but the \textcolor{blue}{methods} are more or less elementary, cf. Lemma \ref{lemma_bounded_geometry_infinitefamilies_bc2}. 
 
 The existence theorem in Hein's thesis then allows us to perturb \( \omega \) (\textcolor{blue}{in the Kähler class consisting of \( K \)-invariant \( \del \delb\)-exact Kähler forms}) to a genuine \textcolor{blue}{\( K \)-invariant} Ricci-flat metric \( \omega_u = \omega + i \del \delb u \), where \( u \) is a \textcolor{blue}{smooth \( K \)-invariant \( \omega\)-psh function} and has decay rate \( O(\rho^{2- \mu}) \), \( \mu > 2 \) being the decay rate of \( \Pcal(\rho) \). The final metric \( \omega_u \) then has the same tangent cone at infinity as \( \omega \) (by a priori estimate). The conclusion then follows.

\subsubsection{Bounded geometry}
We have chosen 
\[ \delta = \alpha_2 - t \alpha_1, \] 
such that \( \sprod{\delta, \alpha_1} < 0 \) and \( \beta = x \alpha_2 + y \alpha_1 \) is a vector in the interior of the Weyl chamber such that \( \sprod{\delta, \beta} = 0 \). Since we want \( \beta \) as the \textit{normalized} Reeb vector of a K-stable horospherical cone, we can suppose that 
\begin{equation} 
\sprod{\varpi, \beta} = n,
\end{equation} 
without loss of generality. 
A straightforward calculation then gives
\begin{equation*} 
\begin{cases} 
&\frac{x}{y} = 
\frac{t \sprod{\alpha_1, \alpha_1} - \sprod{\alpha_1,\alpha_2}}{\sprod{\alpha_2, \alpha_2} -t \sprod{\alpha_1,\alpha_2}}, \\
&  x( A_1 \sprod{\alpha_1, \alpha_2} + A_2 \sprod{\alpha_2, \alpha_2} ) + y (A_1 \sprod{\alpha_1, \alpha_1} + A_2 \sprod{\alpha_1, \alpha_2}) = n\\
\end{cases}
\end{equation*} 
Recall that
\[ \delta_{\pm} = \frac{\lambda_{\pm}}{1+m_{\pm} }, \] 
where \( m_{\pm} \) are the multiplicities of the roots in the polynomial equation 
\[ \int_{\lambda_{-}}^{\lambda{+}} p P(\varpi + p \delta) dp = 0. \] 
By analyzing the asymptotic behavior of the conical potential, we also obtained
\[ a_1^{+} = \delta_{+} = - \frac{\sprod{\varpi, \alpha_1}}{(m_{+} + 1)\sprod{\alpha_1, \delta}}, \quad a_1^{-} = -t \delta_{-} = \frac{t \sprod{\varpi, \alpha_2}}{(m_{-} + 1)\sprod{\alpha_2, \delta}}.   \]
By definition, we have  \( a_0^{+} \wt{\alpha}_2 = b \beta + b_1^{+} \delta \) (and a similar formula for \( a_0^{-} \)) with \( b =  \sprod{\varpi,\beta} / n \sprod{\beta,\beta}  \). It follows that
\[ a_0^{+} \sprod{\wt{\alpha}_2, \beta} = \sprod{\varpi,\beta}/n = 1, \quad a_0^{-} \sprod{\wt{\alpha}_1, \beta} = \sprod{\varpi,\beta}/n = 1.  \] 

In the \( BC_2 \) or \( B_2 \) cases, we have \( a_0^{+} \geq a_1^{+} \) iff 
\[ t \geq \frac{-2m_3 - 2m_3^2 - 2m_2}{(m_1 + m_3 +1)(2m_2 + m_1 + 2m_3)}, \]
\textcolor{blue}{which is always satisfied.} Thus one only needs to verify the bounded geometry condition \( a_0^{-} \geq a_1^{-} \). After a straightforward computation, this is equivalent to 
\[ t( m_2 m_1 - m_1 - 2m_3) \leq A_1(m_2 + 1). \]
When \( m_1 m_2 - m_1 - 2m_3 = 0 \), the condition is trivial. This is the case for the infinite family \( \SO_r / \SO_2 \times \SO_{r-2}\), as well as the class represented by \( \SL_5/ \SL_2 \times \SL_3 \).  On the other hand, when \( m_1 m_2 - m_1 - 2m_3 > 0 \), we need to check that
\[ t \leq \frac{(2m_1 + 2m_2 + 4m_3)(m_2 + 1)}{m_1(m_2 -1) - 2m_3}. \] 

Recall that \( \alpha_2 \) is the long positive root with multiplicity \( m_2 \), and \( \alpha_1 \) the short one with multiplicity \( m_1 \). The positive roots in the \( BC_2 \) and \( B_2 \) cases are \( \alpha_1\), \( \alpha_1 + \alpha_2 \) (with multiplicity \( m_1 \)), \( 2 \alpha_1 \), \( 2 \alpha_1 + 2 \alpha_2 \) (with multiplicity \( m_3 \)), and \( \alpha_2\), \( \alpha_2 + 2 \alpha_1 \) (with multiplicity \( m_2 \)). One has in this case
\begin{equation*}
\begin{aligned}
P(\varpi + p \delta) &= \prod_{\alpha \in R^{+}} \sprod{\alpha,\varpi + p \delta}^{m_{\alpha}} \\
&= (2m_2 + (2+t)p)^{m_2} (2m_1 + 2m_2 + 4m_3 - tp)^{m_2} \\
&(m_1 + 2m_3 - (1+t)p)^{m_1+m_3} (m_1 + 2m_2 + 2m_3 +p)^{m_1 + m_3}. 
\end{aligned} 
\end{equation*} 

From Table \ref{table_constants}, by the result in \cite{Ngh22b} recalled in Theorem \ref{theorem_kstable_horospherical_cone}, the K-stability condition is easily seen to be equivalent to \( t \) being the unique positive root of 

\begin{equation} 
Q(t) = \int_{\frac{-2m_2}{2+t}}^{\frac{m_1 + 2m_3}{1+t}} p P(\varpi + p \delta) dp = 0.
\end{equation}

\begin{lem} \label{lemma_bounded_geometry_finitefamilies_bc2}
Bounded geometry holds for all finite families of restricted root system \( BC_2 \) and \( B_2 \). 
\end{lem}

\begin{proof}
The proof is a straightforward case-by-case computation from the K-stability condition of the horospherical cone and the data in Table \ref{table_constants}. 
\end{proof}

Before verifying bounded geometry for the infinite families, let us make a brief digression on hypergeometric series. The reader might consult \cite{WW} for more information. The \textit{Appell hypergeometric series} \( F_1 \) is defined for \( a,b_1,b_2,c, y, z \in \Cbb \) with  \( \abs{y}, \abs{z} < 1 \) by the double series
\[ F_1(a,b_1,b_2,c;y,z) = \sum_{m,n=0}^{+\infty} \frac{(a)_{m+n} (b_1)_m (b_2)_n}{ (c)_{m+n}} y^m z^n, \]
where \( (\alpha)_m \) is the Pochhammer symbol
\[ (\alpha)_m := \alpha(\alpha+1) \dots (\alpha+m-1), \quad (\alpha)_0 = 1. \] 
When \( z = 0 \), \( F_1 \) is reduced to the more well-known \textit{Gauss hypergeometric series} \( F_{21}(a,b_1,c;z) \)
\[ F_1(a,b_1,b_2,c;y,0) = F_{21}(a,b_1,c;y) := \sum_{m=0}^{\infty} \frac{(a)_m (b_1)_m}{(c)_m} y^m. \]
When \( \Re(c) > \Re(a) > 0 \), \( F_1 \) admits an integral representation as
\[ B(a,c-a) F_1(a,b_1,b_2,c;y,z) = \int_0^1 u^{a-1} (1-u)^{c-a-1}(1-yu)^{-b_1} (1-zu)^{-b_2} du, \]
where \( B(s,t) \) is the beta function \( \int_0^1 u^{s}(1-u)^{t} du \) with values \( \frac{p!q!}{(p+q+1)!} \) at integers \( (p+1,q+1) \). We will only be concerned with values of \( F_1 \) when \( a \geq 1, b_1, b_2 < 0, c > a \) are integers, in which case \( F_1 \) is just a finite sum and a rational expression of \( a,b_{1,2},c,y,z \).  

The language of hypergeometric series might not offer more insight if our goal is just to check the numerical conditions in Lemma \ref{lemma_bounded_geometry_infinitefamilies_bc2}. However, since the calculations get nasty very quickly when \( m_3 = 3 \), such language makes the conditions easier to be implemented on a computer software with inbuilt function to calculate the coefficients of hypergeometric series, e.g. \textit{Mathematica}.

For the two infinite families of restricted root system \( BC_2 \), the condition is translated in terms of \( r \) as in Table \ref{table_bg_bc2}. Note that \( (12r-27)/(2r-10) \geq 6 \) and \( (20r-30)/(6r-27) \geq 10/3 \) for all \( r \geq 5 \). Thus it is enough to show the following. 

\begin{lem} \label{lemma_bounded_geometry_infinitefamilies_bc2}
For the infinite family \( (2r-8,2,1) \), we have \(0 < t < 6 \), and for the family \( (4r-16,4,3) \), we have \(1 < t < 10/3 \). In particular, bounded geometry holds for these families.  
\end{lem}

\begin{proof}
Let \( Q(t,r) := \int_{\lambda_{-}(r)}^{\lambda^{+}(r)} p P(\varpi + p \delta,t,r) dp \). Since for each \( r \), \( t = t(r) \) is the unique root of \( Q \), it is enough to show that \( Q(0,r) > 0 \), \( Q(t_0, r) < 0 \) for \( t_0 = 6, 10/3 \) in each case, and \( Q(1,r) > 0 \) when \( m_3 = 3 \). 

A direct computation shows that \( Q(0,r) > 0 \) iff 
\[ \int_{-m_2}^{m_1 + 2m_3} p(2m_2 + 2p)^{m_2}(m_1 + 2m_3 -p)^{m_1 + m_3} (m_1 + 2m_2 +2m_3 +p)^{m_1 + m_3} dp > 0. \]
After a change of variable \( u := ((2 m_2 + 2 p) / A_1)^2 \), the inequality is equivalent to 
\[ \int_{0}^1 (A_1 u^{\frac{1}{2}} - 2m_2) u^{\frac{m_2-1}{2}}  (1-u) ^{(m_1 + m_3)} du > 0. \] 
This is exactly Condition \cite[(6)]{BD19}, which holds for all infinite families by virtue of Lemma 3.4 in \textit{loc.cit.}.

In the case \( (m_1,m_2,m_3) = (2r-8,2,1) \), we have
\[ Q(t,r) = \int^{\frac{2r-6}{1+t}}_{\frac{-4}{2+t}} p(4 +(2+t)p)^2 (4r-8 -tp)^2 (2r-6 -(1+t)p)^{2r-7} (2r-2+p)^{2r-7} dp, \] 
and we need to verify that \( Q(6,r) < 0 \). Up to a positive constant
\begin{equation*}
Q(6,r) = \int_{- \frac{1}{2}}^{ \frac{2r-6}{7}} p(1+2p)^2(2r-4-3p)^{2} (2r-6-7p)^{2r-7}(2r-2+p)^{2r-7} dp.
\end{equation*}
After a change of variable \( p = \frac{(4r-5)u - 7}{14} \) and clearing positive constants, we can suppose that
\[ Q(6,r) = -\int_{0}^{1} (1-\frac{4r-5}{7}u) u^2 (1- \frac{3}{7} u)^2(1-u)^{2r-7}(1 + \frac{1}{7} u)^{2r-7} du. \]
Consider the function defined on \( \Rbb \) by
\[ H_r (z) := \frac{\int_0^1 u^2(1- \frac{3}{7} u)^2(1-u)^{2r-7}(1-z u)^{2r-7} du}{ \int_0^1 u^3(1- \frac{3}{7} u)^2(1-u)^{2r-7}(1-z u)^{2r-7} du}. \]
Let us show that this is a convex function on \( ]-\infty,1] \) for all \( r \geq 5 \). Indeed, for all \( z_1, z_2 \in ]-\infty, 1] \), \( 1 - z_1 u \) and \( 1 - z_2 u \) are positive for all \( u \in [0,1] \), hence
\[  \tuple{ 1 - \frac{z_1 + z_2}{2} u}^{2r-7} \leq \frac{1}{2} \tuple{ \tuple{\frac{1 - z_1 u}{2}}^{2r-7} + \tuple{\frac{1- z_2 u}{2}}^{2r-7} }. \] 
Moreover we have 
\[ \max \set{ \tuple{\frac{1 - z_1 u}{2}}^{2r-7} , \tuple{\frac{1 - z_2 u}{2}}^{2r-7} } \leq \tuple{ 1 - \frac{z_1 + z_2}{2} u}^{2r-7}.  \]
These two inequalities together show that 
\[ H_r \tuple{\frac{z_1 + z_2}{2}} \leq \frac{H_r (z_1) + H_r (z_2)}{2}, \]
i.e. \( H_r \) is convex on \( ]-\infty,1] \). 

In terms of the Appell series \( F_1 \), the inequality \( Q(6,r) < 0 \) is equivalent to 
\begin{align*} 
H_r \tuple{- \frac{1}{7}} = \frac{B(3,2r-6) F_{1}(3,-2,-(2r-7), 2r-3, \frac{3}{7}, \frac{-1}{7})}{B(4,2r-6) F_{1}(4,-2,-(2r-7), 2r-2, \frac{3}{7}, \frac{-1}{7}) } > \frac{4r-5}{7},
\end{align*} 
where \( B(3,2r-6) / B(4,2r-6) = (2r-3)/3 \). 
By convexity, we have for all \( z < 0 \),
\[ H_r (z) \geq H_r(0) + H'_r (0) z, \]
where by computing the first two terms of the series expansion in the \( z \)-argument of \( F_1 \) at \(  0\),
\[ H_r(0) = \frac{(-1 + 2 r) (327 - 371 r + 98 r^2)}{3 (223 - 315 r + 98 r^2)}, \]
and 
\[ H'_r(0) = \frac{(-7 + 2 r)(-88290 + 296949 r - 400372 r^2 + 260141 r^3 - 80948 r^4 + 9604 r^5)}{3 r (223 - 315 r + 98 r^2)^2}. \] 
Finally,
\begin{align*}
&H_r \tuple{- \frac{1}{7}}  - \frac{4r-5}{7} \\
&\geq \frac{-618030 + 2490711 r - 3779540 r^2 + 2655919 r^3 - 866908 r^4 + 
 105644 r^5}{21 r (223 - 315 r + 98 r^2)^2}.
\end{align*}
Note that the leading coefficient on the polynomial numerator \( P_1(r) \) is positive. We can compute the root of \( P_1 \) and show that they are all strictly smaller than \( 5 \), hence \( P_1(r) \geq P_1(5) > 0 \). Thus \( Q(6,r) < 0 \) for all \( 0 \leq u \leq 1 \).

In the case \( (m_1,m_2,m_3) = (4r-16,4,3) \), we have
\[ Q(t,r) = \int_{ \frac{-8}{2+t} }^{ \frac{4r-10}{1+t}} p(8+(2+t)p)^{4} (8r-12 - tp)^{4} (4r-10 - (1+t)p)^{4r-13} (4r-2+p)^{4r-13} dp, \]
and we need to check that \( Q(10/3,r) < 0 \), i.e. 
\begin{equation*}
\int_{ -\frac{3}{2}}^{ \frac{12r-30}{13}} p(3+2p)^{4} (12r-18-5p)^4(12r-30-13p)^{4r-13}(4r-2+p)^{4r-13} dp < 0. 
\end{equation*}
After the change of variable \( p = \frac{(24r-21)u-39}{26} \), we can suppose that (up to positive constants)
\begin{equation*}
Q(10/3,r) = -\int_{0}^1 (1- \frac{8r-7}{13} u) u^4 (1 -  \frac{5}{13} u)^{4} (1-u)^{4r-13}( 1+ \frac{3}{13} u)^{4r-13} du.
\end{equation*}
By the same reasoning as above, the function
\[H_r(z) :=  \frac{\int_0^1 u^4(1- \frac{5}{13} u)^4 (1-u)^{4r-13}(1-z u)^{4r-13} du}{ \int_0^1 u^5 (1- \frac{5}{13} u)^4 (1-u)^{4r-13}(1-z u)^{4r-13} du} \]
is convex on \( ]-\infty,1] \) for all \( r \geq 5 \). 
We then proceed to check that 
\begin{align*} 
H_r \tuple{ - \frac{3}{13}} = \frac{B(5,4r-12) F_{1}(5,-4,-(4r-13), 4r-7, \frac{5}{13}, \frac{-3}{13})}{B(6,4r-12) F_{1}(6,-4,-(4r-13), 4r-6, \frac{5}{13}, \frac{-3}{13}) } > \frac{8r-7}{13},
\end{align*} 
with \( B(5,4r-12)/ B(6,4r-12) = (4r-7)/5 \). 
Again, the series expansion in the \( z \)-argument of the Appell functions at \( 0 \) and convexity of \(H_r \) show that for all \( z \leq 0 \), 
\[ H_r(z) \geq H_r(0) + H_r'(0)z, \] 
where 
\begin{align*}
H_r(0) = \frac{ (-3 + 4 r) (9009405 - 21344609 r + 18336838 r^2 - 6784336 r^3 + 913952 r^4) }{5 G(r) },
\end{align*} 
and 
\begin{align*}
H_r'(0) = &\frac{2(-13 + 4 r)}{5 (-1+2r) G(r)^2 }(-59648595700050 + 361147658887485 r - 953718016145076 r^2 \\
& + 1440364128450741 r^3 - 1370700222571316 r^4 + 852476267292676 r^5 \\
&- 346597132254464 r^6 + 88867036945024 r^7 - 13043659725824 r^8 \\
&+ 835308258304 r^9 ),    
\end{align*}
where \( G(r) = (5914845 - 15897921 r + 15254278 r^2 - 6221904 r^3 + 913952 r^4) \). Finally, evaluating at \( z = -3/13 \) yields
\begin{align*}
&H_r \tuple{-\frac{3}{13}} - \frac{8r-7}{13} \\
&\geq \frac{1}{G_1(r)} (-2574310334032125 + 12163729606848990 r - 19659125767233897 r^2 \\
&+ 3313560826702152 r^3 + 32935584407475308 r^4 - 54028266881305808 r^5 \\
&+ 43597604159841472 r^6 - 20860496268764672 r^7 \\
&+ 6010547069043712 r^8 - 966644418301952 r^9 + 66824660664320 r^{10}),
\end{align*}
with \( G_1(r) = 13 (7 - 18 r + 8 r^2) G(r)^2 \). Note that the leading coefficient on the numerator is positive. We can then apply the same argument as in the previous case to show that \( Q(10/3,r) < 0 \). 

It remains to show that \( Q(1,r) > 0 \) when \( m_3 = 3 \). We have up to a positive constant
\[ Q(1,r) = \int_{-8/3}^{2r-5} p(8+3p)^{4} (8r-12 - p)^{4} (2r-5 - p)^{4r-13} (4r-2+p)^{4r-13} dp, \]
and after the change of variable \( p = \frac{(6r-7)u-8}{3} \), we obtain up to a positive constant 
\[ Q(1,r) = -\int_{0}^1 (1 - \frac{6r-7}{8}u) u^4 (1 - \frac{u}{4})^4 (1-u)^{4r-13} (1 + \frac{u}{2})^{4r-13} du. \]
Consider now the function defined on \( \Rbb \) by
\[H_r(z) :=  \frac{\int_0^1 u^5(1- \frac{1}{4} u)^4 (1-u)^{4r-13}(1-z u)^{4r-13} du}{ \int_0^1 u^4 (1- \frac{1}{4} u)^4 (1-u)^{4r-13}(1-z u)^{4r-13} du}. \] 
Again this is a convex function on \( ]-\infty,1] \) for all \( r \geq 5 \). The inequality \( Q(1,r) > 0 \) is equivalent to 
\begin{align*} 
H_r \tuple{ -\frac{1}{2}} = \frac{B(6,4r-12) F_{1}(6,-4,-(4r-13), 4r-6, \frac{1}{4}, \frac{-1}{2})}{B(5,4r-12) F_{1}(5,-4,-(4r-13), 4r-7, \frac{1}{4}, \frac{-1}{2}) } > \frac{8}{6r-7},
\end{align*} 
where \( B(6,4r-12) / B(5,4r-12) = 5/(4r-7) \). By the same argument as above, 
\[ H_r (z) \geq H_r(0) + H'_r (0) z, \]
where 
\[ H_r(0) = \frac{5 (15741 - 48336 r + 52928 r^2 - 24576 r^3 + 
   4096 r^4)}{G(r)}, \]
and 
\begin{align*} 
H'_r(0) = \frac{10}{(-1+2r)G(r)^2} &(-13 + 4 r) (-367033275 + 2594714553 r -7943750640 r^2 \\
&+ 13824076800 r^3 - 15075823616 r^4 + 10689089536 r^5 \\
&- 4929290240 r^6 + 1426063360 r^7 - 234881024 r^8 \\
&+ 16777216 r^9), 
\end{align*}    
where \( G(r) = (-3 + 4 r) (27585 - 71832 r + 67904 r^2 - 27648 r^3 + 
   4096 r^4) \). 
Finally, evaluating at \( z = -1/2 \) yields
\begin{align*}
&H_r \tuple{ -\frac{1}{2} } - \frac{8}{6r-7} \\
&\geq \frac{1}{G_1(r)} (-157805645850 + 1283896287045 r - 4637385570222 r^2 \\
&+ 9817511490024 r^3 - 13540927644160 r^4 + 12783832862720 r^5 \\
&- 8436522991616 r^6 + 3896079450112 r^7 - 1235710115840 r^8 \\
&+ 256859176960 r^9 - 31574720512 r^{10} + 1744830464 r^{11}),
\end{align*}
where \( G_1(r) = (3 - 4 r)^2 (-1 + 2 r) (-7 + 6 r) G(r)^2 \). The polynomial numerator has positive leading coefficient with all of the roots being strictly smaller than \( 5 \), hence strictly positive for all \( r \geq 5 \). Thus \( H_r(-1/2) > 8/(6r-7) \), i.e. \( Q(1,r) >  0 \). 
\end{proof}

\begin{rmk}
One could try to bound the function \( H_r \) from below by using integration by parts and elementary inequalities, e.g. the trivial ones and/or the AM-GM inequality, but such methods would only yield information up to the first derivative of \( H_r \), hence a slightly weaker lower bound. Here we need a precise information on the second derivative of \( H_r \) to obtain a sharp lower bound. 
\end{rmk}

\begin{table}[]
    \adjustbox{width= 0.8 \textwidth}{\centering
    \begin{tabular}{|c|c|c| }
    \hline 
    \( \) & \( \text{Multiplicities} \) &\( a_0^{-} \geq a_1^{-} \)  \\
        \hline
         \( \SL_r/ \SL_2 \times \SL_{r-2} \) & \( (2r-8,2,1), r \geq 5\) & \( \text{trivial if \( r = 5 \)},  t \leq \frac{12r-27}{2r-10} \text{if \( r \geq 6\)} \) \\
         \( \Sp_{2r} / \Sp_4 \times \Sp_{2r-4} \) & \( (4r-16,4,3), r \geq 5 \) & \( t \leq \frac{20r-30}{6r-27} \) \\
         \( \SO_{10} / \GL_5 \) & \( (4,4,1) \)& \( t \leq 10 \) \\
         \( E_6 / \SO_{10} \times \SO(2) \) & \((8,6,1) \)& \( t \leq \frac{112}{19} \) \\
         \hline
         \( \SO_r / \SO_2 \times \SO_{r-2} \) & \( (r-4,1,0), r \geq 5 \) &\( \text{trivial} \) \\
         \( \SO_5 \times \SO_5 / \SO_5 \) & \( (2,2,0) \) &\( t \leq 12 \)  \\
         \( \Sp_8 / \Sp_4 \times \Sp_4 \) & \( (4,3,0) \) &\( t \leq 7 \)  \\
         \hline
    \end{tabular}}
    \caption{Bounded geometry condition for \( BC_2 \) and \( B_2 \).}
    \label{table_bg_bc2}
\end{table}

\begin{table}
\adjustbox{width=1\textwidth}{
\begin{tabular}{ |c|c|c|c| }
\hline
  & \( R_1 \times R_1 \) & \( A_2 \) & \( BC_2\) \\
\hline
\( t \) & \( \frac{(m_1 + 2 \wh{m}_1)(m_2 + \wh{m}_2 + 1)}{(m_2 + 2\wh{m}_2)(m_1 + \wh{m}_1 +1)}\) & \( 1 \)  & \( \text{irrational in general} \) \\ 
\hline
\( \sprod{\alpha_1, \alpha_2} \) & \( 0 \) & \(-1/2 \) & \( -1 \) \\ 
\( \sprod{\alpha_1, \alpha_1} \) & \( 1 \)  & \( 1 \) & \( 1 \) \\ 
\( \sprod{\alpha_2, \alpha_2} \) & \( 1 \) & \( 1 \) & \( 2 \) \\ 
\( \wt{\alpha}_1 \) & \( \alpha_1  \) & \( \alpha_1 + \frac{1}{2} \alpha_2 \) & \( \alpha_1 + \frac{1}{2}\alpha_2 \) \\
\( \wt{\alpha}_2 \) & \( \alpha_2  \) & \( \alpha_2 + \frac{1}{2} \alpha_1 \) & \( \alpha_2 +  \alpha_1 \) \\ 
\( \varpi \) & \( (m_1 + 2 \wh{m}_1) \alpha_1 + (m_2 + 2 \wh{m}_2) \alpha_2 \) & \( 2m(\alpha_1 + \alpha_2) \) & \( (2m_1 + 2 m_2 + 4 m_3) \alpha_1 + (m_1 + 2 m_2 + 2 m_3) \alpha_2 \) \\
\( A_1 \) & \( m_1 + 2 \wh{m}_1 \) & \( 2m \) & \( 2m_1 + 2 m_2 + 4 m_3\) \\
\( A_2 \) & \( m_2 + 2 \wh{m}_2 \) & \( 2m\) & \( m_1 + 2 m_2 + 2m_3 \) \\
\( m_{+} \) & \( m_1 + \wh{m}_1 \) & \( m  \) & \( m_1 + m_3 \) \\
\( m_{-} \) & \( m_2 + \wh{m}_2 \) & \( m \) & \( m_2 \) \\
\( n \) & \( 2 + m_{+} + m_{-} \)  &  \( 2 + 3m \)  & \( 2(1 + m_1 + m_2 + m_3) \) \\
\( \beta \)  & \( \frac{m_2 + \wh{m}_2 + 1}{m_2 + 2 \wh{m}_2} \alpha_2 + \frac{m_1 + \wh{m}_1 + 1}{m_1 + 2 \wh{m}_1} \alpha_1  \) & \( \frac{n}{2m}(\alpha_2 + \alpha_1) \)& \( \frac{n ( (2+t) \alpha_1 + (1+t) \alpha_2) }{2m_2 (1+t) +(m_1 + 2m_3)(2+t)}.  \) \\ 
\hline
\( \lambda_{+} \) & \( \frac{(m_2 + 2 \wh{m}_2)(m_1 + \wh{m}_1 +1)}{m_2 + \wh{m}_2 +1} \) & \( 2m/3 \) & \( \frac{m_1 + 2m_3}{1+t}\)\\
\( \lambda_{-} \) & \( -(m_2 + 2 \wh{m}_2) \) & \( -2m/3 \) & \( - \frac{2m_2}{2+ t} \) \\
\( a_0^{+} \) & \( \frac{m_2 + 2 \wh{m}_2}{m_2 + \wh{m}_2 + 1} \) & \( 8m/3n \) & \( \frac{2m_2(1+t) + (m_1 + 2m_3)(2+t)}{n(1+t)} \) \\
\( a_1^{+} = \delta_{+} \) & \( \frac{m_2 + 2 \wh{m}_2}{m_2 + \wh{m}_2 +1 } \) & \( \frac{2m}{3(m+1)}\) & \( \frac{m_1 + 2m_3}{(m_1 + m_3 +1)(1+t)} \) \\ 
\( a_0^{-} \) & \( \frac{m_1 + 2 \wh{m}_1}{m_1 + \wh{m}_1 + 1}  \) & \( 8m /3n \) & \( \frac{ 4m_2(1+t) + ( 2 m_1 + 4m_3)(2+t)}{n(t+2)} \) \\
\( a_1^{-} = - t \delta_{-} \) & \(\frac{m_1 + 2 \wh{m}_1}{m_1 + \wh{m}_1 + 1} \)& \( \frac{2m}{3(m+1)}\) & \(  \frac{2 t m_2}{(m_2 + 1)(2+t)} \)\\
\( a_0^{\pm} \geq a_1^{\pm} \) & \( \text{true} \) & \text{true} & \(  (m_1(m_2 -1) - 2m_3)t \leq (2m_1 + 2m_2 + 4m_3)(m_2 + 1), \text{cf. Table} \; \ref{table_bg_bc2} \)  \\
\hline
\end{tabular}}
\caption{Relevant constants.}
\label{table_constants}
\end{table}

\newpage
\appendix \label{appendix}
\section{Spherical varieties} \label{appendix_spherical_varieties}
In what follows, \( G \)  is a complex simply connected linear reductive group, while \( B \), \( T \), \( U \) denote respectively the Borel subgroup, maximal torus of \( G \) and maximal unipotent subgroup of \( B \). Under an embedding of \( G \) into \( GL_N(\Cbb) \), \( B, T, U\) can be identified with subgroups of upper-triangular, diagonal, and upper-triangular matrices with ones on the diagonal, respectively.  Likewise, the opposite Borel subgroup \( B^{-} \) (resp. opposite maximal unipotent subgroup \( U^{-} \)) can be identified with the subgroup of lower-triangular matrices (resp. lower-triangular matrices with ones on the diagonal). 


\subsection{Luna-Vust theory of spherical embeddings}
Spherical varieties cover a large spectrum of well-known varieties (including toric varieties, grassmannians and complex symmetric spaces). An experienced reader can make analogy with the theory of toric varieties where \( G = B = T \) to guide their intuition.  

\subsubsection{Weight lattice, colors and valuation cone} 
A \textit{spherical space} is a homogeneous space \( G/H \) containing a Zariski-open orbit of a Borel subgroup \( B \subset G \). By a theorem of Chevalley, we may view \( G/H \) as a smooth quasiprojective variety. Let \( \Cbb(G/H)^{(B)} \) be the set of rational functions on \( G/H \) which are eigenvectors of \( B \). An element of \( \Cbb(G/H)^{(B)} \) is called a \( B\)-semi-invariant function. The set of \( B \)-invariant rational functions \( \Cbb(G/H)^B \) is actually \( \Cbb \) by Rosenlicht's theorem. 

\begin{defn} \label{definition_weight_lattice}
To a spherical space are associated two lattices. 
\begin{itemize}
    \item The weights \( \Mcal \) of \( \Cbb(G/H)^{(B)}\), called the \emph{weight lattice}.  
    \item The dual lattice \( \Ncal := \Hom(\Mcal,\Zbb) \), called the \emph{coweight lattice}.  
\end{itemize}
 We will use \( \Mcal_{\Rbb} \) and \( \Ncal_{\Rbb} \) to denote the real vector spaces \( \Mcal \otimes \Rbb\) and \( \Ncal \otimes \Rbb \). The dimension of \( \Mcal_{\Rbb} \) is called the \emph{rank of \( G/H \)}. The vector space \( \Ncal_{\Rbb} \) is also called the \emph{Cartan space} of \( G/H \). 
\end{defn}

Remark that there exists a natural bijection between \( \Mcal  \) and \( \Cbb(G/H)^{(B)}/\Cbb^{*} \), which sends a weight to its eigenvector. Indeed, let \( f,g \) be two functions of the same weight, then \( f_1/ f_2 \) is a \( B \)-invariant rational function, hence constant. 

The open \( B \)-orbit in \( G/H \), which is isomorphic (as an affine variety) to \( (\Cbb^{*})^k \times \Cbb^m \) \cite[Theorem 5]{Ros63}, is an open affine subset of \( G/H \), hence its complement in \( G/H \) is a collection of \( B \)-stable irreducible divisors. They are the only \( B \)-stable irreducible divisors of \( G/H \). 

\begin{defn}
The set $\Dcal$ of irreductible $B$-stable divisors of $G/H$ is called the \emph{set of colors of $G/H$}.
\end{defn}
Recall that a \textit{rational valuation} of a normal variety \( X \) is a map \( \nu : \Cbb(X) \to \Qbb \) satisfying: 
\begin{itemize}
    \item \( \nu (f_1 + f_2 ) \geq \nu(f_1) + \nu(f_2) \), for all \( f_1, f_2 \in \Cbb(X) \backslash \set{0} \) such that \( f_1 + f_2 \neq 0 \). 
    \item \( \nu(f_1 f_2 ) = \nu(f_1) + \nu(f_2) \) for all \( f_1, f_2 \in \Cbb(X) \backslash \set{0} \). 
    \item \( \nu(f) = 0 \) if and only if \( f \) is constant. 
\end{itemize}
If \( X \) is a \( G\)-variety, the valuation is said to be \textit{\(G\)-invariant} if \( \nu(g.f) = \nu(f) \) for all \( g \in G \) and \( f \in \Cbb(X) \), where \( (g.f)(.) := f(g^{-1}.) \). 

\begin{prop}[\!\!\cite{Kno91}]
Denote by \( \Vcal \) the set of \( G\)-invariant rational valuations of \( G/H\). 
\begin{itemize}
    \item The natural mapping
\begin{align*}
\sigma : \Vcal &\to \Hom(\Mcal, \Qbb) = \Ncal_{\Qbb} \\
\nu & \to (\chi \to \nu(f_{\chi}) ) 
\end{align*}
is a well-defined injection. In particular, we can identify \( \Vcal \) with a cone in \( \Ncal_{\Rbb} \), called the \textit{valuation cone} of \( G/H \).  
    \item The valuation cone \( \Vcal \) is a strictly convex and polyhedral cone in \( \Ncal_{\Rbb} \).
\end{itemize}
\end{prop}

\begin{defn}
Each element in \( \Dcal \) induces a valuation of \( \Cbb(G/H) \), hence a natural (in general non-injective) map \( \Dcal \to \Ncal_{\Qbb} \) as above, still denoted by \( \sigma \). We call \( \sigma(\Dcal) \) the \emph{images of the colors of} \( G/H \). 
\end{defn}

\begin{ex} \label{rank_one_horospherical}
Let \( G = \SL_2 \) act on \( \Cbb^2 \) on the left. The stabilizer \( H\) of \( (1,0) \) is then \(U \). Let \( B \subset \SL_2 \) be the Borel subgroup and \( \alpha \) the unique positive root corresponding to \( B \) with coroot \( \alpha^{\vee} \). Let \( B^{-} \) be the opposite Borel subgroup of \( B \).  The open \( G \)-orbit is isomorphic to \( \SL_2 / U \simeq \Cbb^2 \backslash \set{0} \), with open dense \( B^{-} \)-orbit isomorphic to \( B^{-} / U \). The fundamental weight of \( B^{-} \) is
\[ \omega(b_{ij}) = b_{11}. \] 
The \( H \)-right-invariant eigenvector of \( B^{-} \) with weight \( \omega \) is then \( f(A \in \SL_2) = a_{22} \). It follows that \( \Mcal = \Zbb \omega \). The coweight lattice \( \Ncal \) is \( \Zbb (\alpha^{\vee}|_{\Mcal}) \), and \( \Ncal_{\Rbb} \) coincides with the valuation cone \( \Vcal \). The unique color of \( G/H \) is \( D = \set{f = 0} \). Since \( \sigma(D)(f) = \nu_D(f) = 1 \), we have that \( \sigma(D) = \alpha^{\vee}|_{\Mcal} \). 
\end{ex}

\begin{ex}
Let \(  P\) be a parabolic subgroup of \( G \), which can always be given as a matrix triangular in given blocks. Then \( G/P \) is a spherical space of rank zero, known as a flag manifold or homogeneous projective space. 
\end{ex}

\subsubsection{Colored fans and spherical embeddings}
Recall that a spherical embedding is a \( G \)-equivariant embedding of a spherical space \( G/H \). It is clear that a spherical embedding is a spherical variety. Conversely, let \( X \) be a spherical variety with an open orbit \( G.x \). Then we can always choose a stabilizer \( H \) of \( x \) such that \( BH/H \) is open in \( G/H \), hence \( (X,x) \) is a \( G/H \)-spherical embedding. When there is no confusion, we remove \( x \) and simply say that \( G/H \subset X \) is a spherical embedding.

\begin{defn} \label{g_stable_divisors_colors_definition}
Let \(G/H \subset (X,x) \) be a spherical embedding.
\begin{itemize}
\item The divisors of $\Dcal$ whose closure in $X$ contains a closed orbit are called the \emph{colors of $X$}. The set of colors of \( X \) is denoted by \( \Dcal_X \).  
\item We set \( \Vcal_X \) to be the \emph{\( G \)-stable divisors of \( X \)}, which are the irreducible components of \( X \backslash Gx \). If \( D \) is a \( G \)-invariant divisor, then by normality of \( X \), we can associate to it a unique primitive \( G \)-invariant valuation \( \nu_D \), hence under the map \( \sigma \), \emph{\( \Vcal_X \) injects to a finite subset of \( \Vcal \)}. 
\end{itemize}
\end{defn}

If we identify the elements in \( \Dcal \) with their closure in \( X \), then the elements of $\Dcal \backslash \Dcal_X$ correspond to $B$-stable divisors of $X$ not containing any \(G\)-orbit. 

\begin{rmk}
We distinguish between \emph{the colors of $G/H$} and \emph{the colors of $G/H$ as a trivial embedding of itself}. The latter is always empty.  
\end{rmk} 

\begin{defn}
A spherical variety is said to be \emph{simple} if it contains a unique closed \( G \)-orbit. 
\end{defn}


\begin{ex}
Consider Example \ref{rank_one_horospherical}. The natural embedding $X = \Cbb^2 \supset G/H$ is a simple embedding with the unique closed \( G \)-orbit \(Y = \set{0} \).
\end{ex}

\begin{defn} \label{definition_colored_cones} Let \( G/H \) be a spherical space. 
A \emph{colored cone} is a couple $(\Ccal, \Fcal)$ where $\Ccal \subset \Ncal_{\Rbb} $ and  $\Fcal \subset \Dcal$  that verifies: 
\begin{itemize}
    \item $0 \notin \sigma(\Fcal)$ and $\Ccal$ is a strictly convex polyhedral cone generated by $\sigma(\Fcal)$ (which might be empty) and a finite number of elements of $\Vcal$. 
    \item  The relative interior of $\Ccal$ meets $\Vcal$. 
\end{itemize}
A \textit{face} of the colored cone \( (\Ccal, \Fcal) \) is a couple \( ( \Ccal', \Fcal') \) such that \( \Ccal' \) is a face of the cone \( \Ccal \), the relative interior of \( \Ccal' \)  meets \( \Vcal \) , and  \( \Fcal' = \Fcal \cap \sigma^{-1} ( \Ccal' ) \). 
\end{defn}


\begin{thm}[\!\!\cite{Kno91}]
Let \( G/H \) be a spherical space and \( G/H \subset X \) a spherical embedding with \( \Vcal_X, \Dcal_X \) as in Definition \ref{g_stable_divisors_colors_definition}.  Let \( \Ccal_X \) be the cone generated by \( \Vcal_X, \sigma(\Dcal_X) \). The map \( X \to (\Ccal_X, \Dcal_X) \) is a one-to-one correspondence between isomorphism classes of \( G \)-equivariant simple embeddings of \( G/H \) and colored cones. 
\end{thm}

\begin{ex} \label{rank_one_horospherical_simple_embeddings}
Consider the action of \( \SL_2 \) in Example \ref{rank_one_horospherical}. Apart from the trivial embedding, all possible colored cones are
\begin{itemize}
    \item $(\Rbb_{\geq 0} \sigma(D), \varnothing)$, which is $X = \Bl_0(\Cbb^2)$.
    \item $(\Rbb_{\leq 0} \sigma(D), \varnothing)$. The corresponding embedding is $X \simeq \Pbb^2 \backslash [1:0:0]$.
    \item $(\Rbb_{\geq 0} \sigma(D), \set{D})$, with $X  = \Cbb^2$ as the embedding. 
\end{itemize}
Remark that \( \Cbb^2 \) has no \( \SL_2 \)-stable divisor and \( \Bl_0(\Cbb^2) \) has no color. 
\end{ex}

\begin{defn} \label{definition_colored_fans}
A \emph{colored fan} is a finite collection $\Fbb$ of colored cones such that: 
\begin{itemize}
\item Every face of a colored cone in \( \Fbb \) belongs to $\Fbb$. 
\item For all $v \in \Vcal$, there exists at most one $(\Ccal, \Fcal) \in \Fbb$ such that $v \in \text{RelInt}(\Ccal)$.
\end{itemize}
The \emph{support} of  \( \Fbb \) is defined as:
\[ \text{Supp}(\Fbb) = \Vcal \cap \cup_{(\Ccal, \Fcal) \in \Fbb} \text{Supp} (\Ccal). \] 
\end{defn}

\begin{thm}[\!\!\cite{Kno91}]  \label{spherical_embeddings_classification}
Given a spherical space $G/H$, there exists a one-to-one correspondence between the isomorphism classes of \( G \)-equivariant embeddings $X \supset G/H$ and the set of colored fans. Moreover, $X$ is compact if and only if the support of the fan $\Fbb(X)$ coincides with the valuation cone \( \Vcal \). 
\end{thm}

\begin{ex} \label{rank_one_horospherical_non_simple_embeddings}
Let us come back to Example \ref{rank_one_horospherical}. The colored fans of the non-simple embeddings of \( \SL_2 / U \) are 
\begin{itemize}
    \item $\set{ (0, \varnothing), (\Rbb_{\geq 0}\sigma(D), \set{D}), (\Rbb_{\leq 0} \sigma(D), \varnothing) }$.
    \item  $\set{ (0, \varnothing), (\Rbb_{\geq 0}\sigma(D), \varnothing), (\Rbb_{\leq 0} \sigma(D), \varnothing) }$.
\end{itemize}
The associated embeddings are $\Pbb^2$ and $\text{Bl}_0 \Pbb^2$.
\end{ex}

Let \( X \) be a spherical variety. Any  $G$-orbit 
$Y$  is contained in a unique simple spherical variety  $X_{Y}$, which is a union of the orbits whose closure contains $Y$. By the classification of simple embeddings, \( X_Y \) corresponds to a colored cone \( (\Ccal_{X_Y}, \Dcal_{X_Y}) \). 

\begin{prop}[Orbit-cone correspondence \cite{Kno91}] \label{proposition_orbit_cone_correspondence}
Let \( Y \) be any \( G \)-orbit in a spherical variety \( X \). There exists a bijection between the $G$-orbits in $X$ whose closure contains the $G$-orbit $Y$  and the faces of the colored cone \( (\Ccal_{X_{Y}}, \Dcal_{X_{Y}}) \), given by $Z \to (\Ccal_{X_{Z}}, \Dcal_{X_{Z}})$. 
\end{prop}

Like in the toric case, we have a description of equivariant morphisms between spherical embeddings of \( G\)-homogeneous spaces. Let \( G/H \), \( G/H'\) be the spherical spaces such that \( G \supset H' \supset H \). Let \( \phi^{*}: \Mcal(G/H) \to \Mcal(G/H') \) resp. \( \phi_{*}: \Ncal(G/H') \to \Ncal(G/H) \) be the natural injective resp. surjective linear morphism induced by the natural map \( \phi: G/H' \to G/H \). Let \( \Dcal'_{\phi} \) be the set of \( D \in \Dcal' \) such that \( \sigma(D) = G/H \). In particular \( \sigma(D) \in \Dcal(G/H) \) for all \( D \in \Dcal' \backslash \Dcal'_{\phi} \). 

Let \( (\Ccal, \Fcal)\), \( (\Ccal', \Fcal') \) be colored cones of \( G/H, G/H' \) respectively. Then we say that \( ( \Ccal', \Fcal') \) \textit{dominates} \( (\Ccal, \Fcal) \) if \( \phi_{*} (\Ccal') \) is contained in \( \Ccal \) and \( \phi_{*}(\Dcal' \backslash \Dcal'_{\phi}) \subset \Fcal \). 

\begin{prop}[\!\!{\cite[Theorem 4.1]{Kno91}}] \label{proposition_spherical_morphisms}
Let \( X,X' \) be any spherical embeddings of \( G/H, G/H' \) of colored fans \( \Fbb(X') \) and \( \Fbb(X) \). Then \( \phi \) extends to a \( G\)-equivariant morphism \( X' \to X \) if and only if \( \Fbb(X') \) dominates \( \Fbb(X) \). 
\end{prop}

\bibliographystyle{alpha}
\bibliography{biblio} 
\end{document}